\documentclass[12pt]{amsart}
\usepackage{amsthm, amsmath, amssymb, amsrefs}
\usepackage{graphicx}
\usepackage{fullpage}

\usepackage{hyperref}

\newcommand{\T}{\mathbb{T}} 
\newcommand{\C}{\mathbb{C}} 
\newcommand{\D}{\mathbb{D}} 

\newcommand{\R}{\mathbb{R}}
\newcommand{\Z}{\mathbb{Z}} 
\newcommand{\cD}{\overline{\D}} 

\newcommand{\ip}[2]{\langle #1, #2 \rangle}

\newcommand{\refl}[1]{\tilde{#1}}
\newcommand{\mcE}{\mathcal{E}}
\newcommand{\mcF}{\mathcal{F}}
\newcommand{\mcP}{\mathcal{P}}
\newcommand{\mcA}{\mathcal{A}}

\newcommand{\mcG}{\mathcal{G}}
\newcommand{\mcH}{\mathcal{H}}
\newcommand{\mcK}{\mathcal{K}}
\renewcommand{\Re}{\text{Re}}
\renewcommand{\Im}{\text{Im}}

\newcommand{\Hrow}{H_{1\times n}^2}

\newcommand{\defn}{\overset{\text{def}}{=}}

\newcommand{\Lp}{L^2(\frac{d\sigma}{|p|^2})}
\newcommand{\Lq}{L^2(\frac{d\sigma}{|q|^2})}
\newcommand{\udots}{\text{\reflectbox{$\ddots$}}}

\newcommand\xqed[1]{%
  \leavevmode\unskip\penalty9999 \hbox{}\nobreak\hfill
  \quad\hbox{#1}}
\newcommand\eox{\xqed{$\blacklozenge$}}

\numberwithin{equation}{section}

\newtheorem{theorem}{Theorem}[section]
\newtheorem{prop}[theorem]{Proposition}
\newtheorem{corollary}[theorem]{Corollary}
\newtheorem{lemma}[theorem]{Lemma}

\newtheorem{introthm}{Theorem}

\newtheorem{question}{Question}

\theoremstyle{definition}
\newtheorem{definition}[theorem]{Definition}
\newtheorem{notation}[theorem]{Notation}
\newtheorem{example}[theorem]{Example}
\newtheorem*{question*}{Question}
\newtheorem{remark}[theorem]{Remark}

\title{Integrability and regularity of rational functions}

\author{Greg Knese}

\address{Washington University in St. Louis\\ Department of
  Mathematics\\ St. Louis, Missouri 63130}

\email{geknese@math.wustl.edu}

\date{\today}

\thanks{Partially supported by NSF grant DMS-1363239}



\begin{document}

\begin{abstract}
  Motivated by recent work in the mathematics and engineering
  literature, we study integrability and non-tangential regularity on
  the two-torus for rational functions that are holomorphic on the
  bidisk.  One way to study such rational functions is to fix the
  denominator and look at the ideal of polynomials in the numerator
  such that the rational function is square integrable.  A concrete
  list of generators is given for this ideal as well as a precise
  count of the dimension of the subspace of numerators with a
  specified bound on bidegree.  The dimension count is accomplished by
  constructing a natural pair of commuting contractions on a finite
  dimensional Hilbert space and studying their joint generalized
  eigenspaces.

  Non-tangential regularity of rational functions on the polydisk is
  also studied.  One result states that rational inner functions on
  the polydisk have non-tangential limits at \emph{every} point of the
  $n$-torus.  An algebraic characterization of higher non-tangential
  regularity is given.  We also make some connections with the earlier
  material and prove that rational functions on the bidisk which are
  square integrable on the two-torus are non-tangentially bounded at
  every point.  Several examples are provided.

\end{abstract}

\maketitle

\newpage

\tableofcontents 

\section{Introduction}
This paper is about integrability and boundary regularity properties
of rational functions in several variables.  Although this sounds like
well-traveled territory, the questions we are interested in seem to
have no general theory for systematically addressing them.  The paper
focuses on rational functions that are holomorphic on the bidisk 
\[
\D^2 \defn \{(z_1,z_2)\in \C^2: |z_1|,|z_2|<1\}
\]
and their behavior on or near
the distinguished boundary 
\[
\T^2 \defn \{z \in \C^2: |z_1|=|z_2|=1\}.
\]
Our questions are:

\begin{question} \label{q2} For fixed $p\in \C[z_1,z_2]$ with no zeros
  on $\D^2$, is there an algebraic characterization of the
  ideal
\[
\mathcal{I}_p \defn \{q \in \C[z_1,z_2]: q/p \in L^2(\T^2)\}?
\]
Namely, can a finite list of generators be explicitly described?
\end{question}

\begin{question} \label{q3} For fixed $p\in \C[z_1,z_2]$ with no zeros
  in $\D^2$, what is the dimension of
\[
\mcP_{j,k} \defn \{q\in \C[z_1,z_2]: q/p \in L^2(\T^2), \deg q \leq (j,k)\}?
\]
Here $\deg q$ refers to the bidegree of $q$.
\end{question}

\begin{question} \label{q1} When does a rational function $q/p$ on
  $\D^2$ possess a limit as $z\in \D^2 \to \zeta \in \T^2$
  non-tangentially? When does $q/p$ possess higher non-tangential
  regularity?
\end{question}

Readers can certainly imagine many other natural variations on these
questions---change the domain, change the regularity or integrability
conditions---but already these questions are rich.  After applying a
Cayley transform, many of these questions can be converted to
questions about rational functions on a product of upper half-planes
where the boundary of interest is now simply $\R^2$.  For local issues
this makes little difference, but for more global questions having a
compact boundary is important and in particular makes Question
\ref{q2} sensible ($\mathcal{I}_p$ is not an ideal if we replace $\T^2$ with $\R^2$).

Although these questions are certainly fundamental in nature, why are
they worthy of in-depth study?  There are several reasons.

In the engineering literature, there is interest in understanding
``non-essential singularities of the second kind'' of rational
functions.  These are singularities on $\T^2$ where both the numerator
and denominator vanish (assuming they have no factors in common).  An
early influential paper on this was Goodman \cite{goodman} which
studied when the Fourier coefficients of $q/p$ are in $\ell^1$
(bounded-input-bounded-output stability) or $\ell^2$ (square-summable
impulse response) or $\ell^\infty$ (bounded impulse response).  In
particular, a detailed study was given of the examples
\[
G_1(z) = \frac{(1-z_1)^8(1-z_2)^8}{2-z_1-z_2} \quad 
G_2(z) = \frac{(1-z_1)(1-z_2)}{2-z_1-z_2}
\]
\[
G_3(z) = \frac{2}{2-z_1-z_2}.
\]
As shown by computations in \cite{goodman}, $G_1,G_2$ are bounded in
$\D^2$; the Fourier coefficients of $G_1$ are in $\ell^1$; the Fourier
coefficients of $G_2$ are in $\ell^2 \setminus \ell^1$; the Fourier
coefficients of $G_3$ are in $c_0 \setminus \ell^2$.  Using the
techniques presented here it is possible to prove these facts more
systematically.  The recent paper \cite{KM} studies certain 2D linear
systems where singularities on $\T^2$ are forced by the structure at
hand; an example is given to vehicle platooning.  See \cite{KM} for
further references in the engineering literature.

In the mathematics literature, rational functions on $\D^2$ with
singularities on $\T^2$ play a role in several places in complex
analysis, essentially as important extremal functions or illustrative
examples.  In \cite{KneseSchwarz} they appear as the functions
satisfying equality in a certain version of the Schwarz lemma on the
polydisk.  In \cite{AM3point} they appear as solutions of a three
point interpolation problem for bounded holomorphic functions on
$\D^2$.  In \cite{dirichlet}, polynomials with no zeros on $\D^2$ and
some zeros on $\T^2$ appear in a characterization of cyclic
polynomials for Dirichlet type spaces on the bidisk.  A major impetus
for the present paper is our previous study of a certain class of
rational functions called rational inner functions on the bidisk
\cite{KneseAPDE} where the goal was to understand all rational inner
functions and not just the regular ones (those extending analytically
past $\cD^2$) as in the important work \cite{GW}.  This distinction
plays a role in \cite{Scheinker}, a paper about interpolation problems
on the polydisk, where certain theorems are only proven for regular
rational inner functions.  Question \ref{q1} is related to the work in
\cites{AMYcara} where non-tangential convergence is studied for
general bounded analytic functions on the bidisk, and rational
functions appear as important examples.  The paper \cite{AMhankel} is
also relevant.

Although it is something of an aside, these issues are also relevant
in some problems in dynamics, specifically in the study of algebraic
$\Z^d$-actions as in \cites{Lind1, Lind2}.  While the requirement of
non-vanishing in $\D^d$ does not seem to be relevant in this context,
the integrability properties of rational functions on $\T^d$ do seem
to be of interest.  The example $p(z) = 2-z_1-z_2$ makes an appearance
as Example 7.2 of \cite{Lind1} and Example 4.3 of \cite{Lind2}.  They
point out that $G_3$ (or just $1/p$) is in $L^1(\T^2)$ and
$(z_1-1)^3/p(z)$ has absolutely convergent Fourier series.

All of this serves to point out that regularity/integrability of
rational functions on the torus and on the polydisk plays an important
role in a number of contexts, yet there does not seem an associated
theory for addressing it.
We shall give a sampling of our answers to Questions \ref{q2},\ref{q3},\ref{q1}
here in the introduction, and leave more complete answers to later
sections.

Question \ref{q2} is answered directly by giving a finite list of
generators of the ideal $\mathcal{I}_p$; see Theorem \ref{generators}.
The list is too technical for the introduction, so as a temporary
replacement we point out a characterization using an inequality.  If
$p \in \C[z_1,z_2]$ has bidegree $(n,m)$, its reflection is given by
\[
\refl{p}(z) \defn z_1^n z_2^m \overline{p(1/\bar{z}_1,1/\bar{z}_2)}.
\]

\begin{introthm} \label{intthmineq} Assume $p \in \C[z_1,z_2]$ has no zeros
  in $\D^2$ and assume $p$ and $\refl{p}$ have no common factors.  Let
  $q \in \C[z_1,z_2]$.
  Then, $q \in \mathcal{I}_p$ if
  and only if there is a  constant $c>0$ such that
\[
|q(z)|^2 \leq c((n+m)|p(z)|^2 - 2 \Re[(z_1 \partial_1
p(z)+z_2\partial_2 p(z))\overline{p(z)}])
\]
for $z \in \T^2$.  Here $(n,m) = \deg p$.
\end{introthm}

This is given as Corollary \ref{intthmcor}.
We also study the ideal
\[
\mathcal{I}^{\infty}_{p} \defn \{q \in \C[z_1,z_2]: q/p \in L^{\infty}(\T^2) \}
\]
and construct one variable polynomials $g(z_1), h(z_2)$ such that
$gh\mathcal{I}_p \subset \mathcal{I}_p^{\infty}$ in Section
\ref{secinf}.

Question \ref{q3} can be answered directly.  

\begin{introthm} \label{intthmdim} Let $p \in \C[z_1,z_2]$ have no zeros in
  $\D^2$ and assume $p$ and $\refl{p}$ have no common factors.  Let
  $N_{\T^2}(p,\refl{p})$ denote the number of common zeros of $p$ and
  $\refl{p}$ on $\T^2$ where zeros are counted with appropriate
  multiplicities, as in B\'ezout's theorem.  Then,
\[
\dim \mcP_{j,k} = (j+1)(k+1) - \frac{1}{2}N_{\T^2}(p,\refl{p}).
\]
\end{introthm}

The assumption that $p$ and $\refl{p}$ have no common factors is no
serious reduction since common factors divide every element of
$\mathcal{I}_p$.  It is intuitively clear that common zeros of $p$ and
$\refl{p}$ on $\T^2$ should occur with even multiplicity by a
perturbation argument, however we give a proof using Puiseux
series in Appendix C.

Question \ref{q1} can actually be answered in more than two variables
and it has an especially clean answer for rational inner functions,
which are a generalization of finite Blaschke products to several
variables.  

\begin{introthm} \label{intthmnontan} If $\phi:\D^d \to \cD$ is a rational inner function,
  then for \emph{every} $\zeta \in \T^d$, the limit
\[
\lim_{z \to \zeta} \phi(z)
\]
exists as $z\in \D^d$ goes to $\zeta$ non-tangentially.
\end{introthm}

A rational function $\phi = q/p$, holomorphic on $\D^d$,
is \emph{inner} if $|q|=|p|$ on $\T^d$.  By the maximum principle
$\phi$ maps $\D^d$ to $\cD$ and $q(z)$ must be of the form $\mu
z^{\alpha} \refl{p}(z)$ where $\mu \in \T$,
$\alpha$ is a multi-index, and $\refl{p}$ is the
\emph{reflection} of $p$ just as in two variables:
\[
\refl{p}(z_1,\dots, z_d) \defn z_1^{n_1} \cdots z_d^{n_d}
\overline{p(1/\bar{z}_1,\dots, 1/\bar{z}_d)}
\]
assuming the multidegree of $p$ is $(n_1,\dots, n_d)$; see
\cite{rudin}, Theorem 5.2.5.

To say $z \to \zeta$ non-tangentially means the quantities
$|z_j-\zeta_j|$ for $j=1,\dots, d$ and $1-|z_j|$ for $j=1,2,\dots, d$
are all comparable as $z =(z_1,\dots,z_d)\to \zeta= (\zeta_1,\dots,
\zeta_d)$. This result is perhaps surprising because rational inner
functions need not be continuous up to
$\D^d$.

\section*{Acknowledgments}
This paper is partially inspired from the ICMS workshop ``Function
theory in several complex variables in relation to modelling
uncertainty.''  I would like to thank the organizers of that
conference: Jim Agler, Zinaida Lykova, and Nicholas Young, as well as
attendees Joseph Ball and Eric Rogers for pointing out the reference
\cite{KM}.  I thank John McCarthy for useful conversations, and James
Pascoe for graciously allowing me to see an early version of a paper
containing the construction of Example \ref{pascoeex}.  I also owe a
great debt to the references and the authors of \cites{Pickbook,
  AMYcara, BSV, Bickel, CW, GW}.

\section{Overview of the paper} \label{overview}

Question \ref{q2} is addressed by studying the Hilbert space $\Lp$,
where $d\sigma$ is normalized Lebesgue measure on $\T^2$, and certain
special orthogonal decompositions in $\Lp$.  These make it possible to
construct generators for the ideal $\mathcal{I}_p$, thus answering
Question \ref{q2}.  Sections \ref{prelim}-\ref{secthma} are
occupied with this.  

The special orthogonal decompositions of $\Lp$ are used in
\cites{BickelKnese, KneseAPDE} to establish an important sums of
squares formula.  If $p \in \C[z_1,z_2]$ has no zeros in $\D^2$ and
bidegree $(n,m)$ then
\begin{equation} \label{sosformula}
|p(z)|^2 - |\refl{p}(z)|^2 = (1-|z_1|^2) \sum_{j=1}^{n} |A_j(z)|^2 +
(1-|z_2|^2)\sum_{j=1}^{m} |B_j(z)|^2
\end{equation}
for some $A_1,A_2,\dots, A_n, B_1,\dots, B_m \in \C[z_1,z_2]$.  This
formula has several applications: Agler's Pick interpolation theorem,
two variable matrix monotone functions, and determinantal formulas for
distinguished varieties, polynomials with no zeros on $\D^2\cup
(\C^2\setminus \cD^2)$, and hyperbolic polynomials; see \cites{AMY, CW,
  KneseDV, KneseAPDE, KneseSemi}.  Thus, it should pay off to
understand it better.  The Hilbert space approach for proving this
formula produces the $A_j$ and $B_j$ as elements of $\mathcal{I}_p$,
and understanding this approach in depth is the key to addressing
Question \ref{q3}.  A method adapted from Ball-Sadosky-Vinnikov
\cite{BSV} shows that minimal sums of squares formulas for $p$ are in
correspondence with joint invariant subspaces of a pair of commuting
truncated shift operators on a finite dimensional subspace of $\Lp$.
The joint eigenvalues of this pair of operators are directly related
to common zeros of $p$ and $\refl{p}$ and this enables us to compute
the dimension of $\mcP_{j,k}$ in terms of common zeros of $p$ and
$\refl{p}$, thus answering Question \ref{q3}.  Sections
\ref{seccomm}-\ref{secdimthm} are occupied with this.  We include a
background section on intersection multiplicities.

Section \ref{secnontan} is devoted to addressing Question \ref{q1}.
The beginning of this section is actually independent of the rest of
paper and hinges on a proposition stating that the bottom term in the
homogeneous expansion of $p \in \C[z_1,\dots,z_d]$ at a boundary zero
has no zeros in a product of half-planes.  After addressing Theorem
\ref{intthmnontan} we make some connections to earlier material. Namely, if a
rational inner function on $\D^2$ has higher regularity at a boundary
point where $p$ vanishes then this forces a larger intersection
multiplicity of this common zero of $p$ and $\refl{p}$.  We get the
interesting conclusion that the number of points with a certain amount
of regularity but no higher is finite and can be explicitly bounded.

We have attempted to make this paper as accessible as possible.
Consequently, there are several background sections and appendices
which experts in one area or another should be able to skim.  There is
a section with notation at the end of the paper. For further
background reading we recommend: \cites{Pickbook, CW} for reproducing
kernels and bounded analytic functions on $\D^2$, \cite{dangeloineq}
for positive semi-definite polynomials, \cite{simon} for the study of
measures of the form $\frac{1}{|p(e^{i\theta})|^2}d\theta$ in one
variable (Bernstein-Szeg\H{o} measures), \cites{CLO, Fischer, Fulton}
for algebraic curves and intersection multiplicities.

We begin with an example.

\section{Example: $p(z) = 2-z_1-z_2$}
The polynomial $p(z) = 2-z_1-z_2$ is
the simplest non-trivial example that can be used to illustrate many
of the theorems of this paper.

Note $\refl{p}(z) = 2z_1z_2-z_1-z_2$ and
\[
f(z) = \frac{\refl{p}(z)}{p(z)} = \frac{2z_1z_2 - z_1-z_2}{2-z_1-z_2}
\]
is a rational inner function which does not extend continuously to
$\T^2$.  To see this consider the path in $\D$ given by
$z_{\epsilon}(t)= (1-\epsilon e^{it}\cos t, 1- \epsilon e^{-it} \cos
t)$ where $t \nearrow \pi/2$ and $\epsilon >0$ is small. 
Then, for $t \in (0,
\frac{\pi}{2})$ 
\[
f(z_{\epsilon}(t)) = -1 +\epsilon
\]
but $z_{\epsilon}(\frac{\pi}{2}) = (1,1)$.  Theorem \ref{intthmnontan}
tells us that despite this discontinuity, $f$ has a limit along any
non-tangential path to $\T^2$.  Of course, $z_{\epsilon}$ approaches
tangentially, so there is no contradiction.  The key observation to
proving non-tangential convergence at $(1,1)$ is to expand 
\[
f(1-\zeta_1,1-\zeta_2) = -1 + \frac{2\zeta_1\zeta_2}{\zeta_1+\zeta_2}.
\]
If $z \to (1,1)$ non-tangentially in $\D^2$, then $\zeta \to (0,0)$
non-tangentially in $RHP^2$; $RHP=$ the right half plane.  This means
$|\zeta_1|,|\zeta_2|,\Re \zeta_1, \Re \zeta_2$ are all comparable
quantities in a non-tangential approach region and so $|\zeta_1 +
\zeta_2| \geq c|\zeta_1|$. This is enough to show $f(z) \to
-1$ as $z\to (1,1)$ non-tangentially.  A similar estimate will hold for
more general rational inner functions.  Specifically, the lowest order
homogeneous term of $p(1-\zeta_1,1-\zeta_2)$ will be non-vanishing in
$RHP^2$.  

It is also worth pointing out that a function can be bounded
non-tangentially at every point even though it is globally unbounded.
Let 
\[
g(z) = \frac{1-z_1}{2-z_1-z_2}.
\]
Then, $g(1-\zeta_1,1-\zeta_2)= \zeta_1/(\zeta_1+\zeta_2)$ which is
bounded in any non-tangential approach region to $(0,0)$ in
$RHP^2$.
At the same time, if we let $z(\theta) = (1-\theta^2)(e^{i\theta},
e^{-i\theta})$ then for $\theta$ close to $0$
\[
\begin{aligned}
|g(z(\theta))| &=
\frac{|1-(1-\theta^2)e^{i\theta}|}{2-2(1-\theta^2)\cos\theta} \\
&\geq C \frac{|\theta|}{1-\cos\theta + \theta^2\cos \theta} \\
& \geq C \frac{1}{|\theta|}
\end{aligned}
\]
which is unbounded.

The only common zero of $p$ and $\refl{p}$ on $\T^2$ is the point
$(1,1)$, and this zero occurs with multiplicity $2$.  Therefore, by
Theorem \ref{intthmdim}, the space $\mcP_{0,0}$ is trivial which just means
that
\[
\frac{1}{2-z_1-z_2}
\]
is not in $L^2(\T^2)$.  Of course, this could be checked by direct
computation but we emphasize that Theorem \ref{intthmdim} lets us show this
algebraically.  Also, Theorem \ref{intthmdim} tells us that $\mcP_{j,k} =
(j+1)(k+1)-1$ so that the space $\mathcal{I}_p$ has co-dimension one
among all polynomials, and by Theorem \ref{intthmineq} for all $q \in
\mathcal{I}_p$, $q(1,1) =0$ .  Thus, $q/p \in L^2(\T^2)$ iff $q(1,1) =
0$.  So, for example, we automatically know $g$ above is in
$L^2(\T^2)$.  

As mentioned in the overview section, these results are proven by
examining a sums of squares formula which in this case is
\[
|p|^2 - |\refl{p}|^2 = (1-|z_1|^2)2|1-z_2|^2 + (1-|z_2|^2)2|1-z_1|^2.
\]

It follows from later work that we can multiply elements of
$\mathcal{I}_p$ by some specific one variable polynomials
$p_1(z_1),p_2(z_2)$ to force $p_1p_2 \mathcal{I}_p \subset
\mathcal{I}^{\infty}_{p}$.  In this example, $(1-z_1)(1-z_2) \in
\mathcal{I}^{\infty}_{p}$ so that $G_2$ from the introduction is in
$L^{\infty}(\T^2)$.  From this it is not hard to reason that $G_1$ is
four times continuously differentiable and so if $G_1$ has Fourier
coefficients $\{a_{n,m}\}$ then $\sum_{n,m} (n+1)^2(m+1)^2 |a_{n,m}|^2
< \infty$ and therefore $\{a_{n,m}\} \in \ell^1$ by Cauchy-Schwarz.
This shows we can recover many of the details of \cite{goodman} from
our theorems.

\section{Background: Vector polynomials and matrix functions} \label{secvec}

In this section we make a few general observations about vectors and
vector polynomials as well as vector-valued Hardy spaces and
reproducing kernels.  Let
\[
\begin{aligned}
  &\C^n = \text{The space of $n$-dimensional column vectors} \\
  &\C^{1\times n} = \text{The space of $n$-dimensional row vectors} \\
  &\C^{n\times m} = \text{The space of $n\times m$  matrices with entries in $\C$}\\
& V[z_1,z_2] = \text{The space of two variable polynomials with
  coefficients in $V$}
\end{aligned}
\]
where $V$ is some vector space such as $\C^n, \C^{1\times n},
\C^{m\times n}$.

A theorem which is useful for dealing with vector polynomial equations
is the polarization theorem for holomorphic functions. See
\cite{Dangelo} for a proof.

\begin{theorem}[Polarization Theorem] \label{polarthm}
  Suppose $F:\Omega \times \Omega^* \to \C$ is holomorphic where
  $\Omega \subset \C^n$ is a domain and $\Omega^* = \{\bar{z}: z\in
  \Omega\}$.  If $F(z,\bar{z})= 0$ for all $z\in \Omega$, then
  $F(z,w) =0$ for all $(z,w) \in \Omega\times \Omega^*$.
\end{theorem}

An important instance of the polarization theorem is the following
proposition.

\begin{prop} \label{isomprop} If $\vec{A} \in \C^n[z_1,z_2], \vec{B}
  \in \C^m[z_1,z_2]$ and if $|\vec{A}(z)|^2 = |\vec{B}(z)|^2$, then
  there exists an $m \times n$ isometric matrix $U$ such that
  $U\vec{A}(z) = \vec{B}(z)$.  If $\vec{A}$ and $\vec{B}$ have
  linearly independent entries, then $m=n$ and $U$ is a unitary.
\end{prop}

\begin{proof}
By the polarization theorem 
\[
\vec{A}(w)^* \vec{A}(z) = \vec{B}(w)^*\vec{B}(z),
\]
in this case $F(z,w) = \vec{A}(\bar{w})^* \vec{A}(z) -
\vec{B}(\bar{w})^* \vec{B}(z)$ and by assumption $F(z,\bar{z}) \equiv
0$.

In this situation, $\vec{A}$ and $\vec{B}$ are related by an isometric
matrix.  Indeed, the map
\[
\vec{A}(z) \mapsto \vec{B}(z)
\]
extends linearly to an isometry from the span of the vectors on the
left to the span of the vectors on the right.  Indeed, if $a: \C^2 \to
\C$ is a finitely supported function, then for $v_1 =
\sum_{z \in \C^2} a(z) \vec{A}(z)$ and $v_2 = \sum_{z \in \C^2} a(z)
\vec{B}(z)$ we have 
\[
v_1^*v_1 = \sum_{z,w\in \C^2} \overline{a(w)} a(z)
\vec{A}(w)^*\vec{A}(z) = \sum_{z,w\in \C^2} \overline{a(w)} a(z)
\vec{B}(w)^*\vec{B}(z) = v_2^* v_2.
\]
Thus, $v_1 \mapsto v_2$ is at once well-defined ($|v_1|=0$ iff
$|v_2|=0$), linear, and isometric.

This isometry is initially defined on $\text{span}\{\vec{A}(z): z \in
\C^2\}$, but it can be extended to all of $\C^n$ by standard linear
algebra and can then be realized via an $m\times n$ isometric matrix
$V$: $V\vec{A}(z) = \vec{B}(z)$.

If the entries of $\vec{A}$ and $\vec{B}$ form a linearly independent
set of polynomials, then $\text{span}\{\vec{A}(z): z \in \C^2\}=\C^n$
and $\text{span}\{\vec{B}(z): z \in \C^2\} = \C^m$, and $m=n$ because
these spaces are related by an isometry.  
\end{proof}

Often in this paper, we break apart a vector polynomial $\vec{A} \in
\C^N[z_1,z_2]$ into one variable pieces.  For instance, if $\vec{A}$
has degree at most $n-1$ in $z_1$ then it is possible to write
\[
\vec{A}(z) = A(z_2) \Lambda_n(z_1)
\]
where $A \in \C^{N\times n}[z_2]$ is a matrix polynomial and
\begin{equation} \label{Lam}
\Lambda_n(z_1) \defn \begin{pmatrix} 1 \\ z_1 \\ \vdots \\
  z_1^{n-1} \end{pmatrix} \in \C^{n}[z_1].
\end{equation}
This is simply a way of extracting the coefficients of powers of $z_1$
into a matrix.  

Vector polynomials appear most often in this paper in relation to
reproducing kernels.  If $\mcH$ is a finite dimensional Hilbert space
of polynomials and if $\vec{H}$ is a vector polynomial whose entries
form an orthonormal basis for $\mcH$, then $k_w(z) = \vec{H}(w)^*
\vec{H}(z)$ is a reproducing kernel for $\mcH$ in the sense that
\[
\ip{f}{k_w}_{\mcH} = f(w) 
\]
for all $f \in \mcH$. This formula can be proven by first verifying it
for the entries of $\vec{H}$; the general formula then follows by
linearity. 

If the entries of $\vec{H}$ are not an orthonormal basis, then
$\vec{H}(w)^*\vec{H}(z)$ need not be a reproducing kernel for $\mcH$.
Nevertheless, an expression of this form can be characterized as being
a \emph{positive semi-definite kernel function} which abstractly
refers to a function $k: \Omega \times \Omega \to \C$ with the
property that for any finitely supported function $a: \Omega \to \C$
we have
\[
\sum_{z,w \in \Omega} a(z) \overline{a(w)} k(z,w) \geq 0.
\]
Here $\Omega$ is just a set, but if $\Omega$ is actually a domain and
$k(z,w)$ is a polynomial in $z,\bar{w}$ then we get the following.

\begin{prop}\label{kernelprop} Suppose $\Omega$ is a domain in
$\C^n$ and $k$ is a positive semi-definite kernel function such that
$k(z,w)$ is a polynomial in $z,\bar{w}$. Then, there exists a vector
polynomial $\vec{H}$ such that $k(z,w) = \vec{H}(w)^*\vec{H}(z)$.  
\end{prop}

\begin{proof} We build a Hilbert space $\mcH$ and an inner product
  such that $k(z,w)$ is the reproducing kernel.  Indeed, let $\mcH$ be
  the finite dimensional vector space $\text{span} \{k_w: w \in
  \Omega\}$ where $k_w(z) \defn k(z,w)$.  Declare that
  $\ip{k_w}{k_z}_{\mcH} = k(z,w)$ and extend by linearity to all of
  $\mcH$.  This is well-defined because for any finitely supported
  function $a: \Omega \to \C$, if the polynomial $f(z) = \sum_{w \in
    \Omega} \overline{a(w)} k_w(z)$ is identically zero then
  $\ip{f}{g}_{\mcH} = 0$ simply because if $g = \sum_{z\in \Omega}
  \overline{b(z)} k_z$ then
\[
\ip{f}{g}_{\mcH} = \sum_{z \in \Omega} b(z) f(z) = 0.
\]
This is also a bona fide inner product because if $\ip{f}{f}_{\mcH} =
0$, then $|f(z)|^2= |\ip{f}{k_z}_{\mcH}|^2 \leq \ip{f}{f}_{\mcH}
k(z,z) = 0$ for all $z$ so that $f=0$. This inequality follows from
Cauchy-Schwarz for an a priori semi-definite inner product.  With
$\mcH$ built it is not hard to show $k(z,w) = \vec{H}(w)^*\vec{H}(z)$
for some $\vec{H}$ as before.
\end{proof} 

Next, we turn to some background on vector-valued Hardy spaces. Let
$\Hrow$ denote the row-vector valued Hardy space on the unit circle:
$\Hrow \defn H^2(\T) \otimes \C^{1\times n}$.  Row vectors end up
being natural for what follows because we chose column representations
for vector polynomials.  An $n\times n$ matrix function $\Phi$ whose
entries are rational functions of $z\in\C$ with no poles in $\cD$ is a
\emph{matrix rational inner function} if $\Phi$ is unitary valued on $\T$:
\[
\Phi(z)^* \Phi(z) = I \qquad z \in \T.
\]
By the maximum principle, $\|\Phi(z)\| \leq 1$ for all $z \in
\D$---actually, by the maximum principle applied to
$v_{1}^*\Phi(z)v_2$ for all $v_1,v_2 \in \C^n$.  Note $\Hrow \Phi
\subset \Hrow$; multiplication on right by $\Phi$ looks odd, but since
our space consists of row vector valued functions it is correct.

\begin{prop}\label{rowhardy}
  Let $\Phi$ be a $n\times n$ matrix rational inner function.  The
  space $\Hrow\ominus \Hrow \Phi$ is finite dimensional, consists of
  rational vector-valued functions with no poles in $\cD$, and has
  reproducing kernel
\[
K_{w}(z) = \frac{I-\Phi(w)^*\Phi(z)}{1-\bar{w} z}.
\]
\end{prop}
The last statement means that for any $\vec{f} \in \Hrow \ominus \Hrow
\Phi$ and for $w \in \D, v \in \C^{1\times n}$
\[
\ip{\vec{f}}{vK_{w}}_{\Hrow} = \ip{\vec{f}(w)}{v}_{\C^{1\times n}}.
\]
Because of this reproducing property $K$ is a positive semidefinite
kernel function which for matrix valued kernels means that for any
finitely supported function $\vec{a}: \D \to \C^n$ we have
\[
\sum_{z,w \in \D} \vec{a}(w)^* K(z,w) \vec{a}(z) \geq 0.
\]
More generally though, if $\Phi$ is an analytic $n\times m$ matrix
valued function such that $\|\Phi(z)\| \leq 1$ for all $z \in \D$,
then $K$ defined as above will still be a positive semidefinite
kernel. See \cite{Pickbook} for instance.

\begin{proof}[Proof of Proposition]
  Let $\text{adj}(\Phi)$ be the adjugate or ``classical adjoint'' of
  $\Phi$.  Then, $\Phi\ \text{adj}(\Phi)=\det(\Phi)I_n$.  Note that $b
  \defn \det(\Phi)$ is a finite Blaschke product.

Now, $\Hrow \det(\Phi)I_n\subset \Hrow \Phi$ so that $\Hrow \ominus
\Hrow \Phi \subset \Hrow \ominus \Hrow \det(\Phi)I_n$, and the latter
space is a direct sum of the scalar spaces $H^2 \ominus b H^2$.  This
last space is finite dimensional and consists of rational functions
with no poles in $\cD$.  To see this, write $b = \mu
\frac{\refl{p}}{p}$ where $|\mu|=1$, $p \in \C[z]$ and $\refl{p}(z) = z^N
\overline{p(1/\bar{z})}$ for some $N$.  Then, $f \perp bH^2$ iff $f
\bar{z}^N p = \bar{z} \bar{p} \bar{g}$ for some $g \in H^2$.  Or,
$fp=\bar{z} \refl{p} \bar{g} \in H^2 \cap Z^{N-1} \overline{H^2}$
which means $fp$ is a polynomial of degree at most $N-1$ and therefore
$f$ is a rational function with denominator $p$ and numerator with
degree at most $N-1$.

The formula for the reproducing kernel is a basic calculation
depending on the fact that $\Phi$ is unitary valued on $\T$ and
bounded and holomorphic in $\D$.  We omit the details.
\end{proof}

A useful construction of matrix rational inner functions is given
below.

\begin{lemma} \label{lurkisom} If $U$ is a unitary matrix written in
  block form $U = \begin{pmatrix} U_{11} & U_{12} \\ U_{21} &
    U_{22} \end{pmatrix}$ and if $\Xi(z) \defn U_{11} + z U_{12}(I-z
  U_{22})^{-1} U_{21} $, then $\Xi$ is a matrix valued rational inner
  function.
\end{lemma}
\begin{proof}
Observe
\[
\begin{pmatrix} U_{11} & U_{12} \\ U_{21} &
    U_{22} \end{pmatrix} \begin{pmatrix} I \\ z(I-z U_{22})^{-1}U_{21} \end{pmatrix}
= \begin{pmatrix} \Xi(z) \\ (I-z U_{22})^{-1} U_{21} \end{pmatrix}.
\]
Since $U$ is a unitary
\[
I + |z|^2 U_{21}^*(I-\bar{z} U_{22}^*)^{-1} (I-zU_{22})^{-1}
U_{21} = \Xi(z)^*\Xi(z) + U_{21}^*(I-\bar{z} U_{22}^*)^{-1} (I-zU_{22})^{-1}
U_{21}
\]
which rearranges to
\[
I-\Xi(z)^*\Xi(z) = (1-|z|^2) U_{21}^*(I-\bar{z} U_{22}^*)^{-1} (I-zU_{22})^{-1}
U_{21}.
\]
This shows $\Xi(z)$ is unitary valued on $\T$ except possibly at
points where $\det(I-z U_{22})=0$. But, there can only be finitely
many such points and since $\Xi$ is bounded near these points any
singularities (which are at worst poles) must be removable.  Thus,
$\Xi$ is unitary valued on all of $\T$ and holomorphic on $\cD$.
\end{proof}

\section{Background: the Hilbert space $\Lp$} \label{prelim}

\begin{definition} Let $p \in \C[z_1,z_2]$ have degree $(n,m)$ and set
\[
\refl{p}(z) \defn z_1^n z_2^m \overline{p(1/\bar{z}_1,1/\bar{z}_2)},
\text{ the \emph{reflection} of } p.
\]
We shall say $p$ is \emph{semi-stable} if $p$ has no zeros in $\D^2$
and if $p$ and $\refl{p}$ have no common factors.
\end{definition}
Semi-stable polynomials have at most finitely many
zeros on $\T^2$ by B\'ezout's theorem since zeros on $\T^2$ are shared
with the reflected polynomial.

Let $p$ be semi-stable.  We will work in the Hilbert space $\Lp$ where
$d\sigma$ is normalized Lebesgue measure on $\T^2$. All orthogonal
complements and orthogonal direct sums below are taken with respect to
this space.

Let 
\[
\mcP_{j,k} \defn \{f \in \C[z_1,z_2]\cap \Lp :
\deg f \leq (j,k)\}
\]
where $\deg f$ denotes the bidegree of $f$---the ordered pair
consisting of the degree in $z_1$, the degree in $z_2$.

\begin{notation} \label{not:spaces}
We define a number of spaces using orthogonal complements.
\[
\mcE_1 \defn \mcP_{n-1,m} \ominus z_2 \mcP_{n-1,m-1}, \quad \mcF_1 \defn
\mcP_{n-1,m} \ominus \mcP_{n-1,m-1},
\]
\[
\mcE_2 \defn \mcP_{n,m-1} \ominus z_1 \mcP_{n-1,m-1}, \quad \mcF_2 \defn
\mcP_{n,m-1} \ominus \mcP_{n-1,m-1},
\]
and 
\[
\mcG \defn \mcP_{n-1,m-1}.
\]
\end{notation}
By Lemma 6.7 and Theorem 7.4 of \cite{KneseAPDE}, 
\[
\dim \mcE_1 = \dim \mcF_1 = n \text{ and } \dim \mcE_2 = \dim \mcF_2 =
m.
\]
The computation of $\dim \mcG$ is a main result of this paper.  It is
clear that $\dim \mcG \leq nm$.

\begin{notation} \label{not:vecs}
  We let $\vec{E}_1,\vec{F}_1 \in \C^n[z_1,z_2], \vec{E}_2, \vec{F}_2
  \in \C^m[z_1,z_2]$ be vector polynomials whose entries form an
  orthonormal basis for $\mcE_1, \mcF_1,\mcE_2, \mcF_2$, resp.  Let
  $\vec{G}$ be a vector polynomial whose entries form an orthonormal
  basis for $\mcG$.  Note that these vector polynomials are unique up
  to multiplication by a unitary matrix on the left.
\end{notation}

\begin{prop} \label{reflectionprop}
Let $p$ be semi-stable.  Using Notation \ref{not:vecs},
there exist choices of orthonormal bases for $\mcE_1,
\mcE_2$ so that
\[
\vec{E}_1(z) = z_1^{n-1} z_2^m \overline{\vec{F}_1(1/\bar{z}_1,
  1/\bar{z}_2)} \text{ and } \vec{E}_2(z) = z_1^{n} z_2^{m-1} \overline{\vec{F}_2(1/\bar{z}_1,
  1/\bar{z}_2)}.
\]
\end{prop}

Throughout the paper, we will assume $\vec{F}_1,\vec{F}_2$ satisfy the
relationship above.

\begin{proof}
The map on $\Lp$ given by 
\[
Tf = z_1^{n-1} z_2^{m}\overline{f(1/\bar{z}_1,1/\bar{z}_2)}
\]
is an anti-unitary meaning
\[
\ip{Tf}{Tg}_{\Lp} = \ip{g}{f}_{\Lp}
\]
for all $f,g \in \Lp$.  So, it preserves orthogonality and maps an
orthonormal basis to an orthonormal basis.  Because of this, $T$ maps
$\mcF_1$ to $\mcE_1$ and therefore the entries of $z_1^{n-1}z_2^m
\overline{\vec{F}_1(1/\bar{z}_1,1/\bar{z}_2)}$ are an orthonormal
basis for $\mcE_1$.  Hence, this vector polynomial is a valid choice
for $\vec{E}_1$. The formula for $\vec{E}_2$ is similar.
\end{proof}

\begin{definition}
If vector polynomials $\vec{A}_1,\vec{A}_2$ satisfy
\[
|p(z)|^2 - |\refl{p}(z)|^2 = \sum_{j=1}^{2} (1-|z_j|^2)|\vec{A}_j(z)|^2
\]
the above formula will be called an \emph{Agler decomposition} and
$(\vec{A}_1,\vec{A}_2)$ will be called an \emph{Agler pair} for $p$.
\end{definition}
Note that this formula can be polarized.  The following key theorem
says, among other things, that $(\vec{E}_1,\vec{F}_2)$ and
$(\vec{F}_1, \vec{E}_2)$ are both Agler pairs. It is proven as
Corollary 7.5 and Proposition 5.5 of \cite{KneseAPDE}. It can also be
extracted from \cite{BickelKnese}.

  \begin{theorem} \label{thmsos}
Let $p$ be semi-stable. Then, using Notation \ref{not:vecs}
\[
\begin{aligned}
\overline{p(w)} p(z) - \overline{\refl{p}(w)} \refl{p}(z) =&
(1-\bar{w}_1 z_1)\vec{E}_1(w)^* \vec{E}_1(z) + (1-\bar{w}_2 z_2)
\vec{F}_2(w)^* \vec{F}_2(z) \\
=& (1-\bar{w}_1 z_1)\vec{F}_1(w)^* \vec{F}_1(z) + (1-\bar{w}_2 z_2)
\vec{E}_2(w)^* \vec{E}_2(z) \\
=& \sum_{j=1}^{2} (1-\bar{w}_j z_j) \vec{F}_j(w)^* \vec{F}_j(z) \\ &+
(1-\bar{w}_1 z_1)(1-\bar{w}_2 z_2) \vec{G}(w)^* \vec{G}(z)
\end{aligned} 
\]
and 
\[
\vec{G}(w)^* \vec{G}(z)= \frac{\vec{E}_1(w)^* \vec{E}_1(z) -
  \vec{F}_1(w)^* \vec{F}_1(z)}{1-\bar{w}_2 z_2}
= \frac{\vec{E}_2(w)^* \vec{E}_2(z) - \vec{F}_2(w)^*
  \vec{F}_2(z)}{1-\bar{w}_1 z_1}.
\]
\end{theorem}

\begin{example}
  To get a feel for the theorem, it helps to look at a trivial example
  $p(z) = 1$ thought of as a polynomial of degree $(1,2)$ so
  that $\refl{p}(z) = z_1 z_2^2$.  Then, $\mcE_1 = \C, \mcF_1 = \C
  z_2^2, \mcE_2 = \text{span}\{1,z_2\} , \mcF_2= \text{span}\{z_1,
  z_1z_2\}, \mcG = \text{span}\{1, z_2\}$.  And, $\vec{E}_1(z) = 1,
  \vec{F}_1(z) = z_2^2, \vec{E}_2(z) = \begin{pmatrix} 1 \\
    z_2 \end{pmatrix}, \vec{F}_2(z) = z_2 \vec{E}_2(z), \vec{G}(z) =
  \vec{E}_2(z)$.  The resulting formulas (evaluated on the diagonal
  $z=w$) are
\[
\begin{aligned}
1 - |z_1z_2^2|^2 & = (1-|z_1|^2) + (1-|z_2|^2) |z_1|^2(1+|z_2|^2) \\
&= (1-|z_1|^2) |z_2^2|^2 + (1-|z_2|^2)(1+|z_2|^2)
\end{aligned}
\]
and so on.  Of course everything is so easy in this case because $\Lp
= L^2(\T^2)$. \eox
\end{example}    

\begin{definition} \label{maxmin} Let $p \in \C[z_1,z_2]$ be
  semistable and $\deg p = (n,m)$.  Define $\vec{E}_j, \vec{F}_j$ as
  above. We will refer to $(\vec{E}_1,\vec{F}_2)$ as the \emph{max-min
    Agler pair} of $p$ and $(\vec{F}_1,\vec{E}_2)$ as the
  \emph{min-max Agler pair} of $p$.
\end{definition}

In general, there are important orthogonality relations which hold in
$\Lp$. Note that $\Lp \subset L^2(d\sigma)= L^2(\T^2)$ and therefore
Fourier coefficients, $\hat{f}(j,k) \defn \int_{\T^2} f \bar{z}_1^j
\bar{z}_2^k d\sigma$, are well-defined for elements of $\Lp$. Let
$\text{supp} \hat{f} = \{(j,k)\in \Z^2: \hat{f}(j,k)\ne 0\}$.  

\begin{theorem} \label{thmorth} Let $p$ be semi-stable with $\deg p =
  (n,m)$.  Using Notation \ref{not:spaces} we have that in $\Lp$
\[
\mcF_1 \perp \{f \in \Lp: \text{supp} \hat{f}
\subset \{(j,k): (j \geq 0), (k < m)\}\}
\]
\[
\mcE_1 \perp \{f \in \Lp: \text{supp} \hat{f} \subset \{(j,k):( j<n), (k>0)\}\}
\]
\[
\mcF_2 \perp \{f \in \Lp: \text{supp} \hat{f}
\subset \{(j,k): (j < n), (k \geq 0)\}\}
\]
\[
\mcE_2 \perp \{f \in \Lp: \text{supp} \hat{f} \subset \{(j,k):( j>0), (k<m)\}\}
\]
\[
\begin{aligned}
p &\perp \{f\in \Lp: \text{supp} \hat{f} \subset \{(j,k):( k>0) \text{ or
} (k=0 \text{ and } j>0)\}\} \\
p & \perp \{f \in \Lp: \text{supp} \hat{f} \subset \{(j,k):( j>0) \text{ or
} (j=0 \text{ and } k>0)\}\}
\end{aligned}
\]
\[
\begin{aligned}
\refl{p} & \perp \{f\in \Lp: \text{supp} \hat{f} \subset \{(j,k):( k<m) \text{ or
} (k=m \text{ and } j<n)\}\} \\
\refl{p} & \perp \{f \in \Lp: \text{supp} \hat{f} \subset \{(j,k):( j<n) \text{ or
} (j=n\text{ and } k<m)\}\}
\end{aligned}
\]
\end{theorem}

We recommend drawing pictures of the various support sets above.
In Appendix A, we explain how this follows from the work in
\cites{KneseAPDE, BickelKnese}.  One direct consequence we use later
is
\begin{equation} \label{pperp}
\begin{aligned} 
p &\perp \bar{z}_1^j z_2^k \mcG \text{ for } j \geq 0 \text{ and } k >0
\\
\refl{p} &\perp \bar{z}_1^j z_2^k \mcG \text{ for } j \geq 0 \text{
  and } k \geq 0 \text{ (sic). }
\end{aligned}
\end{equation}

An important corollary of the above orthogonality conditions is the
following.

\begin{corollary} \label{cordecomp} Using the setup of the previous
  theorem, if $N\geq n-1, M\geq m-1$, then
\[
\mcP_{N,M} = \mcG \oplus \bigoplus_{j=0}^{N-n} z_1^j \mcF_2 \oplus
\bigoplus_{k=0}^{M-m} z_2^k \mcF_1 \oplus \bigoplus_{\substack{0\leq j
    \leq N-n  \\0\leq k \leq M-m}} z_1^jz_2^k\C \refl{p} \\
\]
\end{corollary}

A direct sum of the form $\bigoplus_{j=0}^{-1}$ is to be interpreted
as the trivial subspace.  Before we prove the corollary we discuss a
few special cases and variations.
The case $N=n, M=m$ is 
\begin{equation} \label{Pnm1}
\mcP_{n,m} = \mcG \oplus \mcF_1 \oplus \mcF_2 \oplus \C\refl{p}.
\end{equation}
If we apply the anti-unitary reflection operation $f \mapsto z_1^n
z_2^m\overline{f(1/\bar{z}_11/\bar{z}_2)}$, we get the decomposition
\begin{equation} \label{Pnm2}
\mcP_{n,m} = z_1z_2 \mcG \oplus z_1 \mcE_1 \oplus z_2 \mcE_2 \oplus \C
p.
\end{equation}

A couple other useful variations of the corollary are
\begin{equation} \label{cordecomp1}
\mcP_{N,M} = z_2^{M-m+1} \mcG \oplus \bigoplus_{j=0}^{N-n} z_1^j
\mcF_2 \oplus \bigoplus_{k=0}^{M-m} z_2^k \mcE_1 \oplus
\bigoplus_{\substack{0\leq j \leq N-n \\ 0\leq k \leq M-m}} z_1^j z_2^k
  \C \refl{p}
\end{equation}
\begin{equation} \label{cordecomp2}
\mcP_{N,M} = z_2^{M-m+1} \mcG \oplus \bigoplus_{j=0}^{N-n} z_2^{M-m+1}
z_1^j \mcF_2 \oplus \bigoplus_{k=0}^{M-m} z_1^{N-n+1} z_2^{k} \mcE_1
\oplus \bigoplus_{\substack{0\leq j \leq N-n \\ 0\leq k \leq M-m}}
z_1^jz_2^k \C p
\end{equation}
The formula \eqref{cordecomp1} follows from applying the reflection
operation $f\mapsto
z_1^{n-1}z_2^M\overline{f(1/\bar{z}_1,1/\bar{z}_2)}$ to $\mcP_{n-1,M}
= \mcG \oplus \bigoplus_{k=0}^{M-m} z_2^k \mcF_1$ to get $\mcP_{n-1,M}
= z_2^{M-m+1} \mcG \oplus \bigoplus_{k=0}^{M-m} z_2^k \mcE_1$.  The
formula \eqref{cordecomp2} follows from reflecting formula
\eqref{cordecomp1} and then interchanging the roles of $z_1$ and
$z_2$.  

\begin{remark} \label{quadremark}
A consequence of \eqref{cordecomp1} or \eqref{cordecomp2} and Theorem \ref{thmorth}
is that $\mcP_{N,M} \ominus z_2^{M-m+1} \mcG$ is orthogonal to the
entire quadrant
\[
\{f\in \Lp: \text{supp} \hat{f} \subset \{ (j,k): j < n \text{ and } k
>  M- m\} \}.
\]
Another way to say this is if $g \in \Lp$ has finite Fourier support
(i.e. is a Laurent polynomial) and the Fourier support is contained in
\[
\{(j,k): j \geq 0 \text{ and } k<m\}
\]
and if $g \perp \mcG$, then in fact $g$ is orthogonal to 
\[
\{f \in \Lp: \text{supp} \hat{f} \subset \{ (j,k): j<n \text{ and } k
\geq 0\}\}.
\]
This follows by multiplying everything in the previous statement by a
power of $\bar{z}_2$.
\end{remark}

\begin{proof}[Proof of Corollary \ref{cordecomp}]
  The fact that all of the spaces involved are pairwise orthogonal
  follows directly from Theorem \ref{thmorth}.

  Since by definition $\mcP_{n-1,m-1} = \mcG$, by induction it is
  enough to show for $j,k \geq 0$
\[
\mcP_{n+j,m-1+k} = \mcP_{n-1+j,m-1+k} \oplus z_1^j \mcF_2 \oplus
\bigoplus_{\alpha = 0}^{k-1} z_1^j z_2^\alpha \C \refl{p}
\]
and by symmetry a similar relation holds for $\mcP_{n-1+j,m+k}$.

Our orthogonality relations directly show
\begin{equation} \label{orthshow}
z_1^j \mcF_2 \oplus
\bigoplus_{\alpha = 0}^{k-1} z_1^j z_2^\alpha \C \refl{p} \subset
\mcP_{n+j,m-1+k} \ominus \mcP_{n-1+j,m-1+k}.
\end{equation}
The space on the left is $m+k$ dimensional.  The space on the right is
at most $m+k$ dimensional.  Indeed, more than $m+k$ elements in this
space would necessarily be linearly dependent as some combination of
them would have no Fourier support on $\{(n+j, 0),(n+j,1),\dots,(n+j,
m-1+k)\}$ and hence would be orthogonal to itself.

Since the dimensions of both sides of \eqref{orthshow} are equal we
must have equality and not just inclusion.
\end{proof}

There is more we can say about $\vec{E}_j, \vec{F}_j$.  Recalling
\eqref{Lam}, we may write
\begin{equation} \label{EFbreakup}
\begin{aligned} 
\vec{E}_1(z) &= E_1(z_2) \Lambda_n(z_1) & \quad & \vec{F}_1(z) =
F_1(z_2) \Lambda_n(z_1) \\
\vec{E}_2(z) &= E_2(z_2) \Lambda_m(z_2) & \quad & \vec{F}_2(z) =
F_2(z_1) \Lambda_m(z_2) 
\end{aligned}
\end{equation}
for some matrix polynomials $E_1,F_1 \in \C^{n\times n}[z_2]$,
$E_2,F_2 \in \C^{m\times m}[z_1]$.  \textbf{Warning}: $E_1,F_1$ are
functions of $z_2$, while $E_2,F_2$ are functions of $z_1$!

By Proposition \ref{reflectionprop}, 
\begin{equation} \label{matrixreflection}
E_1(z_2) =  z_2^{m} \overline{F_1(1/\bar{z}_2)} X_n \qquad 
E_2(z_1) =  z_1^{n} \overline{F_2(1/\bar{z}_1)} X_m
\end{equation}
where $X_N \in \C^{N\times N}$ is the matrix
\begin{equation} \label{Xmat}
X_N = \begin{pmatrix} 0 & \cdots & 0 & 1 \\
0 & \cdots & 1 & 0 \\
\vdots & \udots & \udots & \vdots \\
1 & 0 & \cdots & 0
\end{pmatrix} = ( \delta_{j,N-k})_{j,k =1,\dots, N}
\end{equation}
which appears due to the fact that $z^{N-1}\Lambda_N(1/z) = X_N
\Lambda_N(z)$. We emphasize that we are taking entrywise complex
conjugates of the above matrices $F_1, F_2$.

\begin{prop} \label{Estable}
With the above definitions, $\det E_1(z_2), \det E_2(z_1)$ are
non-vanishing for $z_1,z_2 \in \D$, while all zeros of $\det F_1(z_2),
\det F_2(z_1)$ are in $\cD$.
\end{prop}

This proposition follows from Proposition 5.3 and Lemma 6.7 of
\cite{KneseAPDE} and \eqref{matrixreflection} above.

In Appendix B, we discuss how $E_1$ and $E_2$ can be constructed using
the one variable matrix Fej\'er-Riesz lemma.

\section{General Agler pairs} \label{secagler}

In this section, we examine how general Agler pairs relate to the
canonical Agler pairs constructed earlier.  Along the way, we relate
the spaces $\mcE_j,\mcF_j$ to certain spaces of one variable vector
valued functions and this combined with Corollary \ref{cordecomp}
produces a list of generators for $\mathcal{I}_p$.

Our first observation is that every Agler decomposition leads to what
is known in systems engineering terminology as a transfer function
representation or realization.  This will show the canonical pairs
$(\vec{E}_1,\vec{F}_2), (\vec{F}_1,\vec{E}_2)$ are minimal in a
certain sense.

\begin{lemma} \label{tfr} Assume $p$ is semi-stable and $\deg p
  =(n,m)$.  Let $(\vec{A}_1, \vec{A}_2)$ be an Agler pair for $p$ with
  $\vec{A}_1 \in \C^{N}[z_1,z_2], \vec{A}_2 \in \C^{M}[z_1,z_2]$.
  Then, there exists a $(1+N+M)\times (1+N+M)$ unitary matrix $U
  = \begin{pmatrix} A & B \\ C & D \end{pmatrix}$, where the block
  decomposition corresponds to the direct sum decomposition
  $\C^{1+N+M} = \C \oplus \C^{N+M}$, such that
\[
\frac{\refl{p}(z)}{p(z)} = A + B \Delta(z)(I-D\Delta(z))^{-1} C
\]
where $\Delta(z) = \begin{pmatrix} z_1 I_N & 0 \\ 0 & z_2
  I_M \end{pmatrix}$.  
\end{lemma}

The proof is a standard argument, so we relegate it to Appendix A for
the curious reader.

\begin{lemma} Let $p$ be semi-stable and $\deg p = (n,m)$.  Let
  $\vec{A}_1 \in \C^N[z_1,z_2]$ and $\vec{A}_2 \in \C^{M}[z_1,z_2]$
  and let $(\vec{A}_1, \vec{A}_2)$ be an Agler pair for $p$. Then
  $N\geq n$, $M\geq m$, $\vec{A}_1$ has bidegree at most $(n-1,m)$,
  and $\vec{A}_2$ has bidegree at most $(n,m-1)$.
\end{lemma}
\begin{proof}
  The bounds $N\geq n, M\geq m$ follow from Lemma \ref{tfr} because
  the numerator and denominator of the transfer function realization
  for $\refl{p}/p$ have bidegree at most $(N,M)$ and since $p$ and
  $\refl{p}$ have no common factors this bidegree bound holds for $p$
  and $\refl{p}$ as well.  The bounds on bidegrees for
  $\vec{A}_1,\vec{A}_2$ follow from Theorem
  2.10 of \cite{KneseRIF}.
\end{proof}

Using the lemma we can write
\begin{equation} \label{AAdecomp}
\vec{A}_1(z) = A_1(z_2) \Lambda_n(z_1) \quad \vec{A}_2(z) = A_2(z_1)
\Lambda_m(z_2)
\end{equation}
for one variable matrix polynomials $A_1\in \C^{N\times n}[z_2],A_2\in
\C^{M\times m}[z_1]$.

One of the main results of \cite{KneseAPDE} was a characterization of
our canonical Agler pairs.  The following is a direct result of
Theorem 1.3 of \cite{KneseAPDE} and Proposition \ref{Estable} above.

\begin{lemma}\label{lemmachar} 
  Assume the setup of the previous lemma. If $A_1(z_2)$ is invertible
  for all $z_2 \in \D$, then $(\vec{A}_1,\vec{A}_2) =
  (\vec{E}_1,\vec{F}_2)$ up to unitary multiplication.  If $A_2(z_1)$
  is invertible for all $z_1 \in \D$, then $(\vec{A}_1,\vec{A}_2) =
  (\vec{F}_1, \vec{E}_2)$ up to unitary multiplication.
\end{lemma}

\begin{lemma} \label{mif}
Assume the setup of the previous lemma.
  If we define $A_1,A_2$ as in \eqref{AAdecomp} then
\begin{equation} \label{isoms}
A_1(z_2) E_1(z_2)^{-1} \qquad A_2(z_1) E_2(z_1)^{-1}
\end{equation}
are holomorphic in $\D$ and extend to be holomorphic and
isometry-valued on $\T$.  

If $N=n$, then $A_1 E_1^{-1}$ is rational inner and $F_1 A_1^{-1}$ is
also a well-defined matrix rational inner function.  Similarly, if
$M=m$, then $A_2E_2^{-1}$ and $F_2 A_2^{-1}$ are
matrix rational inner functions.

In particular, 
\[
\Phi_1(z_2) := F_1(z_2) E_1(z_2)^{-1} \quad \Phi_2(z_1) := F_2(z_1)
E_2(z_1)^{-1}
\]
are both one variable matrix rational inner functions on $\D$.  
\end{lemma}

\begin{proof}
  Since $(\vec{A}_1,\vec{A}_2)$ is an Agler pair,
\begin{multline} \label{Aglerpair1}
\sum_{j=1}^{2} (1-\bar{w}_j z_j)\vec{A}_j(w)^*\vec{A}_j(z) \\=
(1-\bar{w}_1 z_1)\vec{E}_1(w)^* \vec{E}_1(z) + (1-\bar{w}_2
z_2)\vec{F}_2(w)^* \vec{F}_2(z).
\end{multline}
If we set $z_2=w_2 \in \T$ we get $\vec{A}_1(w)^*\vec{A}_1(z) =
\vec{E}_1(w)^* \vec{E}_1(z)$ and if we rewrite in terms of matrices we get
\[
\Lambda_n(w_1)^* E_1(z_2)^*E_1(z_2) \Lambda_n (z_1) = \Lambda_n(w_1)^* A_1(z_2)^*
A_1(z_2) \Lambda_n(z_1)
\]
for $z_1,w_1 \in \C$ and $z_2 \in \T$.  So,
\begin{equation} \label{EA}
E_1(z_2)^*E_1(z_2) = A_1(z_2)^* A_1(z_2)
\end{equation}
for $z_2 \in \T$.  This implies $\Phi(z_2) \defn A_1(z_2) E_1(z_2)^{-1}$ is
isometry-valued on $\T$---in particular any singularities on $\T$ are
removable---and extends to be holomorphic on $\cD$.  An analogous
argument holds for $A_2 E_2^{-1}$.

As discussed in Section \ref{secvec}
\[
\frac{I - \Phi(w_2)^* \Phi(z_2)}{1-\bar{w}_2 z_2} 
\]
is positive semi-definite and after multiplying on the left by
$\Lambda_n(w_1)^*E_1(w_2)^*$ and the right by $E_1(z_2)\Lambda_n(z_1)$
we see that
\[
\frac{ \vec{E}_1(w)^*\vec{E}_1(z) - \vec{A}_1(w)^*
  \vec{A}_1(z)}{1-\bar{w}_2 z_2}
\]
is also positive semi-definite.  
By \eqref{Aglerpair1},
\[
(1-\bar{w}_1 z_1)(\vec{E}_1(w)^*\vec{E}_1(z) - \vec{A}_1(w)^*
  \vec{A}_1(z)) = (1-\bar{w}_2 z_2)(\vec{A}_2(w)^* \vec{A}_2(z) -
  \vec{F}_2(w)^*\vec{F}_2(z)),
\]
and so $(1-\bar{w}_1z_1)$ divides the right hand side and
$(1-\bar{w}_2 z_2)$ divides the left.  So,
\begin{equation} \label{consequence}
\frac{ \vec{E}_1(w)^*\vec{E}_1(z) - \vec{A}_1(w)^*
  \vec{A}_1(z)}{1-\bar{w}_2 z_2} = \frac{\vec{A}_2(w)^* \vec{A}_2(z) -
  \vec{F}_2(w)^*\vec{F}_2(z)}{1-\bar{w}_1 z_1}
\end{equation}
is a positive semi-definite polynomial and similarly so is
\begin{equation} \label{similarly}
\frac{ \vec{E}_2(w)^*\vec{E}_2(z) - \vec{A}_2(w)^*
  \vec{A}_2(z)}{1-\bar{w}_1 z_1} = \frac{\vec{A}_1(w)^* \vec{A}_1(z) -
  \vec{F}_1(w)^*\vec{F}_1(z)}{1-\bar{w}_2 z_2}.
\end{equation}

Now assume $N=n$.  Then, $A_1$ is square and hence $A_1E_1^{-1}$ will
be rational inner.  We now show $\Psi(z_2) \defn F_1(z_2)
A_1(z_2)^{-1}$ is a matrix rational inner function.  By Proposition
\ref{kernelprop} we can factor \eqref{similarly} as
$\vec{H}(w)^*\vec{H}(z)$ for some vector polynomial $\vec{H}\in
\C^{K}[z_1,z_2]$ and then
\[
\vec{A}_1(w)^* \vec{A}_1(z) + \bar{w}_2z_2 \vec{H}(w)^* \vec{H}(z) =
\vec{F}_1(w)^*\vec{F}_1(z) + \vec{H}(w)^* \vec{H}(z).
\]
By Proposition \ref{isomprop}, there exists an $(n+K)\times (n+K)$
unitary $\begin{pmatrix} U_{11} & U_{12} \\ U_{21} &
  U_{22} \end{pmatrix}$ such that
\[
\begin{pmatrix} \vec{F}_1(z) \\
  \vec{H}(z) \end{pmatrix} = \begin{pmatrix} U_{11} & U_{12} \\ U_{21}
  & U_{22} \end{pmatrix} \begin{pmatrix} \vec{A}_1(z) \\ z_2
  \vec{H}(z) \end{pmatrix} = \begin{pmatrix} U_{11} \vec{A}_1(z) +
  U_{12} z_2\vec{H}(z) \\ U_{21} \vec{A}_1(z) + U_{22} z_2
  \vec{H}(z) \end{pmatrix}.
\]
Solve for $\vec{H}$ using the second component and insert the result
into the first component to get $\Xi(z_2) \vec{A}_1(z) = \vec{F}_1(z)$
where
\begin{equation} \label{Xi}
\Xi(z_2) = U_{11} + z_2 U_{12} (I-z_2 U_{22})^{-1}U_{21}.
\end{equation}
In terms of matrices $\Xi(z_2) A_1(z_2) = F_1(z_2)$, and since $\det
A_1$ is not identically zero by \eqref{EA} we see that $\Xi = \Psi$.
In particular, $\Psi$ is analytic in $\D$.  By Lemma \ref{lurkisom}, 
$\Xi$ is a matrix-valued rational inner function.
\end{proof}

The following is an important corollary of the proof of Lemma
\ref{mif}.  This was first proven in \cite{KneseAPDE}.

\begin{corollary} \label{corunique} Let $p$ be semi-stable. Then, $p$
  has a unique Agler pair (up to unitary multiplication) iff $\mcE_j =
  \mcF_j$ for $j=1$ or $2$ iff $\mcG = \{0\}$.
\end{corollary}

\begin{proof} If $p$ has a unique Agler pair, then
  $\vec{E}_j,\vec{F}_j$ are unitary multiples of each other for
  $j=1,2$ by Theorem \ref{thmsos}.  This is equivalent to equality of
  the spaces $\mcE_j=\mcF_j$.  By the last equation of Theorem
  \ref{thmsos}, this is equivalent to $\vec{G} = 0$ as well as
  $\mcG=\{0\}$.  Also, by this equation $\mcE_1=\mcF_1$ iff
  $\mcE_2=\mcF_2$.

Finally, if $\vec{E}_j,\vec{F}_j$ are unitary multiples, then the
positive semidefinite expressions in \eqref{consequence} and
\eqref{similarly} must equal zero, meaning $|\vec{A}_j|^2 =
|\vec{E}_j|^2=|\vec{F}_j|^2$ for $j=1,2$.  Hence, Agler pairs are
unique in this case.
\end{proof}

Our next goal is to show that $\mcG$ can be viewed as two different
one variable vector valued Hardy spaces.

\begin{lemma} \label{Hardylemma} Using the notation of Lemma
  \ref{mif}, given $f\in \mcG$ we may write
\[
f(z) = \vec{f}_1(z_2)
  \Lambda_n(z_1) = \vec{f}_2(z_1) \Lambda_m(z_2)
\]
for $\vec{f}_1 \in \C^{1\times n}[z_2], \vec{f}_2 \in \C^{1\times
  m}[z_1]$.  The map $f \mapsto \vec{f}_1 E_1^{-1}$ is a unitary from
$\mcG$ onto $\Hrow \ominus \Hrow \Phi_1$ and the map $f \mapsto
\vec{f}_2 E_2^{-1}$ is a unitary from $\mcG$ onto $H_{1\times m}^2
\ominus H_{1\times m}^2 \Phi_2$.

The inverses of the maps are given by $\vec{f} \mapsto \vec{f}
\vec{E}_j$.  Consequently, any $f \in \mcG$ is of the
form $\vec{f}_1(z_2)\vec{E}_1(z)=\vec{f}_2(z_1)\vec{E}_2(z)$ where
$\vec{f}_1(z_2), \vec{f}_2(z_1)$ are rational
functions with no poles in $\cD$
\end{lemma}

The space $\Hrow$ above is row vector $\C^{1\times n}$-valued Hardy space on
$\D$ as explained in Section \ref{secvec}.  

\begin{proof}
  The space $\mcG$ is a reproducing kernel Hilbert space with
  reproducing kernel
\[
k_{w}(z) = \frac{\vec{E}_1(w)^*\vec{E}_1(z) - \vec{F}_1(w)^*
  \vec{F}_1(z)}{1-\bar{w}_2 z_2} = \vec{E}_1(w)^* \frac{I -
  \Phi_1(w_2)^* \Phi_1(z_2)}{1-\bar{w}_2 z_2} \vec{E}_1(z)
\]
by Theorem \ref{thmsos}.
The space $\Hrow \ominus \Hrow \Phi_1$ is a vector-valued reproducing
kernel Hilbert space with reproducing kernel
\[
K_{w_2}(z_2) = \frac{I-\Phi_1(w_2)^* \Phi_1(z_2)}{1-\bar{w}_2 z_2}
\]
by Proposition \ref{rowhardy}.

Let $T$ be the map in question: $T f= \vec{f}_1 E_1^{-1}$. Then,
\[
T k_w = \vec{E}_1(w)^*K_{w_2}.
\]
Since reproducing kernels span $\mcG$, this shows that $\mcG$ maps
into $\Hrow \ominus \Hrow\Phi_1$.  The map $T$ is onto because if $\vec{f}
\in \Hrow \ominus \Hrow\Phi_1$ is orthogonal to the range of $T$, then
\[
\ip{\vec{f}}{\vec{E}_1(w)^*K_{w_2}}_{\Hrow} = \vec{f}(w_2)\vec{E}_1(w)
=0 \text{  for all } w\in \D^2.
\]
But, this implies $\vec{f}(w_2)E_1(w_2)\Lambda_n(w_1) \equiv 0$, which
implies $\vec{f}(w_2)E_1(w_2)\equiv 0$, which implies $\vec{f} =0$
since $E_1$ is invertible except at finitely many points.

The map is a unitary because
\[
\ip{\vec{E}_1(w)^*K_{w_2}}{\vec{E}_1(z)^* K_{z_2}}_{H_{1\times n}^2} =
\vec{E}_1(w)^*K_{w_2}(z_2)\vec{E}_1(z) = k_{w}(z) = \ip{k_{w}}{k_{z}}_{\mcG}
\]
is enough to show $T$ is isometric on linear combinations of
reproducing kernels.

It is clear that the inverse of $T$ is given by $\vec{f}(z_2) \mapsto
\vec{f}(z_2)\vec{E}_1(z)$.  Proposition \ref{rowhardy} states that
$\vec{f}(z_2)$ is rational with no poles in $\cD$.
\end{proof}

\section{Generators for $\mathcal{I}_p$ and Theorem
  \ref{intthmineq}} \label{secthma}

Question \ref{q2} from the introduction asks for a list of generators
of $\mathcal{I}_p$.  We could settle for saying $\mathcal{I}_p$ is
generated by bases for $\mcG,\mcF_1$, and $\mcF_2$ as well as
$\refl{p}$ by Corollary \ref{cordecomp}.  Since the space $\mcG$ is in
some ways more elusive (e.g. a main theorem of this paper is a formula
for $\dim \mcG$), it is worth pointing out that we can replace $\mcG$
with $\mcE_1$.  We can also remove $\refl{p}$ from the list.

\begin{theorem} \label{generators} Let $p\in \C[z_1,z_2]$ be
  semi-stable. The ideal $\mathcal{I}_p$ is
  generated by the entries of $\vec{E}_1, \vec{F}_1$, and $\vec{F}_2$.
\end{theorem}

\begin{proof} By Lemma \ref{Hardylemma}, if $f \in \mcG$, then $f =
  \vec{f}(z_2) \vec{E}_1(z)$ where $\vec{f}(z_2)$ is a $\C^{1\times
    n}$ valued rational function with no poles in $\cD$.  So, there
  exists a polynomial $g(z_2)$ with no zeros in $\cD$ such that
  $g(z_2)f(z)$ is a multiple of $\vec{E}_1$.  The same argument
  applies to $\refl{f}(z)= z_1^{n-1}z_2^{m-1}
  \overline{f(1/\bar{z}_1,1/\bar{z}_2)}$ so that there is an $h(z_2)
  \in \C[z_2]$ with no zeros in $\cD$ such that $h(z_2) \refl{f}(z)$
  is a multiple of $\vec{E}_1$.  If we reflect this at an appropriate
  degree, we get $\refl{h}(z_2) f(z)$ is a multiple of $\vec{F}_1$.
  Note that $\refl{h}$ will have all zeros in $\D$.  Since $g$ and
  $\refl{h}$ have no common zeros, there exist $A(z_2),B(z_2)$ such
  that $Ag+B\refl{h} = 1$.  Thus, $(Ag+B\refl{h})f = f$ is in the
  ideal generated by the entries of $\vec{F}_1,\vec{E}_1$.

By Corollary \ref{cordecomp}, any $f \in \mathcal{I}_p$ can be written
as a combination of polynomial multiples of $\vec{F}_1,\vec{F}_2$ and
$\refl{p}$ and an element of $\mcG$.  Thus, the entries of
$\vec{E}_1,\vec{F}_1,\vec{F}_2$ and $\refl{p}$ generate the ideal
$\mathcal{I}_p$.

Finally, $\refl{p}$ is in the ideal generated by $\mcE_1,\mcF_2$.  To
see this simply let $w$ be a zero of $p$ which is not a zero of
$\refl{p}$ and insert this value for $w$ into Theorem \ref{thmsos}.
This immediately exhibits $\refl{p}$ as a combination of elements of
$\mcE_1,\mcF_2$.  
\end{proof}

\begin{remark} \label{constructgenerators}
If one wants to construct generators of $\mathcal{I}_p$ it is only
necessary to construct $\vec{E}_1$: $\vec{F}_1$ is just a reflection
of $\vec{E}_1$ (as in Proposition \ref{reflectionprop}) and
$\vec{F}_2(w)^*\vec{F}_2(z)$ can then be solved for using Theorem
\ref{thmsos}.  Once $\vec{F}_2(w)^*\vec{F}_2(z)$ is known, we can
extract coefficients of powers of $z_2,\bar{w}_2$ to get
$F_2(w_1)^*F_2(z_1)$ and if we further extract coefficients of powers
of $z_1,\bar{w}_1$ we can write
\[
F_2(w_1)^*F_2(z_1) = (I_m,\bar{w}_1 I_m,\dots, \bar{w}_1^{n-1} I_m)
H \begin{pmatrix} I_m \\ z_1 I_m \\ \vdots \\ z_1^{n-1}
  I_m \end{pmatrix}
\]
for some $nm\times nm$ positive semi-definite matrix $H$.  We can
factor $H = J^*J$ for some $m\times nm$ matrix $J$ since $H$
necessarily has rank $m$.  Then, $F_2(z_1) = J (I_m, z_1I_m,\dots,
z_1^{n-1} I_m)^t$.  This is the approach taken in Example
\ref{pascoeex}.
\end{remark}

Notice that the common zeros of $\mathcal{I}_p$ are all on $\T^2$ as
one would expect.  This is because $\vec{F}_1$ has no zeros in
$\C\times (\C\setminus \cD)$ and $\vec{E}_1$ has no zeros in $\C\times
\D$ by Lemma \ref{Estable} and this leaves any common zeros in
$\C\times \T$.  By symmetry any common zeros must also be in $\T\times
\C$ and this leaves $\T^2$.

Before we prove Theorem \ref{intthmineq} from the introduction as
Corollary \ref{intthmcor} below, we need the following fact.

\begin{prop} \label{EonT} Suppose $p$ is semi-stable and $\deg p
  =(n,m)$.  Let $(\vec{A}_1,\vec{A}_2)$ be an Agler pair for
  $p$. Then, for
  $z \in \T^2$
\[
|\vec{A}_1(z)|^2 = n|p(z)|^2 - 2\Re(\overline{p(z)}(z_1 \partial_1 p(z)))
\]
\[
|\vec{A}_2(z)|^2 = m|p(z)|^2 - 2\Re(\overline{p(z)}(z_2\partial_2
p(z))).
\]
\end{prop}

\begin{proof}
We can compute
  $|\vec{A}_1(z)|^2$ for $z \in \T^2$ directly as follows.

For $z_2 \in \T$, 
\[
|\vec{A}_1(z)|^2 = \frac{|p(z)|^2 - |\refl{p}(z)|^2}{1-|z_1|^2}
\]
and for $r \in (0,1)$ and $z_1 \in \T$
\[
|\vec{A}_1(rz_1,z_2)|^2 = \frac{|p(rz_1,z_2)|^2 -
  r^{2n}|p(z_1/r,z_2)|^2}{1-r^2}.
\]
Letting $r\to 1$ we get for $z \in \T^2$
\[\begin{aligned}
|\vec{A}_1(z)|^2 &= \frac{4\Re(\overline{p(z)}(\partial_1 p(z)z_1)) -
  2n |p(z)|^2}{-2}\\
& = n|p(z)|^2 - 2\Re(\overline{p(z)}(z_1 \partial_1 p(z))).
\end{aligned}\]
The proof for $\vec{A}_2$ is similar.
\end{proof}

\begin{corollary} \label{intthmcor} Suppose $p \in \C[z_1,z_2]$ is
  semi-stable and $\deg p= (n,m)$.  Let $f\in \C[z_1,z_2]$. Then, $f
  \in \mathcal{I}_p$ iff there is a constant $c>0$ such that for $z
  \in \T^2$
\[
|f(z)|^2 \leq c( (n+m)|p(z)|^2 - 2\Re(\overline{p(z)}(z_1\partial_1 p
+ z_2 \partial_2p)) ).
\]
\end{corollary}

\begin{proof}
By Theorem \ref{generators} any $f$ can be written in terms of
$\vec{E}_1,\vec{F}_1,\vec{F}_2$.  Since
$|\vec{E}_1|=|\vec{F}_2|$ on $\T^2$, it follows
that on $\T^2$ we have
\[
|f| \leq c( |\vec{E}_1|+|\vec{F}_2|)
\]
Since $(\vec{E}_1,\vec{F}_2)$ is an Agler pair, Proposition \ref{EonT}
gives the estimate as claimed after applying Cauchy-Schwarz.

On the other hand, if the inequality holds then $f$ is bounded by
elements of $\Lp$ and therefore must belong to
this space.
\end{proof}

\section{The ideal $\mathcal{I}_{p}^{\infty}$} \label{secinf}

Recall $\mathcal{I}_{p}^{\infty}$ is the the set of $q$ such that
$q/p$ is essentially bounded on $\T^2$.  One way for $q/p$ to be
bounded on $\T^2$ is if $q \in \langle p, \refl{p}\rangle$, the ideal
generated by $p,\refl{p}$.  However this cannot be all elements of
$\mathcal{I}_{p}^{\infty}$ since elements of $\langle p,\refl{p}
\rangle$ vanish at all common zeros of $p$ and $\refl{p}$, which could
include zeros not on $\T^2$ and these should not affect boundedness of
$q/p$ on $\T^2$.  It turns out that we can explicitly construct one
variable polynomials $g(z_1),h(z_2)$ such that $gh \mathcal{I}_p
\subset \mathcal{I}_p^{\infty}$.  It is not surprising that such
polynomials exist but the actual choice of $g,h$ may be of some
interest. 

If $g \in \C[z_1]$ is a one variable polynomial, we can factor $g =
g_1 g_2$ where $g_1$ has no zeros on $\T$ and $g_2$ has all of its
zeros on $\T$.  We will refer to $g_2$ as the $\T$-factor of $g$.  The
$\T$-factor is unique up to constant multiples.  

The following theorem identifies a large subset of
$\mathcal{I}_p^{\infty}$.  We leave the search for a complete
characterization of $\mathcal{I}_p^{\infty}$ for future work.

\begin{theorem} \label{infinitythm} Let $p\in \C[z_1,z_2]$ be
  semi-stable, $\deg p = (n,m)$, and we will use $E_1,E_2$ as defined
  in \eqref{EFbreakup}.  Let $g\in \C[z_2]$ be the $\T$-factor of
  $\det E_1(z_2)$ and let $h \in \C[z_1]$ be the $\T$-factor of $\det
  E_2(z_1)$.  Then, $g(z_2)h(z_1) \mathcal{I}_{p} \subset
  \mathcal{I}_p^{\infty}$.
\end{theorem}

\begin{lemma} \label{Llemma} For an arbitrary $p \in \C[z_1,z_2]$ of
  degree $(n,m)$, let
\[
L_{w_1}(z) = L (z_1,z_2;w_1) = z_2^{m} \frac{p(z) \overline{p(w_1, 1/\bar{z}_2)}
  - \refl{p}(z) \overline{\refl{p}(w_1, 1/\bar{z}_2)} }{1-z_1
  \bar{w}_1}.
\]
Then, $L_{w_1} \in \langle p,\refl{p}\rangle$ for each
$w_1\in \C$.
\end{lemma}

\begin{proof}
Observe
\[
L(z;w_1) = p(z) \frac{z_2^m \overline{p(w_1,1/\bar{z}_2)} - \bar{w}_1^n
  \refl{p}(z)}{1-z_1\bar{w}_1} + \refl{p}(z) \frac{\bar{w}_1^n p(z) - z_2^m
  \overline{\refl{p}(w_1,1/\bar{z}_2)}}{1-z_1\bar{w}_1}.
\]
The denominator $(1-z_1\bar{w}_1)$ divides the numerator in both
fractions above.
\end{proof}

\begin{lemma} \label{Llemma2} With $L$ defined as Lemma \ref{Llemma}, if $p$ is
  semi-stable then
\[
L(z;w_1) = \Lambda_n(w_1)^* X_n F_1(z_2)^t \vec{E}_1(z) =
\Lambda_n(w_1)^* X_n E_1(z_2)^t \vec{F}_1(z).
\]
Recall $X_n ,E_1,F_1$ from \eqref{EFbreakup} and
\eqref{matrixreflection}.  
\end{lemma}

This is just a result of setting $w_2 = 1/\bar{z}_2$ and multiplying
through by $z_2^m$ in Theorem \ref{thmsos}.

\begin{prop} Assume the setup of Theorem \ref{infinitythm}.  Then, the
  entries of $g\vec{E}_1, g\vec{F}_1, h \vec{E}_2, h \vec{F}_2$ belong
  to $\mathcal{I}_{p}^{\infty}$.
\end{prop}

\begin{proof}
  By Lemmas \ref{Llemma}, \ref{Llemma2}, the entries of $F_1(z_2)^t
  \vec{E}_1(z)$ belong to the ideal $\langle p, \refl{p} \rangle$.
  After multiplying by the adjugate matrix of $F_1^t$, we see that the
  entries of $\det F_1(z_2) \vec{E}_1(z)$ belong to $\langle
  p,\refl{p}\rangle$.  By \eqref{matrixreflection}, $\det F_1(z_2)$,
  $\det E_1(z_2)$ have the same zeros with the same multiplicities on
  $\T$.  Thus, if we divide out the factor of $\det F_1(z_2)$
  containing all zeros not on $\T$ we are left with a multiple of $g$.
  Therefore, the entries of $g \vec{E}_1(z)$ belong to
  $\mathcal{I}_p^{\infty}$.  By a similar argument we form the same
  conclusion for the entries of $g \vec{F}_1 h\vec{E}_2, h \vec{F}_2$.
\end{proof}

\begin{proof}[Proof of Theorem \ref{infinitythm}]
  Since any $f \in \mathcal{I}_p$ can be written as a combination of
  $\vec{E}_1,\vec{F}_1,\vec{F}_2$ (Theorem \ref{generators}) it
  follows that $ghf \in \mathcal{I}_{p}^{\infty}$.
\end{proof}

\section{A commuting pair of contractive matrices} \label{seccomm} We
now begin to study Question/Theorem \ref{q3} which asks for an exact
count of the dimension of $\mcP_{j,k}$.  This is accomplished by
finding a pair of commuting contractive operators on $\mcG$ whose
joint eigenvalues are directly related to common zeros of $p$ and
$\refl{p}$.

Let $p \in \C[z_1,z_2]$ be semi-stable, $\deg p = (n,m)$, and refer to
Notation \ref{not:spaces}.  Let $P$ be the orthogonal projection onto
$\mcG$ in $\Lp$. Define $T_j:\mcG \to \mcG$ by
\[
T_jf = P z_j f.
\]
Our goal in this section is to show $T_1$ and
$T_2^*$ commute and the joint invariant subspaces are directly related
to minimal Agler decompositions.

The following is proven in \cite{BSV} in a more general set-up, but
even if we directly applied their theorem here it would still take
work to get to this level of specificity.  The result is found in the
case of $p$ with no zeros on $\cD^2$ in \cite{GW}.

\begin{theorem} With $T_1,T_2$ defined as above, the operators $T_1$
  and $T_2^*$ commute.
\end{theorem}

\begin{proof}
The condition $T_1T_2^* - T_2^* T_1 = 0$ means
\[
P z_1 P \bar{z}_2 f - P \bar{z}_2 P z_1 f = 0
\]
for all $f \in \mcG$.  This is equivalent to 
\[
z_1 P \bar{z}_2 f - \bar{z}_2 P z_1 f \perp \mcG
\]
for all $f \in \mcG$, which is equivalent to 
\[
z_1 z_2 P \bar{z}_2 \bar{z}_1 f - Pf \perp z_2 \mcG
\]
for all $f \in z_1 \mcG$.  Let $P_1$ denote orthogonal projection onto
$\mcP_{n,m}$; let $P_{p}, P_{\refl{p}}$ denote orthogonal projection
onto $\C p, \C \refl{p}$ respectively; let $P_{\mcH}$ denote
orthogonal projection onto a subspace $\mcH$.

By equations \eqref{Pnm1} and \eqref{Pnm2},
\[
P_1 = P +  P_{\mcF_1} + P_{\mcF_2}+ P_{\refl{p}} 
= z_1z_2 P \bar{z}_1 \bar{z}_2 + P_{z_1\mcE_1} + P_{z_2\mcE_2} + P_p
\]
and so
\[
z_1 z_2 P\bar{z}_1 \bar{z}_2 - P =   P_{\mcF_1} +P_{\mcF_2} +
P_{\refl{p}} -(P_{p} + P_{z_1\mcE_1} + P_{z_2\mcE_2}).
\]
Then for $g \in z_2\mcG$ and $f \in z_1\mcG$ we
have
\[
\ip{( P_{\mcF_1} +P_{\mcF_2} +
P_{\refl{p}} -(P_{p} + P_{z_1\mcE_1} + P_{z_2\mcE_2})  )f}{g} = 0
\]
since $z_1\mcE_1 \perp z_2 \mcG$, $z_2\mcE_2 \perp
z_1\mcG$, $\mcF_1 \perp z_1\mcG$, $\mcF_2 \perp
z_2\mcG$, and $p,\refl{p} \perp f,g$.  All of this follows from
Theorem \ref{thmorth}.
\end{proof}

\begin{lemma} \label{Glemma}
Suppose $\mcG = \mcG_1 \oplus \mcG_2$ for some subspaces
$\mcG_1,\mcG_2$.  The following are equivalent
\begin{itemize}
\item $\mcG_1$ is an invariant subspace of $T_1^*$ 
\item $\mcG_1 \subset z_1\mcG_1 \oplus \mcE_2$ 
\item  $\mcG_2$ is an invariant subspace of $T_1$
\item $z_1\mcG_2 \subset \mcG_2 \oplus \mcF_2$.
\end{itemize}
Similarly, the following are equivalent 
\begin{itemize}
\item $\mcG_2$ is an invariant subspace of $T_2^*$
\item $\mcG_2 \subset z_2 \mcG_2 \oplus \mcE_1$ 
\item $\mcG_1$ is an invariant subspace of $T_2$ 
\item $z_2 \mcG_1 \subset \mcG_1 \oplus \mcF_1$.
\end{itemize}

Thus, if $\mcG_1$ is an invariant subspace of $(T_1^*, T_2)$, then the
subspaces
\begin{equation} \label{AA}
\mcA_1 =( z_2 \mcG_2 \oplus \mcE_1) \ominus \mcG_2 \text{ and } \mcA_2 =
(z_1 \mcG_1 \oplus \mcE_2) \ominus \mcG_1
\end{equation}
are well-defined and $\dim \mcA_1 = n, \dim \mcA_2 = m$.
\end{lemma}

\begin{proof}
  Suppose $\mcG_1$ is invariant under $T_1^*$.  For any $f \in \mcG$
  we can write $f = z_1 g + h$ where $g \in \mcG$ and $h \in \mcE_2$
  since $\mcP_{n,m-1} = z_1\mcG \oplus \mcE_2$.  If this $f$ is
  actually in $\mcG_1$ then $T_1^* f = P\bar{z}_1 f = g + P\bar{z}_1 h
  =g$ since $\mcE_2 \perp z_1 \mcG$.  By invariance, $g \in \mcG_1$ so
  that $\mcG_1 \subset z_1\mcG_1 \oplus \mcE_2$.  Conversely, if
  $\mcG_1 \subset z_1\mcG_1 \oplus \mcE_1$, then $\bar{z}_2 \mcG_1
  \subset \mcG_1 \oplus \bar{z}_2\mcE_1$ and so $T_2^* \mcG_1 \subset
  \mcG_1$ since $\mcE_1 \perp z_2\mcG$.

  By properties of adjoints, $\mcG_1$ is invariant for $T_1^*$ iff
  $\mcG_2$ is invariant for $T_1$.  If $\mcG_2$ is invariant for
  $T_1$, then for any $f\in \mcG_2$ we can write $z_1f = g + h$
  where $g\in \mcG_2, h \in \mcF_2$.  Thus, $z_1\mcG_2 \subset
  \mcG_2\oplus \mcF_2$.  The converse is similar.  

The claims about  $T_2,T_2^*$ are similar to those for $T_1,T_1^*$.

Finally, when $\mcG_1$ is invariant under both $T_1^*, T_2$, then
$\mcA_1$ and $\mcA_2$ are well-defined by the inclusions above and the
statement about dimensions follows from the fact that $z_j\mcG_j$ has
the same dimension as $\mcG_j$.  
\end{proof}

The following result is closely related to a result in \cite{BSV}, but
again they work in higher generality and it takes additional arguments
to get to these finite dimensional statements.  We emphasize that the
point of the theorem is that an Agler pair $(\vec{A}_1,\vec{A}_2)$
corresponds to an invariant subspace of $(T_1^*,T_2)$ exactly when the
dimensions of $\vec{A}_1$ and $\vec{A}_2$ match the bidegree of $p$,
namely $(n,m)$.

\begin{theorem} Let $p\in \C[z_1,z_2]$ be semi-stable, $\deg p =
  (n,m)$ and define $T_1,T_2$ as above.  Let $\mcG_1$ be an invariant
  subspace of $T_2,T_1^*$ and let $\mcG_2 = \mcG \ominus \mcG_1$ which
  will be invariant under $T_1,T_2^*$.  Let $\vec{A}_1\in
  \C^n[z_1,z_2], \vec{A}_2\in \C^m[z_1,z_2]$ be vector polynomials
  whose entries form an orthonormal basis for $\mcA_1,\mcA_2$ from
  \eqref{AA}.  Then, the pair $(\vec{A}_1, \vec{A}_2)$ is an Agler
  pair.

  Conversely, suppose $(\vec{A}_1,\vec{A}_2)$ is an Agler pair where
  either $\vec{A}_1 \in \C^n[z_1,z_2]$ or $\vec{A}_2 \in
  \C^m[z_1,z_2]$.  We assume the entries of $\vec{A}_j$ are linearly
  independent.  Then, there exists an invariant subspace $\mcG_1$ of
  $(T_1^*, T_2)$ such that the entries of $\vec{A}_j$ form an
  orthonormal basis for $\mcA_j$ as defined in \eqref{AA} for $j=1,2$.
\end{theorem}


\begin{proof} $(\Rightarrow)$
  Let $\vec{G}_j$ be a vector polynomial whose entries form an
  orthonormal basis for $\mcG_j$. Then, since $\mcA_2 \oplus \mcG_1 =
  \mcE_2 \oplus z_1\mcG_1$, we have $|\vec{A}_2|^2 + |\vec{G}_1|^2 =
  |\vec{E}_2|^2 + |z_1|^2 |\vec{G}_1|^2$ and similarly $|\vec{A}_1|^2+
  |\vec{G}_2|^2 = |\vec{E}_1|^2 + |z_2|^2|\vec{G}_2|^2$.  But,
  $|\vec{E}_2|^2 - |\vec{F}_2|^2 = (1-|z_1|^2)|\vec{G}|^2 =
  (1-|z_1|^2) (|\vec{G}_1|^2+|\vec{G}_2|^2) $.

Then, 
\[
\begin{aligned}
\sum_{j=1}^{2}(1-|z_j|^2) |\vec{A}_j|^2 =& (1-|z_1|^2) (|\vec{E}_1|^2 -
(1-|z_2|^2)|\vec{G}_2|^2) \\ &+ (1-|z_2|^2) (|\vec{E}_2|^2  -(1-|z_1|^2) |\vec{G}_1|^2) \\
=& (1-|z_1|^2) |\vec{E}_1|^2 + (1-|z_2|^2)|\vec{F}_2|^2
\end{aligned}
\]
which shows $(\vec{A}_1,\vec{A}_2)$ is an Agler pair.

$(\Leftarrow)$ Suppose $(\vec{A}_1,\vec{A}_2)$ is an Agler pair with
$\vec{A}_1 \in \C^n[z_1,z_2]$.  By Lemma \ref{mif}, $\Phi_1 = F_1
E_1^{-1}, \Phi = A_1 E_1^{-1}, \Psi = F_1 A_1^{-1}$ are all $n\times
n$ matrix inner functions where $\Phi_1 = \Psi \Phi$.

The space $\Hrow \ominus \Hrow \Phi_1$ therefore possesses the orthogonal
decomposition
\[
\Hrow \ominus \Hrow \Phi_1 = (H_{1\times
  n}^2\ominus \Hrow \Phi) \oplus
(\Hrow\ominus \Hrow \Psi) \Phi.
\]
By Lemma \ref{Hardylemma}, we have $\mcG = \mcG_2 \oplus \mcG_1$ where
$\mcG_2 = (\Hrow\ominus \Hrow \Phi)\vec{E}_1$ and $\mcG_1 =
(\Hrow\ominus \Hrow\Psi)\Phi \vec{E}_1= (\Hrow \ominus \Hrow \Psi)
\vec{A}_1$.  The kernel in \eqref{consequence}, namely
\[
k_w(z) \defn \frac{ \vec{E}_1(w)^*\vec{E}_1(z) - \vec{A}_1(w)^*
  \vec{A}_1(z)}{1-\bar{w}_2 z_2} = \frac{\vec{A}_2(w)^* \vec{A}_2(z) -
  \vec{F}_2(w)^*\vec{F}_2(z)}{1-\bar{w}_1 z_1},
\]
is the reproducing kernel for $\mcG_2$  because
\[
k_w(z) = \vec{E}_1(w)^* \underset{K_{w_2}(z_2)}{\underbrace{\frac{ I -\Phi(w_2)^*
  \Phi(z_2)}{1-\bar{w}_2 z_2}}} \vec{E}_1(z)
\]
which means that for $f(z) = \vec{f}(z_2)\vec{E}(z)$ where $\vec{f}
\in \Hrow \ominus \Hrow \Phi$ we have
\[
\ip{f}{k_w}_{\mcG} = \ip{\vec{f}}{\vec{E}_1(w)^*K_{w_2}}_{\Hrow} =
\vec{f}(w_2) \vec{E}_1(w) = f(w)
\]
because of Lemma \ref{Hardylemma}.  By Theorem \ref{thmsos}, the
reproducing kernel for $\mcG$ is 
\[
\frac{ \vec{E}_1(w)^*\vec{E}_1(z) - \vec{F}_1(w)^*
  \vec{F}_1(z)}{1-\bar{w}_2 z_2}
\]
and therefore the reproducing kernel for $\mcG_1= \mcG\ominus \mcG_2$
is the above kernel minus $k_w(z)$ which is just
\[
j_w(z) \defn \frac{ \vec{A}_1(w)^*\vec{A}_1(z) - \vec{F}_1(w)^*
  \vec{F}_1(z)}{1-\bar{w}_2 z_2} = \frac{ \vec{E}_2(w)^*\vec{E}_2(z) - \vec{A}_2(w)^*
  \vec{A}_2(z)}{1-\bar{w}_1 z_1}.
\]
The second equality is \eqref{similarly}, which holds for Agler pairs
in general.

As in Section \ref{secvec}, we can factor the reproducing kernel for
$\mcG_2$ as $\vec{G}_2(w)^* \vec{G}_2(z)$ for some vector polynomial
$\vec{G}_2$.  So, $k_w(z) = \vec{G}_2(w)^*\vec{G}_2(z)$.  The
reproducing kernel (on the diagonal $z=w$) for $\mcE_1 \oplus z_2
\mcG_2$ is thus $|\vec{E}_1|^2 + |z_2|^2|\vec{G}_2|^2$ which equals
$|\vec{A}_1|^2+|\vec{G}_2|^2$ since $|\vec{G}_2|^2 = k_z(z)$.  This
shows $\mcE_1\oplus z_2\mcG_2$
contains $\mcG_2$, because we can relate $\begin{pmatrix} \vec{E}_1 \\
  z_2\vec{G}_2 \end{pmatrix}$ and $\begin{pmatrix} \vec{A}_1 \\
  \vec{G}_2 \end{pmatrix}$ by a unitary matrix, which means
$\vec{G}_2$ can be given directly as a combination of
$\vec{E}_1,z_2\vec{G}_2$.

Thus, $\mcG_2$ is invariant under $T_2^*$ by Lemma \ref{Glemma} and
$\mcA_1 := (\mcE_1 \oplus z_2\mcG_2) \ominus \mcG_2$ is well-defined
and $n$-dimensional.  The vector polynomial $\vec{A}_1$ will be a
unitary multiple of a vector polynomial consisting of an orthonormal
basis for $\mcA_1$ (by Section \ref{secvec}), which means the entries
of $\vec{A}_1$ also form an orthonormal basis for $\mcA_1$.

Since $k_z(z) = |\vec{G}_2|^2$, we have $|\vec{F}_2|^2 + |\vec{G}_2|^2
= |\vec{A}_2|^2 + |z_1|^2 |\vec{G}_2|^2$ and this shows $z_1\mcG_2
\subset \mcG_2\oplus \mcF_2$. By Lemma \ref{Glemma}, $\mcG_2$ is an
invariant subspace of $T_1$.  Thus, $\mcG_2$ is invariant under $T_1,
T_2^*$ and so $\mcG_1$ is invariant under $T_1^*, T_2$ and the spaces
in \eqref{AA} are well-defined.  

The formula $j_z(z) +|\vec{A}_2|^2 = |\vec{E}_2|^2+ |z_1|^2 j_z(z)$
shows that $|\vec{A}_2|^2$ is the reproducing kernel for
$\mcA_2$. Since the entries of $\vec{A}_2$ are assumed to be linearly
independent, it follows $\vec{A}_2$ is a unitary multiple of a vector
consisting of an orthonormal basis of $\mcA_2$.  
\end{proof}

\section{Common zeros of $p$ and $\refl{p}$ as joint eigenvalues}
The pair of commuting contractions from the previous section can be
used to count common zeros of $p$ and $\refl{p}$ in certain regions,
since as we show below the joint eigenvalues of $T_1,T_{2}^*$ are a
simple transformation of the common zeros of $p$ and $\refl{p}$.  We
also show that $T_1, T^*_2$ \emph{dilate} to multiplication operators
$M_{z_1},M_{\bar{z}_2}$ on $\Lp$.  

As a side note, the common zeros of $p$ and $\refl{p}$ are called
\emph{intersecting zeros} of $p$ in the paper \cite{GWzeros}, where
they discuss the interesting problem of how to construct a stable
polynomial (no zeros in $\cD^2$) with given intersecting zeros.  It
would be interesting to pursue their work in the case of semi-stable
$p$.

Let $\C_{\infty} = \C \cup \{\infty\}$ denote the Riemann sphere and
define $\D^{-1} := \{z\in \C: |z|>1\} \cup \{\infty\} \subset
\C_{\infty}$.  If $p \in \C[z_1,z_2]$ has bidegree $(n,m)$, we
interpret $p(a,\infty) = 0$ to mean $q(z_2) := z_2^m p(a,1/z_2)$
vanishes at $z_2=0$, or equivalently, $p(a,\cdot)$ has degree less
than $m$.  We interpret $p(\infty,\infty) = 0$ to mean
$z_1^nz_2^mp(1/z_1,1/z_2)$ vanishes at $(0,0)$.

\begin{lemma} \label{commonzeros} Suppose $p \in \C[z_1,z_2]$ is
  semi-stable.  Then, all common zeros of $p$ and $\refl{p}$ lie in
  $(\D\times \D^{-1}) \cup \T^2 \cup (\D^{-1}\times\D)$.
\end{lemma}

\begin{proof}
Evidently, $p$ and $\refl{p}$ have no common zeros in $\D^2 \cup
(\D^{-1})^{2}$.  

Next, $p$ has no zeros on $\T\times \D$ as we now explain.  If we
define $q_z(w) = p(z,w)$, then $q_z$ has no zeros in $\D$ for each $z
\in \D$.  If we send $z \in \D$ to a point $a \in \T$, then by
Hurwitz's theorem $q_a$ is either identically zero or non-vanishing in
$\D$.  If $q_a$ is identically zero, then $z_1-a$ divides
$p(z_1,z_2)$.  However this would imply $z_1-a$ divides $\refl{p}$
which contradicts our assumption that there are no common factors.

We conclude $p$ and $\refl{p}$ have no common zeros in $\T\times \D$
as well as $\T \times \D^{-1}$, $\D\times \T$, and $\D^{-1} \times
\T$.  The only place left for common zeros is the set
\[
(\D\times \D^{-1}) \cup \T^2\cup  (\D^{-1} \times \D).
\]
\end{proof}

The first key observation is that the common zeros of $p$ and
$\refl{p}$ are closely related to joint eigenvalues of $(T_1,T_2^*)$.
When $p$ has no zeros on the closed bidisk, the following can
essentially be found in \cite{GW} with the minor difference that we
deal with joint eigenvalues.

\begin{theorem} \label{thmjoint} Let $p\in \C[z_1,z_2]$ be
  semi-stable.  If $(w_1,w_2)\in \D\times \D^{-1}$ is a common zero of
  $p$ and $\refl{p}$, then $(\bar{w}_1,1/\bar{w}_2)$ is a joint
  eigenvalue of $(T_1^*, T_2)$ and $(w_1,1/ w_2)$ is a joint
  eigenvalue of $(T_1,T_2^*)$.
\end{theorem}

\begin{proof} A point $(\lambda_1,\lambda_2)$ is a joint eigenvalue of
  $(T_1^*, T_2)$ if and only if there is an $f \in \mcG$ such that
\[
\bar{z}_1 f = \lambda_1 f + \bar{z}_1 h \text{ for some } h \in \mcE_2
\]
\[
z_2f = \lambda_2 f + g \text{ for some } g \in \mcF_1
\]
or equivalently, there is $f \in \mcG$ such that $(1-z_1
\lambda_1)f \in \mcE_2$ and $(z_2-\lambda_2)f \in \mcF_1$.

If we replace $w$ with $(\lambda_1, 1/\lambda_2)$ and multiply through
by $\bar{\lambda}_2^m$ in the first formula of Theorem \ref{thmsos},
we get
\begin{align}
\overline{q(\lambda)}p(z) -& \overline{\refl{q}(\lambda)}
\refl{p}(z) \nonumber\\
&=
(1-\bar{\lambda}_1 z_1) \bar{\lambda}_2^m
\vec{F}_1(\lambda_1,1/\lambda_2)^*\vec{F}_1(z) +(\bar{\lambda}_2 -
z_2) \bar{\lambda}_2^{m-1} \vec{E}_2(\lambda_1,1/\lambda_2)^*
\vec{E}_2(z) \nonumber \\
&= (1-\bar{\lambda}_1 z_1) \Lambda_n(\lambda_1)^* X_n
E_1(\bar{\lambda}_2)^t \vec{F}_1(z) + (\bar{\lambda}_2-z_2)
\Lambda_m(\lambda_2)^* X_m E_2(\lambda_1)^* \vec{E}_2(z) \label{subl}
\end{align}
where we use \eqref{EFbreakup} and \eqref{matrixreflection}.

Now, if $(w_1,w_2)\in \D\times \D^{-1}$ is a common zero of $p$ and
$\refl{p}$, then $\lambda = (\lambda_1,\lambda_2)\defn (w_1,1/w_2)$ is a
common zero of $q(z):=z_2^mp(z_1,1/z_2)$ and $\refl{q}$.  If we
subsitute this value for $\lambda$ into \eqref{subl} we get zero.

By Proposition \ref{Estable}, we know $\det E_1(\bar{\lambda}_1) \ne
0$, since $\lambda_1 \in \D$.  Thus, $\Lambda_n(\lambda_2)^*X_n
E_1(\bar{\lambda}_1)^t \ne 0$, and hence
\[
f_1(z):= \Lambda_n(\lambda_1)^* X_n E_1(\bar{\lambda}_2)^t
\vec{F}_1(z) \in \mcF_1
\]
is nonzero. Since \eqref{subl}$=0$
\[
(1-\bar{\lambda}_1 z_1) f_1(z) = -(\bar{\lambda}_2-z_2)
\Lambda_m(\lambda_2)^* X_m E_2(\lambda_1)^* \vec{E}_2(z)
\]
and so $(\bar{\lambda}_2 - z_2)$ divides $f_1$.  We can then define
$f:= (\bar{\lambda}_2-z_2)^{-1} f_1$ which will be an element of
$\mcG$. Note $\lambda_2 \in \D$ so dividing by this factor does not
affect whether $f \in \Lp$.  Then,
\[
(1-\bar{\lambda}_1 z_1) f(z)  = -\Lambda_m(\lambda_2)^* X_m
E_2(\lambda_1)^* \vec{E}_2(z) \in \mcE_2.
\]
So,
\[
f_1=(\bar{\lambda}_2-z_2)f \in \mcF_1 \text{ and } (1-\bar{\lambda}_1 z_1) f \in \mcE_2
\]
which implies $(\bar{\lambda}_1,\bar{\lambda}_2) =
(\bar{w}_1,1/\bar{w}_2)$ is a joint eigenvalue of $(T_1^*, T_2)$, as desired.
\end{proof}

In operator model theory language, the following theorem says that the
commuting contractions $T_1,T_2^*$ have a unitary dilation to
$M_{z_1}, M_{\bar{z}_2}$ on $\Lp$; see \cite{Pickbook}.  This is
interesting because although And\^o's dilation theorem guarantees that
some unitary dilation exists, it is surprising that the unitaries are
simple and natural.

\begin{theorem} \label{dilationthm} Let $p\in \C[z_1,z_2]$ be
  semi-stable and define $T_1,T_2$ as above.  For any $s \in
  \C[z_1,z_2]$, and $g \in \mcG$
\[
s(T_1,T_2^*) g= P s(z_1,\bar{z}_2) g.
\]
\end{theorem}

\begin{proof} Let $g \in \mcG$.  Let $j,k\geq 0$ and assume $Pz_1^k
  \bar{z}_2^j g = T_1^k (T_2^*)^j g$ and we will show $Pz_1^{k+1}
  \bar{z}_2^j g = T_1^{k+1} (T_2^*)^j g$ and $Pz_1^k
  \bar{z}_2^{j+1} g = T_1^k (T_2^*)^{j+1} g$.  By induction and by
  linearity the theorem will then follow.

We shall think of $z_1^kg$ as an element of 
\[
\mcP_{n-1+k, m-1+j} =
z_2^{j} \mcG \oplus( \mcP_{n-1+k,m-1+j} \ominus z_2^j \mcG).
\]
So,
using this decomposition we write
\[
z_1^k g = z_2^j g_0 + h
\]
with $g_0\in \mcG, h \in \mcP_{n-1+k,m-1+j} \ominus z_2^j \mcG$.
Thus,
\[
z_1^k \bar{z}_2^j g = g_0 + \bar{z}_2^j h
\]
and $P z_1^k \bar{z}_2^j g = T_1^k (T_2^*)^j g = g_0$ by hypothesis.
Since $\mcP_{n,m-1} = \mcG \oplus \mcF_2$ we can write $z_1g_0 = T_1
g_0 + f$ where $f \in \mcF_2$.  Then,
\[
z_1^{k+1} \bar{z}_2^j g = T_1 g_0 + f + z_1 \bar{z}_2^j h
\]
and so $P(z_1^{k+1} \bar{z}_2^j g) = T_1^{k+1}(T_2^{*})^j g + Pf +
P(z_1\bar{z}_2^j h)$.  But, $f \in \mcF_2 \perp \mcG$ so $Pf =0$, and
$h\in \mcP_{n-1+k,m-1+j}\ominus z_2^j\mcG \perp \bar{z}_1 z_2^j \mcG$
by Remark \ref{quadremark} since
\[
\bar{z}_1 z_2^j \mcG \subset \{f \in \Lp: \text{supp}\hat{f} \subset
\{ (a,b) \in \Z^2: a<n \text{ and } b \geq j\} \}.
\]
Thus, $P z_1\bar{z}_2^j h = 0$ and we conclude that $T_1^{k+1}
(T_2^*)^{j} g = P z_1^{k+1} \bar{z}_2^j g$.

The claim involving $T_2^*$ is similar with the key ingredient found
in Remark \ref{quadremark}.
\end{proof}

For $Q \in \C[z_1,z_2]$, we let $Z_Q = \{z \in \C^2: Q(z) =0\}$.

 \begin{theorem}\label{thmjointspec}
   Let $p\in \C[z_1,z_2]$ be semi-stable and $\deg p = (n,m)$.  Let
   $q(z) = z_2^m p(z_1,1/z_2), \refl{q}(z) = z_2^m
   \refl{p}(z_1,1/z_2)$.  Then, $q(T_1,T_2^*) = \refl{q}(T_1,T_2^*) =
   0$, and the joint spectrum $\sigma(T_1,T_2^*)$ is $Z_q\cap
   Z_{\refl{q}}\cap\D^2$.
\end{theorem}

\begin{proof}
Notice that on $\T^2$, $q(z_1,\bar{z}_2) = \bar{z}_2^m p(z)$.
Let $f,g  \in \mcG$.  By Theorem \ref{dilationthm}
\[
q(T_1,T_2^*)f = P q(z_1,\bar{z}_2) f = P \bar{z}_2^m p(z) f.
\]
The inner product of this with $g$ is 
\[
\ip{p }{ z_2^m \bar{f} g}_{\Lp}.
\]
The frequency support of $z_2^m\bar{f}g$ is in $\{(j,k): j<n, k>0\}$,
and such a function is orthogonal to $p$ and $\refl{p}$ by Theorem
\ref{thmorth}.  This shows $q(T_1,T_2^*)f = 0$.  A similar argument
applies to $\refl{q}$.

Thus, if we have a joint eigenvalue $\lambda$ with joint eigenvector
$f$, then $q(T_1,T_2^*) f = q(\lambda)f = 0$ and $\refl{q}(T_1,T_2^*)f
= \refl{q}(\lambda)f = 0$.  So, $q(\lambda) = \refl{q}(\lambda) = 0$.
This shows $\sigma(T_1,T_2^*) \subset Z_q\cap Z_{\refl{q}} \cap
\cD^2$.  Neither $T_1$ nor $T_2^*$ can have unimodular eigenvalues
because we would get $(z_1-\lambda_1)f =g \in \mcF_2$ and this implies
$\|f\|^2 = |\lambda_1|^2\|f\|^2 + \|g\|^2$ and so $|\lambda_1|<1$; and
similarly for $T_2^*$.

By Theorem \ref{thmjoint}, common zeros of $q, \refl{q}$ inside $\D^2$
are joint eigenvalues of $(T_1,T_2^*)$.  Thus,
\[
\sigma(T_1,T_2^*) = Z_q \cap Z_{\refl{q}} \cap \D^2.
\]
\end{proof}

We see that $\dim \mcG$ is at least the number of elements of $(Z_q
\cap Z_{\refl{q}} \cap \D^2)$.  Our goal is to show $\dim \mcG =
\#(Z_q\cap Z_{\refl{q}}\cap \D^2)$ if we count roots with appropriate
multiplicities.

\section{Switching from $\D^2$ to $\D\times \D^{-1}$}
As the previous section indicates, the polynomial $q(z) =
z_2^mp(z_1,1/z_2)$, which has no zeros in $\D\times \D^{-1}$, is in
some ways more natural than $p$.  One approach to counting $\dim \mcG$
is to study $\Lq$ instead and write out formulas and orthogonality
relations analogous to Theorems \ref{thmsos} and \ref{thmorth}.  

Rather than go through all of that, we shall do some simple
conversions between $p$ and $q$ that we will need later. The main
technical fact we need is as follows.

\begin{prop} \label{proppq} Let $p\in \C[z_1,z_2]$ be semi-stable and
  define $q$ as above.  There exist one variable polynomials $h_1 \in
  \C[z_2], h_2 \in \C[z_1]$ with no zeros in $\D$ such that the
  entries of
\[
h_1(z_2) \bar{F}_1(z_2) X_n \Lambda_n(z_1) \text{ and } h_2(z_1)
F_2(z_1) X_m \Lambda_m(z_2)
\]
belong to the ideal $\langle q,\refl{q} \rangle$. Here $\bar{F}_1(z_2)
= \overline{F_1(\bar{z}_2)}$.  
\end{prop}

\begin{proof}
Recall from Lemma \ref{Llemma} and Lemma \ref{Llemma2}
 \[
\begin{aligned}
 L_{w_1}(z) &= L (z_1,z_2;w_1) = z_2^{m} \frac{p(z) \overline{p(w_1, 1/\bar{z}_2)}
   - \refl{p}(z) \overline{\refl{p}(w_1, 1/\bar{z}_2)} }{1-z_1
   \bar{w}_1}\\
&= \Lambda_n(w_1)^* X_n F_1(z_2)^t E_1(z_2) \Lambda_n(z_1).
\end{aligned}
\]
Now let $H_{w_1}(z) = H(z_1,z_2;w_1) = z_2^{2m} L(z_1,1/z_2;w_1)$.
Then, $H_{w_1} \in \langle q, \refl{q} \rangle$ for each $w_1$ by
Lemma \ref{Llemma} since 
\[
H(z_1,z_2;w_1) = z_2^{m} \frac{q(z) \overline{q(w_1, 1/\bar{z}_2)}
  - \refl{q}(z) \overline{\refl{q}(w_1, 1/\bar{z}_2)} }{1-z_1
  \bar{w}_1}.
\]
On the other hand,
\[
\begin{aligned}
H(z_1,z_2;w_1) &= \Lambda_n(w_1)^* X_n z_2^m F_1(1/z_2)^t z_2^m
E_1(1/z_2) \Lambda_n(z_1) \\
&= \Lambda_n(w_1)^* \overline{E_1(\bar{z}_2)}
\overline{F_1(\bar{z}_2)} X_n \Lambda_n(z_1).
\end{aligned}
\]
Therefore, the entries of $\bar{E}_1(z_2) \bar{F}_1(z_2) X_n
\Lambda_n(z_1)$ belong to $\langle q,\refl{q} \rangle$.  If we
multiply by the adjugate of $\bar{E}_1$ we see that the entries of
$\det(\bar{E}_1(z_2)) \bar{F}_1(z_2) X_n \Lambda_n(z_1)$ belong to
$\langle q ,\refl{q} \rangle$.  Since $\det(\bar{E}_1(z_2))$ has no
zeros in $\D$, the first part of the proposition is proved.

The second part is similar. Let
\[
K_{w_2}(z) =  K(z_1,z_2;w_2) = z_1^{n} \frac{p(z) \overline{p(1/\bar{z}_1, w_2)}
  - \refl{p}(z) \overline{\refl{p}(1/\bar{z}_1, w_2)} }{1-z_2
  \bar{w}_2}.
\]
Then, by Theorem \ref{thmsos}
\[
K(z_1,z_2;w_2) = \Lambda_m(w_2)^* X_m E_2(z_1)^t
F_2(z_1) \Lambda_m(z_2)
\]
and if we define $J(z_1,z_2;w_2) =(z_2\bar{w}_2)^{m-1}
K(z_1,1/z_2;1/w_2)$ then
\[
\begin{aligned}
J(z_1,z_2;w_2) &= z_1^{n} \frac{q(z) \overline{q(1/\bar{z}_1, w_2)}
  - \refl{q}(z) \overline{\refl{q}(1/\bar{z}_1, w_2)} }{1-z_2
  \bar{w}_2} \\
&=\Lambda_m(w_2)^* E_2(z_1)^t F_2(z_1) X_m
\Lambda_m(z_2). 
\end{aligned}
\]
As before, the entries of $\det E_2(z_1) F_2(z_1) X_m \Lambda_m(z_2)$
belong to $\langle q, \refl{q} \rangle$.  Since $\det E_2$ has no
zeros in $\D$, this proves the second claim of the proposition.
\end{proof}

\section{Background: intersection multiplicities} \label{background}
This section discusses intersection multiplicities for plane curves.
We also discuss B\'ezout's theorem for $\C_{\infty} \times \C_{\infty}$. 

Historically, there are at least 3 equivalent ways to compute the
intersection multiplicity of a common zero of two plane curves.  One
can use resultants (see \cite{Fischer} section 2.7), however this
method requires putting the polynomials into general position through
linear change of variables.  This simple approach seems to be fraught
since the polynomials we are interested in do not behave well under
linear transformations.  The other ways to compute intersection
multiplicity, outlined below, are dimension counts of quotients of
local rings, dimension counts of generalized eigenspaces, and order
of vanishing of resultants of Puiseux expansions.

Let $I \subset \C[z_1,z_2]$ be a zero-dimensional ideal; meaning
$V(I)\defn \{z: \forall f \in I, f(z)=0\}$ is a finite set.  For $\lambda
\in V(I)$ we let $\mathcal{O}_{\lambda}$ denote the localization of
$\C[z_1,z_2]$ at $\lambda$, or in concrete terms the ring of rational
functions whose denominators do not vanish at $\lambda$.  The
intersection multiplicity $N_{\lambda}(I)$ is defined by
\[
N_{\lambda} (I) \defn \dim (\mathcal{O}_{\lambda}/I \mathcal{O}_{\lambda}).
\]
See \cites{CLO, Fulton}. Here the ideal $I\mathcal{O}_{\lambda}$ is the ideal
generated by $I$ in $\mathcal{O}_{\lambda}$.  If $I = \langle
p, q\rangle$, the ideal generated by $p, q \in
\C[z_1,z_2]$, we may write $N_{\lambda}(p, q)$ for
$N_{\lambda} (I)$.

The intersection multiplicities can be computed as the dimensions of
certain generalized eigenspaces as well.  Let $[f]$ denote the
equivalence class of $f \in \C[z_1,z_2]$ in $Q = \C[z_1,z_2]/I$.
Then, the maps
\[
M_1 [f] := [z_1 f] \qquad M_2[f] := [z_2 f]
\]
are well-defined linear maps on $Q$ such that $M_1$ and $M_2$ commute.
A point $\lambda$ is in $V(I)$ if and only if $\lambda =
(\lambda_1,\lambda_2)$ is a joint eigenvalue of $M_1,M_2$ and the
corresponding joint generalized eigenspace is isomorphic to
$\mathcal{O}_{\lambda}/I\mathcal{O}_{\lambda}$.  The book \cite{CLO} proves something that
amounts to the same thing, namely if $g \in \C[z_1,z_2]$, then the
eigenvalues of the map $[f] \mapsto [gf]$ are the values of $g$ on
$V(I)$. Furthermore, if $g$ takes on distinct values $g(\lambda_1),
\dots, g(\lambda_M)$ on the points of $V(I)$ then the generalized
eigenspaces are isomorphic to $\mathcal{O}_{\lambda_j}/I \mathcal{O}_{\lambda_j}$ for
$j=1,\dots, M$.  See exercise 12 chapter 4.2 of \cite{CLO}.  Thus, the
dimensions of generalized eigenspaces can be used to compute
intersection multiplicities. 
In other words, if
\begin{equation} \label{geneig}
G_{\lambda} = \{[f] \in Q: \exists N,M \text{ such that } (z_1-\lambda_1)^Nf,
(z_2-\lambda_2)^M f \in I\}
\end{equation}
then $\dim G_{\lambda} = N_{\lambda}(I)$.

\begin{remark} \label{intaxioms}
The book \cite{Fulton} also gives a list of properties of the
intersection multiplicity of two polynomials $p,q$ with no common
factor that yields an algorithm for its computation.  Among these are
\begin{enumerate}
\item $N_{\lambda} (p,q) = 0$ if and only if $\lambda$ is not a common zero of
  $p,q$.
\item $N_{\lambda} (p,q) \geq m_{\lambda} (p) m_{\lambda}(q)$ where
  $m_{\lambda}$ denotes the order of vanishing at $\lambda$ (the
  degree of the lowest non-zero term in the homogeneous expansion at
  $\lambda$) of the given polynomial.
\item $N_{\lambda} (p,q) = N_{\lambda} (p, q+ r p)$ for any $r \in
  \C[z_1,z_2]$.  
\item If $p = \prod p_j^{s_j}$ and $q = \prod q_j^{t_j}$ is the
  decomposition into irreducible components, then $N_{\lambda}(p,q) =
  \sum s_j t_k N_{\lambda}(p_j,q_k)$.  
\end{enumerate}
\end{remark}

In Appendix C, we make use of an older method for computing intersection
multiplicity using Puiseux series if $I=\langle p,q\rangle$; see
\cite{Fischer} section 8.7.  For
simplicity we assume $\lambda =(0,0) \in V(I)$.  We may write $p=u_1
\prod p_j^{s_j}$ and $q = u_2 \prod q_j^{t_k}$ where now $u_1,u_2$ are
functions analytic and non-vanishing in a neighborhood of $(0,0)$ and
the $p_j$'s and $q_k$'s are irreducible Weierstrass polynomials.  The
intersection multiplicity can be computed via
\[
N_0 (p,q) = \sum s_j t_k N_0(p_j,q_k)
\]
where it is shown separately in \cite{Fischer} how to compute
$N_0(p_j,q_k)$ for Weierstrass polynomials.  We may as well assume for
simplicity of notation that $p=p_j, q=q_k$.  By Puiseux's theorem (see
Chapter 7 of \cite{Fischer}), there exist univariate functions
$\phi, \psi$ which are analytic in a neighborhood of $0$ such that 
\[
p(t^N, \phi(t)) = 0 \text{ and } q(t^M, \psi(t)) = 0
\]
for some positive integers $N,M$.  The intersection multiplicity of
$p$ and $q$ at $0$ can now be computed as the order of vanishing of
the following formal power series in fractional powers of $t$ 
\[
f(t) = \prod_{j=1}^{N}\prod_{k=1}^{M} (\phi(\mu^j t^{1/N}) - \psi(\nu^k
t^{1/M})).
\]
Here $\mu = e^{2\pi i/N}, \nu = e^{2\pi i/M}$.  One can show $f$ is actually a power series in $t$ (and does not involve
fractional powers in the end) and the order of vanishing of $f$ at $0$
equals $N_0(p,q)$.  

A few words about B\'ezout's theorem for $\mathbb{P}\times \mathbb{P}$
will be helpful for later. Although this is a standard result in
algebraic geometry it is difficult to find an elementary discussion of
it in the literature, in contrast to the setting of two-dimensional
projective space $\mathbb{P}^2 \ne \mathbb{P}\times \mathbb{P}$ .

Let $F,G \in \C[z_0,z_1,w_0,w_1]$ be bihomogeneous, meaning
homogeneous in $(z_0,z_1)$ and $(w_0,w_1)$ separately.  We can then
associate bidegrees $(n_1,n_2), (m_1,m_2)$ to $F$ and $G$
respectively; e.g. $n_1$ is the degree of $F$ with respect to
$(z_0,z_1)$.  Assuming $F,G$ have no common factors, B\'ezout's
theorem for $\mathbb{P}\times \mathbb{P}$ says that $F,G$ have
\[
n_1m_2 + n_2 m_1
\]
common zeros in $\mathbb{P}\times \mathbb{P}$ with multiplicities
counted using the local ring definition as presented in Section
\ref{background}. This is found in \cite{shaf} (Chapter 4, Section
2.1, Example 4.9), however we caution that it is stated for ``divisors
in general position,'' which if one tracks through the definitions in
\cite{shaf} gives the result above.

In this paper we deal with the related situation of $p,q \in \C[z,w]$
with bidegrees $(n_1,n_2),(m_1,m_2)$ and no common factors which we
can ``bihomogenize'' via
\[
F(z_0:z_1,w_0:w_1) = z_0^{n_1}w_0^{n_2} p(z_1/z_0, w_1/w_0) 
\]
\[
 G(z_0:z_1,w_0:w_1) = z_0^{m_1} w_0^{m_2} q(z_1/z_0, w_1/w_0).
\]
Then for instance a common zero of $p,q$ at $(a,\infty)$ is just a
common zero of $F,G$ at $(z_0,z_1,w_0,w_1)= (1,a, 0,1)$; we are simply
using the Riemann sphere $\C_{\infty}$ model instead of projective
space $\mathbb{P}$.  Thus, $p,q$ will have $n_1m_2+n_2m_1$ common
zeros in $\C_{\infty} \times \C_{\infty}$ as before.

\section{The dimension theorem: Theorem
  \ref{intthmdim}} \label{secdimthm} Let $p\in \C[z_1,z_2]$ be
semi-stable and define $q(z)= z_2^mp(z_1,1/z_2)$ and $\refl{q}$ as in
Theorem \ref{thmjointspec}.  

For $\lambda \in Z_q \cap Z_{\refl{q}} \cap \D^2$ we define the joint
generalized eigenspace of $(T_1,T_2^*)$ for eigenvalue $\lambda$ to be
\[
\mcG_{\lambda} := \{ f\in \mcG: \exists N,M \text{ such that }(T_1-\lambda_1)^N f =
(T_2^* - \lambda_2)^M f = 0\}
\]
 
The goal of this section is to prove the following theorem.
\begin{theorem} \label{thmdimsum} Let $p\in \C[z_1,z_2]$ be
  semi-stable with $\deg p = (n,m)$.  Let $q(z) = z_2^mp(z_1,1/z_2)$.
  Then, for each $\lambda \in Z_q \cap Z_{\refl{q}} \cap \D^2$
\[
N_{\lambda}(q,\refl{q}) = \dim \mcG_{\lambda}.
\]
Therefore,
\[
\dim \mcG = \sum_{\lambda \in Z_q\cap Z_{\refl{q}} \cap \D^2}
N_{\lambda} (q,\refl{q}) 
\]
\end{theorem}

Let $\mathcal{I} = \langle q, \refl{q}\rangle$, the ideal generated by
$q,\refl{q}$,  and $Q = \C[z_1, z_2]/\mathcal{I}$ which is
necessarily finite dimensional.  
We will use $[f]$ to denote the
 equivalence class of $f \in \C[z_1,z_2]$ in $Q$.  
Recall \eqref{geneig}
\[
G_\lambda = \{ [f] \in Q: \exists N,M \text{ such that } (z_1-
\lambda_1)^N f, (z_2-\lambda_2)^Mf \in \mathcal{I}\}.
\]
and $\dim G_{\lambda} = N_{\lambda}(q,\refl{q})$.

For $f \in \C[z_1,z_2]$ we define 
\[
f^{\#}(z_1,z_2) \defn z_2^{m-1} f(z_1,1/z_2) \in \C[z_1,z_2,z_2^{-1}]
\]
a Laurent polynomial in $z_2$.  Notice that on $\T^2$, $f^{\#}(z) =
z_2^{m-1} f(z_1,\bar{z}_2)$ and also $q^{\#}(z) = \bar{z}_2 p(z)$.

\begin{lemma} Suppose $\lambda \in Z_q\cap Z_{\refl{q}}\cap \D^2$.  If
  $[f] \in G_{\lambda}$, then $f^{\#} \in \Lp$, $Pf^{\#} \in
  \mcG_{\lambda}$,  and if $[f] = 0$, then $f^{\#}\perp \mcG$.  Thus,
  the linear map $V[f] \defn P f^{\#}$ from $G_{\lambda}$ to
  $\mcG_{\lambda}$ is well-defined.
\end{lemma}

\begin{proof} If $[f] \in G_{\lambda}$, then there exist $N,M$ such that
  $(z_1-\lambda_1)^N f, (z_2-\lambda_2)^Mf \in \mathcal{I}$.  So, we may write
\[
(z_1-\lambda_1)^N f = A q + B \refl{q} \text{ and } (z_2-\lambda_2)^M
f = C q + D \refl{q}
\]
for some $A,B,C,D \in \C[z_1,z_2]$.  Then, applying the $\#$ operation and
restricting $z \in \T^2$
\[
\begin{aligned}
  (z_1-\lambda_1)^N f^{\#} &= A(z_1,\bar{z}_2) \bar{z}_2 p +
B(z_1,\bar{z}_2) \bar{z}_2 \refl{p} \\
  (\bar{z}_2-\lambda_2)^N f^{\#} &= C(z_1,\bar{z}_2) \bar{z}_2 p +
D(z_1,\bar{z}_2) \bar{z}_2 \refl{p} 
\end{aligned}
\]
Since $\lambda_1 \in \D$, we see that $f^{\#} \in \Lp$.  By
\eqref{pperp}, $z_1^j\bar{z}_2^{k} p$ is orthogonal to $\mcG$ for
$j\geq 0$ and $k>0$ and the same holds for $\refl{p}$.  So,
$A(z_1,\bar{z}_2)\bar{z}_2p \perp \mcG, B(z_1,\bar{z}_2)\bar{z}_2
\refl{p} \perp \mcG$ and similarly for $C,D$.  Therefore,
$(z_1-\lambda_1)^N f^{\#}, (\bar{z}_2-\lambda_2)^Mf^{\#} \perp \mcG$.
We may write $f^{\#} = Pf^{\#} + (I-P)f^{\#}$.

By Remark \ref{quadremark}, since $(I-P)f^{\#} \perp \mcG$ we see that
$(I-P)f^{\#}$ is in fact orthogonal to $\bar{z}_1^j z_2^k \mcG$ for
$j,k \geq 0$.  Therefore, $(z_1-\lambda_1)^N (I-P)f^{\#},
(\bar{z}_2-\lambda_2)^M (I-P)f^{\#}$ are both orthogonal to $\mcG$,
whence we conclude $(z_1-\lambda_1)^N Pf^{\#}, (\bar{z}_2-
\lambda_2)^{M}Pf^{\#}$ are orthogonal to $\mcG$.  By Theorem
\ref{dilationthm},
\[
\begin{aligned}
0 &= P (z_1-\lambda_1)^N Pf^{\#} = (T_1-\lambda_1)^N Pf^{\#} \\
0 &= P (\bar{z}_2-\lambda_1)^M Pf^{\#} = (T_2^{*}-\lambda_2)^M Pf^{\#} 
\end{aligned}
\]
and this implies $Pf^{\#} \in \mcG_{\lambda}$.

Next, if $[f]=0$, we know $f = A q + B \refl{q}$ for some $A, B \in
\C[z_1,z_2]$.  Again, $f^{\#} = \bar{z}_2 A(z_1,\bar{z}_2) p +
\bar{z}_2 B(z_1,\bar{z}_2) \refl{p}$. Again by \eqref{pperp},
$f^{\#} \perp \mcG$ so that $V[f] = 0$.
\end{proof}

\begin{lemma} Assume the setup of the previous lemma.
  The map $V: G_{\lambda} \to \mcG_{\lambda}$ is injective.  
\end{lemma}

\begin{proof} Suppose $[f] \in G_{\lambda}$ and $f^{\#} \perp \mcG$.
  Letting $(J,K) = \deg f$, $g := z_2^{K-m+1} f^{\#} \in \mcP_{J,K} \ominus
  z_2^{K-m+1} \mcG$.  By \eqref{cordecomp2}
\[
g(z) = g_0(z_1,z_2) p(z) +z_1^{L-n+1} \vec{g}_1(z_2) \vec{E}_1(z) +
z_2^{K-m+1} \vec{g}_2(z_1) \vec{F}_2(z)
\]
for $g_0 \in \C[z_1,z_2]$ of degree at most $(L-n,K-m)$,
$\vec{g}_2 \in \C^{m}[z_1]$ of degree at most $L-n$, $\vec{g}_1 \in
\C^{n}[z_2]$ of degree at most $K-m$.  We convert back to $f$ by
replacing $z_2$ with $1/z_2$ and multiplying through by $z_2^K$ to get
\[
f(z) = f_0(z_1,z_2) q(z) + z_1^{L-n+1} \vec{f}_1(z_2) \bar{F}_1(z_2)
X_n \Lambda_n(z_1) + \vec{f}_2(z_1) F_2(z_1) X_m \Lambda_m(z_2)
\]
for appropriate polynomials $f_0, \vec{f}_1, \vec{f}_2$; here
$\bar{F}_1(z_2)= \overline{F_1(\bar{z}_2)}$, where we are taking an
entrywise conjugate.  By Proposition \ref{proppq}, there exist one
variable polynomials $h_1 \in \C[z_2], h_2\in \C[z_1]$ with no zeros
in $\D$ such that $h_1(z_2)h_2(z_1) f(z) \in \langle q,\refl{q}
\rangle$.  Since $[f] \in G_{\lambda}$, we know $(z_1-\lambda_1)^N f,
(z_2-\lambda_2)^Mf \in \langle q,\refl{q}\rangle$.  There exist
$a_0,a_1,a_2 \in \C[z_1,z_2]$ such that
\[
1 = a_0(z) h_1(z_1)h_2(z_2) + a_1(z) (z_1-\lambda_1)^{N} + a_2(z)
(z_2-\lambda_2)^{M}
\]
since the ideal $\langle h_1h_2, (z_1-\lambda_1)^{N},
(z_2-\lambda_2)^{M}\rangle$ is all of $\C[z_1,z_2]$ (by
Nullstellensatz).  Hence, $f \in \langle q,\refl{q} \rangle$, which shows
$[f] =0$ and the map $V$ is injective.
\end{proof}

\begin{lemma} Assume the setup of the previous lemma.
  The map $V: G_{\lambda} \to \mcG_{\lambda}$ is surjective.
\end{lemma}

\begin{proof}
Let $g \in\mcG_{\lambda}$.  Then, $(z_1-\lambda_1)^N g,
(\bar{z}_2-\lambda_2)^Mg \perp \mcG$ for some $N,M$.  So,
$(z_1-\lambda_1)^N g \in \mcP_{N+n-1,m-1}\ominus \mcG$ and by
Corollary \ref{cordecomp} 
\[
(z_1-\lambda_1)^N g = \vec{g}_2(z_1) \vec{F}_2(z)
\]
where $\vec{g}_2$ has degree at most $N-1$.  Also, $(1-\lambda_2
z_2)^M g \in \mcP_{n-1,M+m-1} \ominus z_2^M \mcG$ and by
\eqref{cordecomp1}
\[
(1-\lambda_2 z_2)^M g = \vec{g}_1(z_2) \vec{E}_1(z)
\]
where $\vec{g}_1$ has degree at most $M-1$.  Applying the $\#$
operation yields
\[
(z_1-\lambda_1)^N g^{\#} = \vec{g}_2(z_1) F_2(z_1) X_m \Lambda_m(z_2)
\]
\[
(z_2-\lambda_2)^M g^{\#} = z_2^{M-1}\vec{g}_1(1/z_2) \bar{F}_1(z_2)
X_n \Lambda_n(z_1).
\]
By Proposition \ref{proppq}, there exist $h_1 \in \C[z_2], h_2 \in
\C[z_1]$ with no zeros in $\D$ such that
\[
h_2(z_1) (z_1-\lambda_1)^N g^{\#}, h_1(z_2) (z_2-\lambda_2)^M g^{\#}
\in \langle q ,\refl{q} \rangle,
\]
so that $h_1(z_2)h_2(z_1)g^{\#} \in G_{\lambda}$.
Using a well-known trick \cites{Fulton, CLO}, we define
\[
h = 1 - (1 - \frac{h_1(z_2)h_2(z_1)}{h_1(\lambda_2) h_2(\lambda_1)}
)^{N+M}.
\]
Now $1-h=(1 - \frac{h_1(z_2)h_2(z_1)}{h_1(\lambda_2) h_2(\lambda_1)}
)^{N+M}$ can be expanded as a combination of terms $(z_1-\lambda_1)^j
(z_2-\lambda_2)^k$ where $j \geq N, k \geq M$.  Thus,
$H := 1-h(z_1,\bar{z}_2)$ is a combination of terms
$(z_1-\lambda_1)^j(\bar{z}_2-\lambda_2)^k$ where $j\geq N, k\geq M$
and therefore $H g$ is orthogonal to $\mcG$ by
Remark \ref{quadremark} (for instance the remark could be applied to
$(z_1-\lambda_1)^Ng$).  But, $h$ is a combination of powers of
$h_1h_2$ so that $h g^{\#} \in G_{\lambda}$.  Therefore,
\[
V[hg^{\#}] = P(h g^{\#})^{\#} = P( (1-H)g) =  Pg = g
\]
which shows $V$ is surjective.   
\end{proof}

\begin{proof}[Proof of Theorem \ref{thmdimsum}]
  We conclude from these lemmas that $N_{\lambda}(q,\refl{q}) = \dim
  G_{\lambda} = \dim \mcG_{\lambda}$ for $\lambda \in Z_q\cap
  Z_{\refl{q}} \cap \D^2$.  The theorem follows immediately because
  the sum of the dimensions of the generalized eigenspaces equals the
  dimension of the underlying space.
\end{proof}

If $\lambda = (\lambda_1,\lambda_2) \in Z_q\cap Z_{\refl{q}} \cap \T^2$, then
$\refl{\lambda} = (\lambda_1, \bar{\lambda}_2) \in Z_{p}\cap Z_{\refl{p}} \cap \T^2$
and the multiplicities match $N_{\lambda}(q,\refl{q}) =
N_{\refl{\lambda}} (p,\refl{p})$.  This follows from the
isomorphism between the localizations 
\[
\mathcal{O}_{\lambda}/ (\langle q, \refl{q} \rangle
\mathcal{O}_{\lambda})\to
\mathcal{O}_{\refl{\lambda}}/ (\langle p, \refl{p} \rangle
\mathcal{O}_{\refl{\lambda}})
\]
given by $f(z) \mapsto f(z_1,1/z_2)$.

Let $N_{\T^2}(p,\refl{p})$ denote the sum of the multiplicities of the
common roots of $p$ and $\refl{p}$ on $\T^2$.  By the above remarks,
$N_{\T^2} (p,\refl{p}) = N_{\T^2} (q,\refl{q})$.  Theorem \ref{intthmdim}
from the introduction is given by the following corollary.

\begin{corollary} Let $p\in \C[z_1,z_2]$ be semi-stable and $\deg p =
  (n,m)$.  For $j\geq n-1, k \geq m-1$
\[
\dim \mcP_{j,k} = (j+1)(k+1) - \frac{1}{2}N_{\T^2}(p,\refl{p}).
\]
\end{corollary}

\begin{proof} By the B\'ezout theorem for $\C_{\infty}\times
  \C_{\infty}$ (see Section \ref{background}), $p$ and $\refl{p}$ have
  $2nm$ common zeros in $\C_{\infty} \times \C_{\infty}$, where we
  count zeros with appropriate multiplicities.  Let
  $N_{\D^2}(q,\refl{q})$ be the sum of the intersection multiplicities of the common
  roots of $q$ and $\refl{q}$ in $\D^2$.  By reflective symmetry of
  the common roots of $q,\refl{q}$ we have
\[
2nm = 2 N_{\D^2}(q,\refl{q}) + N_{\T^2}(q,\refl{q}) = 2 \dim \mcG + N_{\T^2}(p,\refl{p})
\]
by Theorem \ref{thmdimsum} and since $Z_q\cap Z_{\refl{q}} \subset
\D^2 \cup \T^2 \cup (\D^{-1})^2$; this follows from Lemma
\ref{commonzeros}.

This proves the corollary for $j=n-1,k=m-1$.  In general we use the
orthogonal decomposition of Corollary \ref{cordecomp} to see that
\[
\begin{aligned}
\dim \mcP_{j,k} =& \dim \mcG + n(k-m+1)\\
&+m(j-n+1) + (k-m+1)(j-n+1) \\
=& nm
+  n(k-m+1)+m(j-n+1) \\
&+ (k-m+1)(j-n+1) -
\frac{1}{2}N_{\T^2}(p,\refl{p}) \\
=& (j+1)(k+1) - \frac{1}{2}N_{\T^2}(p,\refl{p}) .
\end{aligned}
\]
\end{proof}

\begin{corollary} Let $p \in \C[z_1,z_2]$ be semi-stable.  Then, $p$
  has a unique Agler pair (up to unitary multiplication) iff $p$ and
  $\refl{p}$ have $2nm$ common zeros on $\T^2$, counting
  multiplicities; i.e. all common roots in $\C_{\infty}\times
  \C_{\infty}$ must be on $\T^2$.
\end{corollary}

\begin{proof} By Corollary \ref{corunique}, uniqueness of Agler pairs
  is equivalent to $\mcG = \{0\}$.  By the previous corollary, $\mcG$
  is trivial iff $N_{\T^2}(p,\refl{p}) = 2nm$.
\end{proof}

We show in Appendix C that the multiplicity at every common zero on
$\T^2$ is even.  This is obvious given a local B\'ezout theorem which
would say that if two polynomials have $k$ common zeros counting
multiplicities in an open set, then small perturbations of the
polynomials have this property.  We give a direct proof using Puiseux
series.

\section{Non-tangential boundary behavior and Theorem \ref{intthmnontan}} \label{secnontan}

Next we examine the non-tangential boundary behavior of rational
functions holomorphic in $\D^2$.  Some of our results hold naturally
in $d$ variables, so we keep this level of generality until we need
machinery that is only valid in two dimensions.

Let $u = (1,1,\dots, 1) \in \T^d$. For $z \in \C^d$ we write
$1/z \defn (1/z_1,\dots, 1/z_d)$.  This section is
entirely about local behavior at a point of $\T^d$ so we can without
loss of generality focus on $u$.  Let $RHP = \{z \in \C: \Re z >0\}$.
Everything in this section hinges on the following fact.

\begin{theorem} \label{thmbottom}
  Let $p\in \C[z_1,\dots,z_d]$ have total degree $n$, no zeros in
  $\D^d$, and assume $p(u) = 0$ with order $M$, meaning
\[
p(u-\zeta) = \sum_{j=M}^{n} P_j(\zeta)
\]
where the $P_j$ are homogeneous polynomials of degree $j$.  Then,
$P_M$ has no zeros in $RHP^d$.
\end{theorem}

\begin{proof}
Observe that 
\[
P_M(\zeta) = \lim_{r \searrow 0} \frac{1}{r^M}
p(u-r\zeta).
\]
Now, $\zeta \mapsto p(u-r\zeta)$ has no zeros in the
region
\[
R_r = \{\zeta\in \C^d: \Re \zeta_j >  r |\zeta_j|^2 \text{ for }
j=1,\dots, d\}.
\]
These regions increase as $r>0$ decreases to $0$.  By Hurwitz's
theorem $P_M$ has no zeros in $R_r$ for every $r>0$ ($P_M$ is not
identically zero by construction).  Since $\bigcup_{r>0} R_r = RHP^d$,
$P_M$ has no zeros in $RHP^d$.  
\end{proof}

When studying $\zeta \in RHP^d$ approaching $0$ non-tangentially we
will think of the elements of
\[
D_{\zeta} = \{|\zeta_1|,\dots, |\zeta_d|, \Re \zeta_1,
\dots, \Re \zeta_d\}
\]
 all comparable to a quantity $r$ which is going
to $0$.  To be specific we can arbitrarily say $r=|\zeta_1|$.

A non-tangential approach region to $0$ in $RHP^d$ will be a region of
the form
\[
AR_c = \{\zeta \in RHP^d: c \geq x/y \geq 1/c \text{ for any } x,y \in D_{\zeta}\}
\]
for $c >1$.  (AR = ``approach region.'') Notice that $AR_c \cap
\{|\zeta_1|=r\}$ is a compact set in $RHP^d$ where every element of $D_{\zeta}$
is between $cr$ and $r/c$.  It is useful to point out that if $P$ is
homogeneous of degree $M$ and non-vanishing in $RHP^d$ then
\[
|P(\zeta)| \geq C r^M
\]
where $r=|\zeta_1|$ and $C = \inf\{|P(\zeta)|: |\zeta_1|=1, \zeta \in
AR_c\} >0$.

We say $f$ is non-tangentially bounded at $u$ if $f$ is bounded on
non-tangential approach regions to $u$.  We say $f = q/p$ has a
non-tangential limit at $u$ if the limit
\[
\lim_{\underset{\zeta \in AR_c}{\zeta \to 0}} f(u-\zeta)
\]
exists. We say $f$ is non-tangentially $C^k$ at $u$ if there exists a
polynomial $L$ of degree at most $k$ such that
\[
f(u-\zeta)-L(\zeta) = o(r^k)
\]
for $\zeta \to 0$ in $AR_c$ where $|\zeta_1| = r$.

\begin{prop} \label{nontanbound} Let $p \in \C[z_1,\dots,z_d]$ have no
  zeros in $\D^d$ assume and $p$ vanishes to order $M$ at $u$.  Let $q
  \in \C[z_1,\dots,z_d]$.  If $f=q/p$ then $f$ is bounded along
  non-tangential approach regions to $u$ iff $q$ vanishes to order at
  least $M$ at $u$.
\end{prop}

\begin{proof} Write
\begin{equation} \label{q}
q(u-\zeta) = \sum_{j\geq 0} Q_j(\zeta)
\end{equation}
where each $Q_j$ is homogeneous of degree $j$. Then, let
\[
g(\zeta) = f(u-\zeta) = \frac{\sum_{j\geq 0} Q_j(\zeta)}{\sum_{j=M}^{n}
  P_j(\zeta)}.
\]
If $f$ is bounded along non-tangential approach regions then certainly
$Q_0=0$.  If $q$ vanishes to order $K$, then $Q_1=\dots=Q_{K-1} = 0$
and $Q_K\ne 0$.  Choose $a_1,\dots,a_d >0$ such that for
$a=(a_1,\dots,a_d)$, $Q_K(a) \ne 0$. Then, as $r\searrow 0$
\[
g(r a) = \frac{r^K Q_K(a) + O(r^{K+1})}{r^M P_M(a) + O(r^{M+1})} =
r^{K-M} \frac{Q_K(a) + O(r)}{P_M(a) + O(r)}
\]
which can only be bounded if $K\geq M$ since $Q_K(a), P_M(a) \ne 0$.

Conversely, if $q$ vanishes to order at least $M$, then $Q_j=0$ for
$j<M$, and since $|P_M(\zeta)| \geq c r^M$ for $\zeta$ in a
non-tangential approach region and $r=|\zeta_1|$ (or any other
comparable quantity) we have
\[
|g(\zeta)| \leq \frac{O(r^M)}{c r^M + O(r^{M+1})} = O(1).
\]
\end{proof}

\begin{prop} \label{nontanprop} Let $p \in \C[z_1,\dots,z_d]$ have no
  zeros in $\D^d$ and assume $p$ vanishes to order $M$ at $u$.  Let $q
  \in \C[z_1,\dots,z_d]$ vanish to order at least $M$ at $u$.  If
  $f=q/p$ then $f$ has a limit along non-tangential approach regions
  to $u$ iff $Q_M = bP_M$ for some constant $b$, with $Q_M$ defined
  as in \eqref{q}.  In this case, the non-tangential limit will equal
  the constant $b$.
\end{prop}

\begin{proof}
  If $f$ has a limit along non-tangential approach regions to $u$
  then, employing $g$ as in the previous proof, there exists $b$ such
  that
\[
o(1) = g(\zeta) - b = \frac{\sum_{j \geq M} Q_j(\zeta) - bP_j(\zeta)
}{\sum_{j\geq M} P_j(\zeta) }= \frac{Q_M(\zeta) - b P_M(\zeta)}{
  P_M(\zeta)} \frac{1}{1+O(r)}  + O(r). 
\]
Thus, $\frac{Q_M(\zeta) - b P_M(\zeta)}{P_M(\zeta)}$ goes to $0$ as $r
\searrow 0$.  This is not possible unless $Q_M = bP_M$ by homogeneity.
Indeed, if $Q_M(a) - b P_M(a) \ne 0$ for some $a \in (0,\infty)^{d}$,
then $\frac{r^M( Q_M(a)-bP_M(a))}{r^M P_M(a)}$ is a nonzero
constant. Thus, $Q_M - bP_M$ vanishes identically.

If $Q_M = b P_M$, the above computation shows $g(\zeta) - b = O(r)$,
so $f(u-\zeta)$ goes to $b$ as $\zeta \to 0$ non-tangentially.
\end{proof}

The next fact is included for convenience.

\begin{lemma} \label{pineq}
Let $p \in \C[z_1,\dots,z_d]$ have no zeros in $\D^d$, multidegree
$n=(n_1,n_2,\dots, n_d)$, and set
\[
\refl{p}(z) \defn z^n \overline{p(1/\bar{z})}.
\]
Then, $|\refl{p}(z)| \leq |p(z)|$ for $z \in \D^d$.
\end{lemma}

\begin{proof}  If $p$ has no zeros in $\cD^d$ then $\refl{p}/p$ is
  analytic in a neighborhood of $\cD^d$ and unimodular on $\T^d$.
  Therefore, by the maximum principle $|\refl{p}/p|\leq 1$ on $\D^d$.
  If there are zeros on the boundary we look at $p_t(z) \defn p(tz)$ for
  $t \in (0,1)$ and 
\[
\refl{p}_t(z) = t^{|n|} \refl{p}(z/t).
\]
We have $|t^{|n|} \refl{p}(z/t)| \leq |p(tz)|$ for $z \in \D^d$ and if
we let $t\nearrow 1$ we get $|\refl{p}(z)|\leq |p(z)|$.
\end{proof}

\begin{prop} \label{phomo} Assume the setup of Lemma \ref{pineq}.
  Suppose $p$ vanishes to order $M$ at $u$ so that
  we can write
\[
p(u - \zeta) = \sum_{j=M}^{|n|} P_j(\zeta) \qquad \tilde{p}(u-\zeta) =
\sum_{j=M}^{|n|} Q_j(\zeta)
\]
where $P_j,Q_j \in \C[\zeta_1,\dots, \zeta_d]$ are homogeneous of
degree $j$.  Then, $\nu P_M$ has real coefficients for some $\nu \in
\T$ and $Q_M$ is a unimodular multiple of $P_M$.  
\end{prop}

\begin{proof}
If we perform the reflection operation $f \mapsto \refl{f}$ at degree
$n$ to $(z-u)^\alpha=(z_1-1)^{\alpha_1} \cdots (z_d-1)^{\alpha_d}$ we
get
\[
\begin{aligned}
z^{n-\alpha} (1-z_1)^{\alpha_1}\cdots (1-z_d)^{\alpha_d} &=
z^{n-\alpha} (-1)^{|\alpha|} (z-u)^{\alpha} \\
&=
(-1)^{|\alpha|}(z-u)^{\alpha}  \pm (z^{n-\alpha} -1)(z-u)^{\alpha} 
\end{aligned}
\]
which shows that reflecting $(z-u)^{\alpha}$ yields
$(-1)^{|\alpha|}(z-u)^{\alpha}$ plus terms of higher total degree.
This implies that in the homogeneous expansion of $\refl{p}$ we have
$Q_M = (-1)^{M} \bar{P}_M$ where $\bar{P}_M$ denotes taking conjugates
of the coefficients of $P_M$.

By Lemma \ref{pineq}, we have that for $\zeta \in RHP^d$ and $t>0$
sufficiently small
\[
|p(u-t\zeta)|^2 - |\refl{p}(u-t\zeta)|^2
\geq 0
\]
whereas when $t<0$ we have the opposite inequality.  In terms of
homogeneous expansions this expression on the left is
\[
t^{2M} |P_M(\zeta)|^2 - t^{2M}|Q_M(\zeta)|^2 + O(t^{2M+1}).
\]
Dividing by $t^{2M}$ and sending $t$ to $0$ from the left and right we
see that $|P_M(\zeta)|^2-|Q_M(\zeta)|^2$ is both $\leq$ and $\geq 0$.
Thus $|P_M(\zeta)|^2 = |Q_M(\zeta)|^2$ for $\zeta\in RHP^d$ which
implies $P_M = \mu Q_M = \mu (-1)^M \bar{P}_M$ for some $\mu \in \T$.
In turn, it follows that for $\nu = \sqrt{\bar{\mu} (-1)^M}$, $\nu P_M =
\bar{\nu} \bar{P}_M$ has real coefficients.     

\end{proof}

As mentioned in the introduction, rational inner functions on $\D^d$
are of the form
\[
\mu z^{\alpha} \frac{\refl{p}(z)}{p(z)}
\]
where $\mu \in \T$ and $\alpha$ is a multi-index (see \cite{rudin}).
Therefore, Theorem \ref{intthmnontan} from the introduction follows
from the next corollary, which is a direct consequence of Propositions
\ref{nontanprop} and \ref{phomo}.

\begin{corollary} \label{corthma} If $p \in \C[z_1,\dots, z_d]$ has no
  zeros in $\D^d$ then for any $\zeta \in \T^d$
\[
\lim_{z \to \zeta} \frac{\refl{p}(z)}{p(z)}
\]
exists as $z \to \zeta$ non-tangentially.  Moreover, this limit will
be an element of $\T$.
\end{corollary}

Proposition \ref{phomo} has the following corollary in two dimensions.

\begin{corollary}
Suppose $p \in \C[z_1,z_2]$ is semi-stable and vanishes to order
$M$ at $\lambda \in \T^2$.  Then, $N_{\lambda}(p,\refl{p}) \geq
M(M+1)$.
\end{corollary}

\begin{proof}
  By Proposition \ref{phomo} there is a $\nu \in \T$, such that $p-\nu
  \refl{p}$ vanishes to order at least $M+1$ at $\lambda$.  By Remark
  \ref{intaxioms}
\[
N_{\lambda}(p,\refl{p}) = N_{\lambda}(p, p-\nu \refl{p}) \geq M(M+1).
\]
\end{proof}

We now study higher regularity for rational inner functions.

\begin{theorem} \label{thmnontan} Suppose $f = \refl{p}/p$ has
  non-tangential value $\nu$ at $u$. Then, $f$ is non-tangentially
  $C^1$ at $u$ iff $P_M$ divides $Q_{M+1} - \nu P_{M+1}$.  More
  generally, $f$ is non-tangentially $C^k$ at
  $u$ iff
\[
F_1 \defn \frac{Q_{M+1}-\nu P_{M+1}}{P_M} \in \C[\zeta_1,\dots,
\zeta_d]
\]
\[
F_2 \defn \frac{Q_{M+2} -\nu P_{M+2} - F_1 P_{M+1}}{P_M} \in
\C[\zeta_1,\dots, \zeta_d]
\]
and so on up to the last condition
\[
F_k\defn \frac{Q_{M+k}-\nu P_{M+k} - \sum_{j=1}^{k-1} F_j
  P_{M+k-j}}{P_M} \in \C[\zeta_1,\dots, \zeta_d].
\]
In this case the non-tangential Taylor expansion is given by
$\sum_{j=1}^{k} F_j$ in the sense that
\[
f(u-\zeta)  - (\nu+\sum_{j=1}^{k} F_j(\zeta)) = o(r^k)
\]

\end{theorem}

\begin{proof} We can multiply by a unimodular constant to put $p$ in
  the form
\[
p(u-\zeta) = \sum_{j=M}^{|n|} P_j(\zeta)
\]
\[
\refl{p}(u-\zeta) = \nu P_M(\zeta) + \sum_{j=M+1}^{|n|}
Q_j(\zeta)
\]
where $P_M$ has real coefficients and no zeros in $RHP^d$.

Observe that
\begin{align}
\frac{\refl{p}(z)}{p(z)} -\nu &= \frac{\sum_{j\geq 1}
  Q_{M+j}(\zeta) - \nu P_{M+j}(\zeta) }{ \sum_{j\geq 0}
  P_{M+j}(\zeta)} \nonumber\\
&= \left(\sum_{j\geq 1} \frac{Q_{M+j}(\zeta) - \nu
  P_{M+j}(\zeta)}{P_M(\zeta)} \right) \left(\frac{1}{1 + \sum_{j\geq
    1} \frac{P_{M+j}(\zeta)}{P_M(\zeta)}} \right). \label{Or2}
\end{align}

Next, if $F_1 = (Q_{M+1} - \nu P_{M+1})/P_M$, then
\begin{multline*}
f(u-\zeta) -(\nu + F_1(\zeta)) \\
=  \left(\sum_{j\geq 1} \frac{Q_{M+j+1}(\zeta) - \nu
  P_{M+j+1}(\zeta)- F_1(\zeta) P_{M+j}(\zeta)}{P_M(\zeta)} \right) \left(\frac{1}{1 + \sum_{j\geq
    1} \frac{P_{M+j}(\zeta)}{P_M(\zeta)}} \right) \\
= \frac{O(r^2)}{1+O(r)} = O(r^2)
\end{multline*}
which shows $f$ is non-tangentially $C^1$ at $u$ assuming $F_1 \in
\C[\zeta_1,\dots,\zeta_d]$ . On the other hand, if $f$ is
non-tangentially $C^1$ at $u$, then there is a degree $1$ homogeneous
polynomial $Q$ such that
\[
f(u-\zeta) - (\nu+Q(\zeta)) = o(r)
\]
so
\[
 F_1(\zeta) -Q(\zeta) = o(r) \text{ by \eqref{Or2}}
\]
which means $F_1 = Q$ by homogeneity.  

The general case is proved similarly by induction using the formula
\begin{multline*}
f(u-\zeta) -(\nu + \sum_{j=1}^{k} F_j(\zeta)) \\
= \left(\sum_{j\geq 1} \frac{Q_{M+j+k} - \nu
  P_{M+j+k}- \sum_{m=1}^{k} F_m P_{M+k+j-m}}{P_M} \right) \left(\frac{1}{1 + \sum_{j\geq
    1} \frac{P_{M+j}}{P_M}} \right) \\
= O(r^{k+1}).
\end{multline*}

\end{proof}

We get from the above proof the existence of a non-tangential
directional derivative function
\[
F_1(\zeta) = \frac{Q_{M+1}(\zeta)-  \nu P_{M+1}(\zeta)}{P_M(\zeta)}
\]
for $f=\refl{p}/p$ even when $f$ is not non-tangentially $C^1$.  This
is closely related to a main result of \cite{AMYcara}, which holds for
bounded analytic functions on $\D^2$ (i.e. not just rational inner
functions).  The paper \cite{ATY} goes further and characterizes the
possible ``slope functions'' in two variables.

Restricting to two variables, we see that if $f$ is non-tangentially
$C^k$ at $u$, then $N_{u}(p,\refl{p})\geq M(M+k+1)$ because of the
following observation:
\[
\begin{aligned}
N_{u}(p,\refl{p}) &= N_{0}(P_M + \sum_{j\geq 1} P_{M+j}, \nu P_M +
  \sum_{j\geq 1}Q_{M+j}) \\
&= N_0 (P_M + \sum_{j\geq 1} P_{M+j}, \sum_{j\geq 1} (Q_{M+j} -\nu
P_{M+j})) \\
&= N_0 (P_M + \sum_{j\geq 1} P_{M+j}, \sum_{j\geq 1} (Q_{M+j+1} -
\nu P_{M+j+1}- F_1 P_{M+j})) \\
&= \cdots \\
&= N_0 (P_M + \sum_{j\geq 1} P_{M+j}, 
\sum_{j\geq 1}( Q_{M+j+k} - \nu 
  P_{M+j+k}- \sum_{m=1}^{k} F_m P_{M+k+j-m}) ) \\
& \geq M(M+k+1)
\end{aligned}
\]
This computation is based on the rules from Remark \ref{intaxioms}.

\begin{corollary} Suppose $p \in \C[z_1,z_2]$ is semi-stable and
  $\refl{p}/p$ is non-tangentially $C^k$ at a point $\lambda \in
  \T^2$.  If $p$ vanishes to order $M$ at $\lambda$, then
\[
N_{\lambda} (p,\refl{p}) \geq M(M+k+1).
\]
\end{corollary}

For example, if $\refl{p}/p$ is $C^1$ at $\lambda$, then
$N_{\lambda}(p,\refl{p}) \geq 4$, since the intersection multiplicity
must be even.  An interesting consequence is that the number of $C^1$
points which are not $C^2$ is finite (i.e. at most $nm/4$).

Finally, we point out that at least in two variables, if $f=q/p \in
L^2(\T^2)$ then $f$ is non-tangentially bounded at every point in
$\T^2$.  By Proposition \ref{nontanbound}, this is equivalent to
showing that $q$ vanishes at least to the same order as $p$ at every
zero of $p$ on $\T^2$.  

\begin{theorem} Assume $p \in \C[z_1,z_2]$ is semi-stable and $q\in
  \mathcal{I}_p$.  Then, $f:=q/p$ is non-tangentially bounded at every
  point of $\T^2$; equivalently, if $p$ vanishes to order $M$ at some
  point of $\T^2$ then every element of $\mathcal{I}_p$ vanishes to at
  least order $M$.
\end{theorem}

\begin{proof} We may assume $p$ vanishes to order $M$ at
  $u=(1,1)$. Let $\phi(\zeta) = \frac{\refl{p}}{p}(u-\zeta)$.  By
  Corollary \ref{corthma}, in a non-tangential approach region to
  $(0,0)$ in $RHP^2$, $\phi(\zeta) = \nu + O(r)$ for some $\nu \in
  \T$---actually this is the last line of the \emph{proof} of
  Proposition \ref{nontanprop}.

For any Agler pair $(\vec{A}_1,\vec{A}_2)$, we see that 
\[
1 - |\phi(\zeta)|^2 \geq (1-|1-\zeta_1|^2)
\frac{|\vec{A}_1(u-\zeta)|^2}{|p(u-\zeta)|^2}.
\]
A similar inequality could be written for $\vec{A}_2$.  Now,
$1-|\phi(\zeta)|^2 = 1-|\nu + O(r)|^2 = O(r)$ and $1-|1-\zeta_1|^2 =
2\Re \zeta_1 - |\zeta_1|^2 \geq c r$ for $|\zeta_1|$ small enough
(because we are in a non-tangential approach region).  Thus,
$O(1)\geq \frac{|\vec{A}_1(u-\zeta)|^2}{|p(u-\zeta)|^2}$ is bounded
along every non-tangential approach region to $(0,0)$.  Similarly,
$\frac{|\vec{A}_2|^2}{|p|^2}$ is bounded along non-tangential approach
regions to $u$.  

This allows us to conclude that $1/p(z)$ times any of
$\vec{E}_1,\vec{E}_2,\vec{F}_1,\vec{F}_2$ gives a rational function
bounded along non-tangential approach regions to $u$.

By Theorem \ref{generators}, every element of $\mathcal{I}_p$ can be
written in terms of polynomial multiples of $\mcE_1, \mcF_1,\mcF_2$.
Therefore, every element $q$ of $\mathcal{I}_p$ will vanish to at
least order $M$ at $u$, or equivalently $q/p$ will be non-tangentially
bounded at $u$.
\end{proof}

\section{Examples}
This section contains three examples to illustrate Theorems A,B,and C.
See \cite{Bickel} for a construction of more examples.

\begin{example}
The following example is taken from \cite{AMYcara}. Let
\[
p(z_1,z_2) = 4-z_1-3z_2-z_1z_2+z_2^2 \qquad \refl{p}(z_1,z_2) = 4z_1z_2^2 - z_2^2 - 3z_1z_2 - z_2 + z_1.
\]
The special Agler pairs for $p$ can be constructed as described in
Appendix B.  Namely, set $|z_2|=1$ and consider
\[
\frac{|p(z)|^2 - |\refl{p}(z)|^2}{1-|z_1|^2} = 4|(1-z_2)^2|^2.
\]
Since $(1-z_2)^2$ has no zeros in $\D$ it follows that $\vec{E}_1(z) =
2(1-z_2)^2$.  Since the reflection of this equals itself, we see that
$\vec{F}_1=\vec{E}_1$.  This automatically implies that $p$ has unique
Agler decomposition (up to unitary multiples of Agler pairs). 

The vector polynomial $\vec{E}_2=\vec{F}_2$ can be constructed as in
Remark \ref{constructgenerators}.  We get

\[
\vec{F}_2(z) = \vec{E}_2(z)= 2\begin{pmatrix} (1-z_1)(1-z_2) \\
  \sqrt{2}(1-z_1z_2) \end{pmatrix}.
\]

The Agler decomposition for $p$ is given by
\[
|p|^2- |\refl{p}|^2 = 4(1-|z_1|^2)|(1-z_2)^2|^2 + 4(1-|z_2|^2)(
|(1-z_1)(1-z_2)|^2 + 2|1-z_1z_2|^2).
\]
Because this is unique we know $\mcP_{0,1} = \{0\}$.  We can also see
this by computing the intersection multiplicity at $(1,1)$.  

The expansion of $p$ at $(1,1)$ is given by
\[
p(1-\zeta,1-\eta) = 2(\zeta + \eta) + \eta^2 - \zeta \eta \qquad
\refl{p}(1-\zeta, 1-\eta) = -2(\zeta + \eta) +5\zeta \eta + 3\eta^2 -
4\zeta \eta^2
\]
and so by Remark \ref{intaxioms}
\[
\begin{aligned}
N_{(1,1)} (p,\refl{p}) &= N_{(1,1)} (p, \refl{p}+p) \\
&=N_0 (2(\zeta + \eta) + \eta^2 - \zeta \eta , \eta(\zeta +\eta) -
\zeta \eta^2) \\
&= N_0(2(\zeta + \eta) + \eta^2- \zeta \eta, \eta) + N_0(2(\zeta +
\eta) + \eta^2- \zeta \eta, (\zeta + \eta) - \zeta \eta) \\
&= 1 + N_0( \eta^2 + \zeta \eta, \zeta + \eta - \zeta \eta) \\
&= 1 + N_0(\eta, \zeta + \eta -\zeta \eta) + N_0(\zeta + \eta, \zeta +
\eta -\zeta \eta) \\
&= 1 + 1 + N_0(\zeta + \eta, \zeta \eta) = 4.
\end{aligned}
\]
Thus, $\dim \mcP_{0,1} = 2 - (1/2)(4) = 0$. 

More generally, $\dim \mcP_{j,k} = (j+1)(k+1) - 2$.  This suggests $2$
conditions force $q \in \mathcal{I}_p$.  They are $q(1,1) = 0$ and
$\partial_1 q(1,1)=\partial_{2} q(1,1)$.  To see this, note that
$\{(1-z_2)^2, (1-z_1)(1-z_2), (1-z_1z_2)\}$ generates $\mathcal{I}_p$.  These
generators satisfy the two conditions $q(1,1)= 0, \partial_1
q(1,1)=\partial_2 q(1,1)$ and it can be shown that these conditions
determine an ideal in $\C[z,w]$ with codimension $2$.  Therefore, $q/p
\in L^2(\T^2)$ iff $q(1,1) = 0$ and $\partial_1q(1,1)
= \partial_2q(1,1)$.

The rational inner function $f = \refl{p}/p$ is non-tangentially $C^1$
because $\zeta + \eta$ divides the second order term of $p+\refl{p}$,
which is $4\eta^2+4\zeta\eta$.  In fact, $f$ is non-tangentially $C^2$
because $\nu = -1$, $F_1 = 2\eta$ and $F_2 = -\eta^2$:
\[
\frac{\refl{p}}{p} -(-1 + 2\eta - \eta^2) =
\eta^3\frac{\eta-\zeta}{2(\zeta+\eta)+\eta^2-\zeta\eta}.
\]
No higher regularity is possible because this would force an
intersection multiplicity at least $6$. This can also be seen
directly. \eox
\end{example}

\begin{example}
The next example is taken from \cite{KneseSchwarz}.  Let
\[
\begin{aligned}
p(z) = &\frac{1}{18} \left(3 \sqrt{5} z_2^2-2 z_2^2-6 \sqrt{5} z_2-9 z_2+18\right)\\
&+ \frac{1}{18} \left(9 z_2^2-14 z_2+6 \sqrt{5}-9\right) z_1 \\
&+\frac{1}{18} \left(9 z_2-3 \sqrt{5}-2\right) z_1^2.
\end{aligned}
\]
This example was designed to have the feature that for $f = \refl{p}/p$
\[
f(z_1,z_1) = z_1 \text{ and } f\left(\frac{z_1-\sqrt{5}/3}{1-(\sqrt{5}/3)z_1},
  \frac{z_1+\sqrt{5}/3}{1+(\sqrt{5}/3)z_1}\right) = z_1
\]
which means $f$ acts as an automorphism of the disk when restricted to
certain embedded disks.  One can check (with simple computer algebra)
that $N_{(1,1)} (p,\refl{p}) = 6$ and $N_{(-1,-1)}(p,\refl{p}) = 2$.
Therefore, 
\[
\dim \mcP_{1,1} = 2\cdot 2 - \frac{1}{2}(6+2) = 0
\]
by Theorem \ref{intthmdim}.  So, $p$ has a unique Agler pair.  

Now, $p$ vanishes to order $2$ at $(1,1)$, and elements
of $\mathcal{I}_p$ must have the same property.  This puts $3$
conditions on elements of $\C[z_1,z_2]$.  Also, $p$ vanishes to order
$1$ at $(-1,-1)$ and this puts one additional condition on elements of
$\C[z_1,z_2]$.  Thus, $q \in \mathcal{I}_p$ iff
\[
q(1,1)=\partial_1 q(1,1) = \partial_2 q(1,1) = 0 = q(-1,-1).
\]
These conditions are enough to show $\mcP_{2,1}$ is spanned by
$\{z_1^2-z_1-z_1z_2+z_2, z_1^2z_2-z_1-z_1z_2+1\}$ and $\mcP_{1,2}$ is spanned by
$\{z_2^2-z_2-z_1z_2+z_1, z_1z_2^2-z_2-z_1z_2+1\}$.  Since $\mcG$ is trivial, this implies
$\mcE_1=\mcF_1,\mcE_2=\mcF_2$ and therefore we can use spanning sets
for $\mcP_{2,1}=\mcE_2,\mcP_{1,2} = \mcE_1$ to generate
$\mathcal{I}_p$.  Therefore,
\[
\mathcal{I}_p = \langle z_1^2-z_1-z_1z_2+z_2, z_1^2z_2-z_1-z_1z_2+1, z_2^2-z_2-z_1z_2+z_1,
z_1z_2^2-z_2-z_1z_2+1\rangle
\]
which illustrates Theorem \ref{generators}.

Next, we discuss non-tangential regularity.  The bottom homogeneous
term of $p(1-\zeta,1-\eta)$ is
\[
(7/18-\sqrt{5}/6)\zeta^2 + (11/9)\zeta\eta +(7/18+\sqrt{5}/6)\eta^2
\]
and the bottom homogeneous term of
$p(1-\zeta,1-\eta)-\refl{p}(1-\zeta,1-\eta)$ is
\[
(1-\sqrt{5}/3)\zeta^2\eta+(1+\sqrt{5}/3) \zeta \eta^2.
\]
Since the former does not divide the latter, $f=\refl{p}/p$ is not
non-tangentially $C^1$ at $(1,1)$; notice that order of vanishing at
$(1,1)$ alone does not reveal this.  Since $N_{(-1,-1)}(p,\refl{p})
=2$, $f$ is also not non-tangentially $C^1$ at $(-1,-1)$.

\eox
\end{example}

\begin{example} \label{pascoeex} J. Pascoe has a method to construct
  rational inner functions which are non-tangentially $C^k$ but not
  $C^{k+1}$ at a point of $\T^2$---his construction will appear in
  forthcoming work \cite{Pascoe}.  He has generously allowed us to
  include the following example which comes from his construction.

Let
\[
p(z) = 4-5z_1-2z_2+2z_1z_2+3z_1^2-z_1^2z_2-z_1^3z_2 
\]
\[
\refl{p}(z) = 4 z_2 z_1^3-5 z_2 z_1^2+3 z_2 z_1-2 z_1^3+2 z_1^2-z_1-1.
\]
Note $p$ has degree $(3,1)$.  One can compute that
\[
N_{(1,1)}(p,\refl{p}) = 6
\]
which again means $\mcG=\mcP_{2,0} = \{0\}$ and $p$ has a unique Agler
pair.  Thus, $\mcE_1=\mcF_1=\mcP_{2,1}, \mcE_2=\mcF_2=\mcP_{3,0}$.  

By the dimension theorem $\mcP_{j,k} = (j+1)(k+1)-3$ for $j\geq
2,k\geq 0$.  So, $\mathcal{I}_p$ has codimension $3$ in $\C[z_1,z_2]$
and it is of interest to determine the 3 conditions imposed on
elements of $\mathcal{I}_p$.  Necessarily, $q(1,1) = 0$ for all $q\in
\mathcal{I}_p$. 

Using the method of Appendix B, one can compute that $\vec{E}_2(z) =
\sqrt{2}(1-z_1)^3$.  Once $\vec{E}_2$ is known, we can use the method
outlined after the proof of Theorem \ref{generators} to find
$\vec{E}_1=\vec{F}_1$.  In this case, we get the
following orthonormal basis for $\mcF_1=\mcE_1$
\begin{multline*}
\{2 \sqrt{\frac{2}{7}} (1-z_2 z_1), \sqrt{\frac{2}{133}} (-21 z_2 z_1^2+20 z_2
z_1-7 z_2+7 z_1^2+1), \\ \frac{1}{\sqrt{19}} (11 z_2 z_1^2-15 z_2 z_1+10 z_2-10 z_1^2+19
z_1-15) \}.
\end{multline*}
Thus, if we put these polynomials into a vector we get $\vec{F}_1$ and
hence we have computed the unique Agler pair $(\vec{F}_1,\vec{E}_2)$
(up to unitary multiplication).

Since this is messy, we use a slightly different approach to get
manageable numbers.  The coefficients of powers of $w$ in
$\vec{E}_1(w)^*\vec{E}_1(z)$ will span $\mcE_1$ and these can be found
directly in terms of $p,\refl{p}, \vec{E}_2$ by Theorem \ref{thmsos}.
This makes it possible to find the following non-orthonormal basis for
$\mcE_1$
\[
\{(1-z_1)^2(1-z_2), (1-z_1)(2-z_1-z_2), (1-z_1z_2)\}.
\]

Thus, the ideal $\mathcal{I}_p$ is generated by
\[
\{ (1-z_1)(2-z_1-z_2), (1-z_1z_2), (1-z_1)^3\}.
\]
The polynomial $(1-z_1)^2(1-z_2)$ in the basis for $\mcE_1$ can be
written in terms of these so we can safely remove it.  Using this we
can find defining relations for $\mathcal{I}_p$; namely, $q \in
\mathcal{I}_p$ iff $q(1,1)=\partial_1q(1,1)-\partial_2q(1,1)=0$ and
\[
\partial_{11}q(1,1) - 2\partial_{12}q(1,1)+\partial_{22}q(2,2) +2\partial_1q(1,1) = 0.
\]
One can check that these conditions actually define an ideal which has
codimension $3$ in $\C[z_1,z_1]$ which must then coincide with
$\mathcal{I}_p$.

We omit the details, but using Theorem \ref{thmnontan} we can
show $f=\refl{p}/p$ is non-tangentially $C^4$ but not $C^5$ at $(1,1)$.

\eox
\end{example}

\section{Appendix A: Theorem \ref{thmorth} and Lemma \ref{tfr}}
In this appendix we explain how Theorem \ref{thmorth} in Section
\ref{prelim} follows from the work in \cites{KneseAPDE, BickelKnese}.

The orthogonality relations for $p$ follow from Proposition 7.1 of
\cite{KneseAPDE}.  Since reflection $f\mapsto
\overline{f(1/\bar{z}_1,1/\bar{z}_2)}$ is an anti-unitary, the
orthogonality relations for $\refl{p}$ follow from those for $p$.  

We only need to establish the orthogonality relations for $\mcF_1$
since the relations for $\mcE_1$ follow by applying the anti-unitary
reflection, and the relations for $\mcE_2,\mcF_2$ follow by symmetry.

The orthogonality relation for $\mcF_1$ is not easily quotable from
\cite{KneseAPDE}, so we shall carefully explain how it follows from
work in \cite{BickelKnese}.

The setup of \cite{BickelKnese} is slightly different.  Instead of
working in $\Lp$ we work in $L^2(\T^2)$ and $H^2(\T^2)$ with Lebesgue
measure but the spaces of interest end up containing functions of the
form $f/p$ so that there is a direct comparison between this paper and
\cite{BickelKnese}.  Let $\phi = \frac{\refl{p}}{p}$; notice that
$\phi \in H^\infty$ and multiplication by $\phi$ is a unitary on
$L^2(\T^2)$ since $|\phi| =1$ a.e. on $\T^2$. Define
\[
\begin{aligned}
\mcH_{\phi} &= H^2 \ominus \phi H^2 \\
\mcH_{\phi}^{1} &= H^2 \cap \phi L^2_{\bullet -} \\
\mcH_{\phi}^{2} &= H^2 \cap \phi L^2_{-\bullet} \\
\mcK_{\phi} &= H^2 \cap \phi L^2_{--} \\
\mcK_{\phi}^{1} &= H^2 \cap z_1 \phi L^2_{--} \\
\mcK_{\phi}^{2} &= H^2 \cap z_2 \phi L^2_{--}
\end{aligned}
\]
where 
\[
\begin{aligned}
L^2_{\bullet -} &= \{f\in L^2(\T^2): \text{supp} \hat{f}
\subset \{(j,k): k<0\} \} \\
L^2_{-\bullet} &= \{f\in L^2(\T^2): \text{supp} \hat{f}
\subset \{(j,k): j<0\} \} \\ 
L^2_{--} &= \{f \in L^2(\T^2): \text{supp} \hat{f} \subset
\{(j,k): j,k<0\} \}. 
\end{aligned}
\]
We emphasize we are taking orthogonal complements in $L^2(\T^2)$.  We
will also use $L^2_{+\bullet}, L^2_{\bullet +}$ which are the
functions in $L^2(\T^2)$ with Fourier support in $\{(j,k):j\geq 0\},
\{(j,k): k\geq 0\}$ respectively. Warning: ``+'' refers to a non-strict
inequality in this notation and ``-'' refers to a strict inequality.

The following Proposition is similar to Proposition 5.1 of
\cite{BickelKnese}.

\begin{prop} \label{appprop}
\[
H^2 \ominus \mcH_{\phi}^1 = L^2_{+\bullet} \ominus (L^2_{+\bullet}
\cap \phi L^2_{\bullet -})
\]
\end{prop}
\begin{proof}
Let $P_{L^2_{+\bullet}}$ denote orthogonal projection onto $L^2_{+\bullet}$.
Observe
\[
\begin{aligned} 
L^2_{+\bullet} \ominus (L^2_{+\bullet}
\cap \phi L^2_{\bullet -}) &= P_{L^2_{+\bullet}}( (L^2_{+\bullet} \cap
  \phi L^2_{\bullet -})^{\perp}) \\
&= P_{L^2_{+\bullet}} (L^2_{-\bullet} \vee \phi L^2_{\bullet +}) \\
&= \text{closure} (P_{L^2_{+\bullet}} (L^2_{-\bullet} + \phi
L^2_{\bullet +}) )\\
&= \text{closure}( P_{L^2_{+\bullet}} (\phi L^2_{\bullet +}) )\\
&\subset \text{closure}( P_{L^2_{+\bullet}} (L^2_{\bullet +}) )\\
& \subset H^2
\end{aligned}
\]
The main facts we are using are $(L^2_{+\bullet})^{\perp} =
L^2_{-\bullet}$ and $(\phi L^2_{\bullet -})^{\perp} = \phi
L^2_{\bullet +}$ since multiplication by $\phi$ is a unitary. Also,
$\phi L^2_{\bullet +} \subset L^2_{\bullet +}$ since $\phi \in
H^\infty$.  

Then, we can conclude that
\[
H^2 \ominus (H^2 \cap \phi L^2_{\bullet -}) = L^2_{+\bullet} \ominus (L^2_{+\bullet}
\cap \phi L^2_{\bullet -})
\]
by a simple Hilbert space lemma from \cite{BickelKnese}. Lemma 2.6 of
\cite{BickelKnese} says that if $\mcK_1,\mcK_2$ are closed subspaces
of a Hilbert space $\mcH$ satisfying $\mcH\ominus \mcK_1 \subset
\mcK_2$, then $\mcH \ominus \mcK_1 = \mcK_2 \ominus
(\mcK_1\cap\mcK_2)$.
\end{proof}

By Proposition 2.3, 5.1 and Corollary 9.2 of \cite{BickelKnese} we have
\[
\mcH_\phi \ominus \mcH_{\phi}^{1} = \mcH_{\phi}^{2} \ominus
\mcK_{\phi}
\]
\[
(\mcH_\phi^{2} \ominus \mcK_{\phi}) \ominus z_2(\mcH_\phi^2 \ominus
\mcK_{\phi}) = \mcK_{\phi}^2 \ominus \mcK_{\phi}
\]
\[
\begin{aligned}
\mcK_{\phi} &= \{ q/p \in H^2: q \in \C[z_1,z_2], \deg q \leq (n-1,m-1) \} \\
\mcK_{\phi}^2 &= \{q/p \in H^2: q\in \C[z_1,z_2], \deg q \leq (n-1,m).
\}
\end{aligned}
\]

\begin{prop}\label{prop81} If $f \in L^2_{+\bullet}$ and $f/p \in L^2$, then $f/p
  \in L^2_{+\bullet}$.
\end{prop}

This is Proposition 8.1 of \cite{BickelKnese} except
\cite{BickelKnese} has the hypothesis $f\in H^2$ and conclusion
$f/p\in H^2$. However, the proof of Proposition 8.1 of
\cite{BickelKnese} is structured so that it proves Proposition
\ref{prop81} and then by symmetry $f \in L^2_{\bullet +}$ and $f/p\in
L^2$ implies $f/p \in L^2_{\bullet +}$, which yields the result for
$H^2$.

This shows
\[
\mcK_{\phi} = p^{-1} \mcG \text{ and } \mcK_{\phi}^2 = p^{-1} \mcP_{n-1,m}
\]
which shows the connection to this paper.  Thus, $\mcK_{\phi}^2
\ominus \mcK_{\phi} = p^{-1} \mcF_1 \subset \mcH^2_{\phi} \ominus
\mcK_{\phi} \subset H^2 \ominus \mcH_{\phi}^{1} \perp L^2_{+\bullet}
\cap \phi L^2_{\bullet -}$ by Proposition \ref{appprop} above.  
Therefore, $p^{-1} \mcF_1$ is
orthogonal (in $L^2(\T^2)$) to $\mathcal{Q} \defn L^2_{+\bullet} \cap \phi L^2_{\bullet
  -}$.  If we can show 
\[
\{f/p \in L^2: \text{supp} \hat{f} \subset \{(j,k): j \geq 0
\text{ and } k < m\} \} \subset \mathcal{Q}
\]
then we will be finished.  

Now, if $f/p \in L^2$ and $\text{supp} \hat{f} \subset \{(j,k):j\geq
0, k<m\}$ then $f/p \in L^2_{+\bullet}$ by Proposition \ref{prop81}.
We need to show $(f/p) \bar{\phi} \in L^2_{\bullet -}$ or
$\overline{(f/p)} \phi \in z_2 L^2_{\bullet +}$ so we compute
\[
\frac{\bar{f}}{\bar{p}} \frac{\refl{p}}{p}  = \frac{z_1^n z_2^m
  \bar{f}}{p} \text{ on } \T^2.
\]
Since $z_2^{m-1} \bar{f} \in L^2_{\bullet +}$ we see that the above is
in $z_2L^2_{\bullet +} $ by Proposition \ref{prop81}.  This proves
$f/p \in \mathcal{Q}$.

Thus, $p^{-1} \mcF_1 \perp f/p$ for any $f$ with $f/p \in L^2$ and $\text{supp}
\hat{f} \subset \{(j,k):j\geq 0\text{ and } k<m\}$, where the
  orthogonality ``$\perp$'' is in $L^2(\T^2)$.  But, this exactly
  means
\[
\mcF_1 \perp \{f\in \Lp: \text{supp}
\hat{f} \subset \{(j,k):j\geq 0\text{ and } k<m\} \}
\]
using the inner product of $\Lp$.  This concludes our explanation of
Theorem \ref{thmorth}.  We now prove Lemma \ref{tfr}. 

\begin{proof}[Proof of Lemma \ref{tfr}]
By definition of ``Agler pair'' we have,
\[
|p|^2+\sum_{j=1}^{2} |z_j\vec{A}_j|^2 
= |\refl{p}|^2 + \sum_{j=1}^{2} |\vec{A}_j|^2
\]
and by Proposition \ref{isomprop} there exists a $(1+N+M)\times
(1+N+M)$ isometric matrix $U$, which is necessarily a unitary because
it is square, such that
\[
U \begin{pmatrix} p \\ z_1 \vec{A}_1 \\ z_2 \vec{A}_2 \end{pmatrix}
= \begin{pmatrix} \refl{p} \\ \vec{A}_1 \\ \vec{A}_2\end{pmatrix}.
\]
Write $U$ in block form $\begin{pmatrix} A & B \\ C & D \end{pmatrix}$
where the blocks correspond to the direct sum $\C^{1+N+M} = \C\oplus
\C^{N+M}$.  Recall $\Delta(z) = \begin{pmatrix} z_1 I_N & 0 \\ 0 & z_2
  I_M \end{pmatrix}$.  Then, for $\vec{A} = \begin{pmatrix} \vec{A}_1
  \\ \vec{A}_2 \end{pmatrix}$ we have
\[
\begin{pmatrix} A & B \\ C &
    D \end{pmatrix} \begin{pmatrix} p \\ \Delta \vec{A} \end{pmatrix}
  = \begin{pmatrix} A p + B \Delta \vec{A}\\ Cp + D \Delta
  \vec{A}  \end{pmatrix} = \begin{pmatrix} \refl{p} \\ \vec{A}\end{pmatrix}.
\]
Then, $\vec{A} = p (I-D\Delta)^{-1} C$ and consequently $p(A+B\Delta
(I-D\Delta)^{-1} C) = \refl{p}$.
\end{proof}

\section{Appendix B: Constructing Agler pairs} 
Theorem \ref{thmsos} and Proposition \ref{Estable} make it possible to
construct $\vec{E}_j, \vec{F}_j$ using the one variable matrix
Fej\'er-Riesz lemma (see \cites{Rosenblatt, Rosenblum}).  This
approach can actually be pushed further to prove the main formula in
Theorem \ref{thmsos} using the method of Kummert \cite{Kummert}, but
we will not do this here.  The construction goes as follows. For $z_2
\in \T$ we write
\[
\frac{\overline{p(w_1,z_2)} p(z) - \overline{\refl{p}(w_1,z_2)}
  \refl{p}(z)}{1-\bar{w}_1 z_1} = \Lambda_n(w_1)^* T_1(z_2)
\Lambda_n(z_1) \\
\]
for some matrix Laurent polynomial $T_1(z_2) \in \C^{n\times
  n}[z_2,z_2^{-1}]$.  In fact, if we write $p(z) = \sum_{j=0}^{n}
p_j(z_2)z_1^j$ and define $\refl{p}_j(z_2) = z_2^m
\overline{p_j(1/\bar{z}_2)}$ as well as
\[
R(z_2) = \begin{pmatrix} p_0(z_2) & p_1(z_2) & \cdots & p_{n-1}(z_2)
  \\
0 &p_0(z_2) & \cdots & p_{n-2}(z_2) \\
\vdots & \vdots & \ddots & \vdots \\
0 & 0 & \cdots & p_0(z_2) \end{pmatrix}
\]
\[
S(z_2) = \begin{pmatrix} \refl{p}_n(z_2) & \refl{p}_{n-1}(z_2) & \cdots & \refl{p}_{1}(z_2)
  \\
0 &\refl{p}_n(z_2) & \cdots & \refl{p}_{2}(z_2) \\
\vdots & \vdots & \ddots & \vdots \\
0 & 0 & \cdots & \refl{p}_n(z_2) \end{pmatrix}
\]
then by direct calculation $T_1(z_2) = R(z_2)^* R(z_2) - S(z_2)^*
S(z_2)$ for $z_2 \in \T$.  

By Theorem \ref{thmsos}, for $z_1,w_1 \in \C, z_2 \in \T$
\[
\Lambda_n(w_1)^* T_1(z_2)
\Lambda_n(z_1) = \Lambda_n(w_1)^* E_1(z_2)^*E_1(z_2) \Lambda_n(z_1).
\]
 As this formula holds
for all $z_1,w_1 \in \C$, we get 
\[
T_1(z_2) = E_1(z_2)^* E_1(z_2)
\]
for $z_2 \in \T$.  By the matrix Fej\'er-Riesz lemma, there exists a
matrix polynomial $A(z_2)$ with $\det A(z_2)$ non-vanishing for $z_2
\in \D$ such that $A(z_2)^* A(z_2) = E_1(z_2)^* E_1(z_2)$ for $z_2 \in \T$.
The functions $\Psi := A E_1^{-1}$ and $\Psi^{-1}$ are both matrix
rational inner functions on $\D$ which by the maximum principle can
only happen if $\Psi$ is a constant unitary matrix.  Thus, $E_1$ can
be constructed via the Fej\'er-Riesz lemma, which yields a
construction for $F_1$ via \eqref{matrixreflection}.  The construction
for $E_2,F_2$ is analogous.  

If $p$ has no zeros on $\cD^2$ this construction can be done
numerically since there are algorithms for performing Fej\'er-Riesz
factorizations for univariate matrix Laurent polynomials which are
positive on $\T$; see Theorem 3.1 of \cite{GeronimoLai} as well as the
references of \cite{GeronimoLai}.  In our case, $T_1$ is positive
definite on $\T$ except at finitely many points.  One could certainly
apply the algorithm to $T_1+\epsilon I$ for $\epsilon >0$ and let
$\epsilon \to 0$, but we suspect this would have numerical issues.  It
would be interesting, then, to produce a matrix Fej\'er-Riesz
factorization algorithm with prescribed singularities on $\T$.

\section{Appendix C: Multiplicities on $\T^2$ are even} 
This section is technically not necessary for the main theorems of
this paper, but it is perhaps reassuring to know that our formula for
the dimension of $\mcP_{j,k}$ in Theorem \ref{intthmdim} does not
actually involve any fractions.  Lemma \ref{Puilemma}, describing the
initial power series development of a Puiseux series associated to the
zero set of a semi-stable polynomial around a zero in the boundary,
may be of some independent interest.

\begin{prop} Let $p \in \C[z_1,z_2]$ be semi-stable.  If $p(t) = 0$
  for $t \in \T^2$, then $N_t(p,\refl{p})$ is even.
\end{prop}

To prove this, we switch to the product upper half plane
$\C_{+}^2=\{(z_1,z_2)\in \C^2: \Im z_1,\Im z_2 >0\}$ using a Cayley
transform.  Thus, we assume $p \in \C[z_1,z_2]$ has no zeros in
$\C_{+}^2$ and no common factors with $\bar{p}(z_1,z_2) \defn
\overline{p(\bar{z}_1,\bar{z}_2)}$.  We assume $p(0,0) = 0$ and we
will show $N_0(p,\bar{p})$ is even.  Using the local ring definition
of multiplicity and the Cayley transform, this yields our original
proposition, but we will not go through all of the details of this
conversion.

First, note that $p(0,z_2)$ is not identically zero, and by the
Weierstrass preparation theorem (see \cite{Fischer}) we can factor
\[
p = q p_1^{n_1} p_2^{n_2} \cdots p_m^{n_m}
\]
where $p_1,\dots, p_m$ are irreducible Weierstrass polynomials in
$z_2$ and $q$ is analytic and non-vanishing in a neighborhood of
$(0,0)$.  We can reflect this formula to obtain a factorization of
$\bar{p}$
\[
\bar{p} = \bar{q} \bar{p}_1^{n_1} \cdots \bar{p}_m^{n_m}.
\]
Section \ref{background} explains how to compute the intersection
multiplicity at $0$ of $p$ and $\bar{p}$.  First,
\[
N_{0}(p, \bar{p}) = \sum_{j,k} n_j n_k N_{0}(p_j,\bar{p}_k) =
2\sum_{j<k} n_j n_k N_{0} (p_j,\bar{p}_k) + \sum_{j} n_j^2
N_0(p_j,\bar{p}_j)
\]
and therefore it suffices to show $N_0(p_j,\bar{p}_j)$ is even for
each $j$.  We may as well drop the $j$ and prove the following lemma.

\begin{lemma} Let $p \in \C\langle z_1 \rangle [z_2]$ be an
  irreducible Weierstrass polynomial, analytic in a neighborhood $\Omega$
  of $(0,0) \in \C^2$.  Suppose $p$ is non-vanishing in $\Omega\cap
  \C_{+}^2$.  Then, $N_0(p,\bar{p})$ is even.
\end{lemma}

\begin{proof}
  By Puiseux's theorem (see \cite{Fischer}), there is a function
  $\phi(t) = a_r t^{r}+ a_{r+1} t^{r+1} + \dots$, analytic in a
  neighborhood of $0\in \C$, and a positive integer $k$ such that
\[
p(X,Y) = \prod_{j=1}^{k} (Y-\phi(\mu^j X^{1/k}) )
\]
where $\mu$ is a primitive $k$-th root of unity.  The expression
$X^{1/k}$ can be interpreted as a formal symbol whose $k$-th power is
$X$ and the above can be viewed as a formal power series computation.
Alternatively, the map $t \mapsto (t^k, \phi(t))$ gives a local
parametrization of the zero set of $p$.  By performing the reflection
operation,
\[
\bar{p}(X,Y) = \prod_{j=1}^{k} (Y-\bar{\phi}(\mu^j X^{1/k}) )
\]
and $t \mapsto (t^k,\bar{\phi}(t))$ gives a local parametrization of
the zero set of $\bar{p}$.

Now, $N_{0}(p,\bar{p})$ is given as the order of vanishing of the
resultant of $p$ and $\bar{p}$ which is given by
\[
\prod_{i=1}^{k}\prod_{j=1}^{k} (\phi(\mu^i X^{1/k}) - \bar{\phi}(\mu^j
X^{1/k}) ).
\]
We will show the order of vanishing is even by examining $\phi$
separately.  We will use the following lemma which is proved at the
end of the section.

\begin{lemma} \label{Puilemma} Let $\phi$ be analytic in a
  neighborhood of $0 \in \C$, $\phi(0)=0$ and assume $t\mapsto
  (t^k,\phi(t))$ is injective into $\C^2 \setminus \C_{+}^2$.  Then,
  $\phi$ has power series expansion
\[
\phi(t) = \sum_{j=1}^{M} a_{j} t^{jk} + b t^{2kL} +
\sum_{j=2kL+1}^{\infty} b_j t^j
\]
where $a_1<0$, $a_2,\dots, a_M \in \R$, $\arg b \in (\pi, 2\pi)$.
\end{lemma}

Thus, the initial terms of $\phi$ must be powers of $t^k$ taken with
real coefficients until an even power of $t^k$ is attained with
coefficient in the lower half plane, and after that point all we say
is that there must be terms that are not powers of $t^k$, else $(t^k,
\phi(t))$ would not be injective.

We proceed to look at the resultant computation
\[
\prod_{i=1}^{k}\prod_{j=1}^{k} (\phi(\mu^i X^{1/k}) - \bar{\phi}(\mu^j
X^{1/k}) ) = \prod_{i=1}^{k}\prod_{j=1}^{k} (b-\bar{b}) X^{2L} + \text{
  higher order} 
\]
which vanishes to order $2Lk^2$ since $b \ne \bar{b}$.  Therefore the
intersection multiplicity is even and we are finished aside from the
proof of Lemma \ref{Puilemma}
\end{proof}

\begin{proof}[Proof of Lemma \ref{Puilemma}]
To begin we analyze the first term of $\phi(t) = a t^r + \cdots$,
$a\ne 0$.  Our assumption is that $t\mapsto (t^k,\phi(t))$ does not
map into the upper half plane.  So, if $t = |t| e^{i\theta}$, then
$\sin k\theta >0$ implies $\Im \phi(t) \leq 0$.  So, writing $a = |a|
e^{i\alpha}$ and letting $\theta$ be fixed and satisfy $\sin k\theta
>0$ we have
\[
0 \geq \lim_{|t| \to 0} \frac{1}{|t|^r} \Im \phi(t) = |a| \Im
e^{i(\alpha +
  r\theta)} = |a| \sin (\alpha+ r\theta).
\]
The fact that $\sin(\alpha + r\theta)$
has constant sign on an interval of length $\pi/k$ means $r \leq k$.
On the other hand, if $\sin(\alpha + r\theta) >0$, then the above
limit calculation shows that $\Im \phi(t) >0$ for $|t|$ sufficiently
small in which case we must have $\sin(k\theta) \leq 0$.  Thus,
$\sin(k\theta)$ has constant sign on an interval of length $\pi/r$
which means $k\leq r$.  Therefore, $k=r$.  As a result, $\sin(\theta)
>0$ implies $\sin(\alpha + \theta) \leq 0$ which is only possible if
$\alpha = \pi$ (modulo multiples of $2\pi$).  Thus, $a$ is a negative
real number.

Next, we may suppose 
\[
\phi(t) = a_1 t^k +a_2 t^{2k}+\cdots a_M t^{Mk} + b t^L + \cdots
\]
where $a_1,\dots, a_M \in \R$ and $b t^L$ is the first term not of
this type (either $b$ is not real or $L$ is not a multiple of $k$).
Note that we allow $M=1$.  Choose $\theta$ so $\sin(k\theta) = 0$;
namely $\theta$ is an integer multiple of $\pi/k$.  Then, writing $b =
|b| e^{i\beta}$
\[
\lim_{|t| \to 0} \frac{1}{|t|^L} \Im \phi(t) = |b| \sin(\beta +
L\theta).
\]
This must be non-positive.  Otherwise, $\Im \phi(t)$ would be positive
for $|t|$ small enough and then we could perturb $\theta$ to get a
point where $(t^k, \phi(t))$ is in the upper half plane.  So, $\beta +
L \pi j/k \in [\pi, 2\pi]+2\pi \mathbb{Z}$ for $j=0,1,2,\dots$.  This
can only happen if $L$ is a multiple of $k$---by our assumption this
means $b$ is not real.  If $L$ is an odd multiple of $k$ then $\beta,
\beta +\pi \in [\pi,2\pi]$ which can only happen if $b$ is real which
is not true by assumption.  Thus, $L$ must be an even multiple of $k$
in which case $\beta \in (\pi,2\pi)$, again since $b$ is not real.
\end{proof}

\section*{Notation} 
We collect the notation of the paper in one place and refer to where
it was defined if possible.

\allowdisplaybreaks
\begin{align*}
\C &= \text{ the set of complex numbers }\\
\D &= \{z \in \C: |z| <1\} = \text{ the unit disk}\\
\D^2 & = \D \times \D = \text{ the bidisk}\\
\T &= \{z \in \C: |z|=1\} = \text{ the unit circle}\\
\T^2 &= \T\times \T = \text{ the two-torus}\\
\C^n &= \text{ $n$-dimensional column vectors with entries in $\C$}\\
\C^{1\times n} &= \text{ $n$-dimensional row vectors with entries in
  $\C$}\\
\C^{m \times n} &= \text{ $n\times m$ matrices with entries in $\C$}
\\
\C[z_1,z_2] &= \text{ polynomials in $z_1,z_2$ with coefficients in
  $\C$}\\
\deg p &= \text{ the bidegree of $p \in \C[z_1,z_2]$ }\\
V[z_1,z_2] &= \text{ polynomials in $z_1,z_2$ with coefficients in
  $V$} \\
L^2(\T^2) &= L^2 \text{ with respect to Lebesgue measure on $\T^2$}\\
\mathcal{I}_p &= \{q \in \C[z_1,z_2]: q/p \in L^2(\T^2)\} \\
\mcP_{j,k} &= \{q \in \mathcal{I}_p: \deg q \leq (j,k) \} \quad \text{
  Note $p$ is taken from context }\\
\refl{p}(z_1,z_2) &= z_1^nz_2^m\overline{p(1/\bar{z}_1,1/\bar{z}_2)} \text{
  for } p \in \C[z_1,z_2] \text { with } \deg p = (n,m)\\
\mathcal{I}_p^{\infty} &= \{ q \in \C[z_1,z_2]: q/p \in
L^{\infty}(\T^2)\} \\
N_{\T^2}(p,q) &= \text{ the number of common zeros of $p,q$ in
  $\T^2$,} \\
& \qquad \text{ counted with multiplicity as in Section
  \ref{background}}\\
\Lp &= L^2 \text{ space with respect to Lebesgue measure on $\T^2$
  times $\frac{1}{|p|^2}$} \\
RHP &= \{z\in \C: \Re z >0\} = \text{ right half plane}\\
\Hrow &= \text{ row-vector valued Hardy space on $\T$}\\
\mcE_j,\mcF_j,\mcG & \quad \text{ See Notation \ref{not:spaces} }\\
\vec{E}_j,\vec{F}_j,\vec{G} & \quad \text{ See Notation
  \ref{not:vecs}}\\
\oplus &= \text{ orthogonal direct sum in Hilbert space}\\
\Lambda_n &= \text{ See \eqref{Lam}}\\
E_j,F_j &= \text{ See \eqref{EFbreakup} }\\
X_n &= \text{ See \eqref{Xmat} }\\
 P &= \text{ projection onto $\mcG$ in Section \ref{seccomm}}\\
T_j &= \text{ the linear map on $\mcG$ given by $f \mapsto Pz_jf$}\\
\C_{\infty} & = \C\cup \{\infty\} = \text{ the Riemann sphere}\\
\D^{-1} &= \{z\in \C: |z|>1\} \cup \{\infty\} \subset \C_{\infty} \\
Z_Q &= \{z \in \C^2: Q(z) = 0\} \text{ for } Q \in \C[z_1,z_2]\\
N_{\lambda}(I) &= \text{ Intersection multiplicity of $\lambda$ in the
  ideal $I$; see Section \ref{background} }\\
\langle p,q\rangle &= \text{ the ideal generated by $p,q \in
  \C[z_1,z_2]$}\\
\ip{f}{g}_{\mcH} &= \text{ Inner product in Hilbert space $\mcH$}\\
N_{\lambda}(p,q) & = N_{\lambda}(\langle p ,q\rangle) \\ &= \text{
  intersection multiplicity of the common zero $\lambda$ of $p,q$} \\
f^{\#}(z_1,z_2) &= z_2^{m-1} f(z_1,1/z_2)\\
u &= (1,1,\dots, 1)\in \C^d \text{ in Section \ref{secnontan}}\\
\C_{+} &= \{z \in \C: \Im z >0\} = \text{ upper half plane}
\end{align*}


\begin{bibdiv}
\begin{biblist}

\bib{AM3point}{article}{
   author={Agler, Jim},
   author={McCarthy, John E.},
   title={The three point Pick problem on the bidisk},
   journal={New York J. Math.},
   volume={6},
   date={2000},
   pages={227--236 (electronic)},
   issn={1076-9803},
   review={\MR{1781508 (2001j:32007)}},
}

\bib{Pickbook}{book}{
   author={Agler, Jim},
   author={McCarthy, John E.},
   title={Pick interpolation and Hilbert function spaces},
   series={Graduate Studies in Mathematics},
   volume={44},
   publisher={American Mathematical Society, Providence, RI},
   date={2002},
   pages={xx+308},
   isbn={0-8218-2898-3},
   review={\MR{1882259 (2003b:47001)}},
}


\bib{AMhankel}{article}{
   author={Agler, Jim},
   author={McCarthy, John E.},
   title={Hankel vector moment sequences and the non-tangential regularity
   at infinity of two variable Pick functions},
   journal={Trans. Amer. Math. Soc.},
   volume={366},
   date={2014},
   number={3},
   pages={1379--1411},
   issn={0002-9947},
   review={\MR{3145735}},
   doi={10.1090/S0002-9947-2013-05952-1},
}

\bib{AMYcara}{article}{
   author={Agler, Jim},
   author={McCarthy, John E.},
   author={Young, N. J.},
   title={A Carath\'eodory theorem for the bidisk via Hilbert space methods},
   journal={Math. Ann.},
   volume={352},
   date={2012},
   number={3},
   pages={581--624},
   issn={0025-5831},
   review={\MR{2885589}},
   doi={10.1007/s00208-011-0650-7},
}


\bib{AMY}{article}{
   author={Agler, Jim},
   author={McCarthy, John E.},
   author={Young, N. J.},
   title={Operator monotone functions and L\"owner functions of several
   variables},
   journal={Ann. of Math. (2)},
   volume={176},
   date={2012},
   number={3},
   pages={1783--1826},
   issn={0003-486X},
   review={\MR{2979860}},
   doi={10.4007/annals.2012.176.3.7},
}

\bib{ATY}{article}{
   author={Agler, J.},
   author={Tully-Doyle, R.},
   author={Young, N. J.},
   title={Boundary behavior of analytic functions of two variables via
   generalized models},
   journal={Indag. Math. (N.S.)},
   volume={23},
   date={2012},
   number={4},
   pages={995--1027},
   issn={0019-3577},
   review={\MR{2991931}},
   doi={10.1016/j.indag.2012.07.003},
}



\bib{BSV}{article}{
   author={Ball, Joseph A.},
   author={Sadosky, Cora},
   author={Vinnikov, Victor},
   title={Scattering systems with several evolutions and multidimensional
   input/state/output systems},
   journal={Integral Equations Operator Theory},
   volume={52},
   date={2005},
   number={3},
   pages={323--393},
   issn={0378-620X},
   review={\MR{2184571 (2006h:47013)}},
   doi={10.1007/s00020-005-1351-y},
}

\bib{dirichlet}{misc}{
author={B\'en\'eteau, Catherine},
author={Knese, Greg},
author={Kosi\'nski, {\L}ukasz},
author={Liaw, Constanze},
author={Seco, Daniel},
author={Sola, Alan},
title={Cyclic polynomials in two variables},
status={submitted},
date={2014},
}

\bib{Bickel}{article}{
   author={Bickel, Kelly},
   title={Fundamental Agler decompositions},
   journal={Integral Equations Operator Theory},
   volume={74},
   date={2012},
   number={2},
   pages={233--257},
   issn={0378-620X},
   review={\MR{2983064}},
   doi={10.1007/s00020-012-1992-6},
}


\bib{BickelKnese}{article}{
   author={Bickel, Kelly},
   author={Knese, Greg},
   title={Inner functions on the bidisk and associated Hilbert spaces},
   journal={J. Funct. Anal.},
   volume={265},
   date={2013},
   number={11},
   pages={2753--2790},
   issn={0022-1236},
   review={\MR{3096989}},
   doi={10.1016/j.jfa.2013.08.002},
}

\bib{CW}{article}{
   author={Cole, Brian J.},
   author={Wermer, John},
   title={Ando's theorem and sums of squares},
   journal={Indiana Univ. Math. J.},
   volume={48},
   date={1999},
   number={3},
   pages={767--791},
   issn={0022-2518},
   review={\MR{1736979 (2000m:47014)}},
   doi={10.1512/iumj.1999.48.1716},
}

\bib{CLO}{book}{
   author={Cox, David A.},
   author={Little, John},
   author={O'Shea, Donal},
   title={Using algebraic geometry},
   series={Graduate Texts in Mathematics},
   volume={185},
   edition={2},
   publisher={Springer, New York},
   date={2005},
   pages={xii+572},
   isbn={0-387-20706-6},
   review={\MR{2122859 (2005i:13037)}},
}

\bib{Dangelo}{book}{
   author={D'Angelo, John P.},
   title={Several complex variables and the geometry of real hypersurfaces},
   series={Studies in Advanced Mathematics},
   publisher={CRC Press, Boca Raton, FL},
   date={1993},
   pages={xiv+272},
   isbn={0-8493-8272-6},
   review={\MR{1224231 (94i:32022)}},
}

\bib{dangeloineq}{book}{
   author={D'Angelo, John P.},
   title={Inequalities from complex analysis},
   series={Carus Mathematical Monographs},
   volume={28},
   publisher={Mathematical Association of America, Washington, DC},
   date={2002},
   pages={xvi+264},
   isbn={0-88385-033-8},
   review={\MR{1899123 (2003e:32001)}},
   doi={10.5948/UPO9780883859704},
}


\bib{Fischer}{book}{
   author={Fischer, Gerd},
   title={Plane algebraic curves},
   series={Student Mathematical Library},
   volume={15},
   note={Translated from the 1994 German original by Leslie Kay},
   publisher={American Mathematical Society, Providence, RI},
   date={2001},
   pages={xvi+229},
   isbn={0-8218-2122-9},
   review={\MR{1836037 (2002g:14042)}},
}
		
\bib{Fulton}{book}{
   author={Fulton, William},
   title={Algebraic curves},
   series={Advanced Book Classics},
   note={An introduction to algebraic geometry;
   Notes written with the collaboration of Richard Weiss;
   Reprint of 1969 original},
   publisher={Addison-Wesley Publishing Company, Advanced Book Program,
   Redwood City, CA},
   date={1989},
   pages={xxii+226},
   isbn={0-201-51010-3},
   review={\MR{1042981 (90k:14023)}},
}

\bib{GeronimoLai}{article}{
   author={Geronimo, Jeffrey S.},
   author={Lai, Ming-Jun},
   title={Factorization of multivariate positive Laurent polynomials},
   journal={J. Approx. Theory},
   volume={139},
   date={2006},
   number={1-2},
   pages={327--345},
   issn={0021-9045},
   review={\MR{2220044 (2007a:47023)}},
   doi={10.1016/j.jat.2005.09.010},
}


\bib{GW}{article}{
   author={Geronimo, Jeffrey S.},
   author={Woerdeman, Hugo J.},
   title={Positive extensions, Fej\'er-Riesz factorization and
   autoregressive filters in two variables},
   journal={Ann. of Math. (2)},
   volume={160},
   date={2004},
   number={3},
   pages={839--906},
   issn={0003-486X},
   review={\MR{2144970 (2006b:42036)}},
   doi={10.4007/annals.2004.160.839},
}

\bib{GWzeros}{article}{
   author={Geronimo, Jeffrey S.},
   author={Woerdeman, Hugo J.},
   title={Two-variable polynomials: intersecting zeros and stability},
   journal={IEEE Trans. Circuits Syst. I Regul. Pap.},
   volume={53},
   date={2006},
   number={5},
   pages={1130--1139},
   issn={1057-7122},
   review={\MR{2235187 (2007c:93050)}},
   doi={10.1109/TCSI.2005.862180},
}



 \bib{goodman}{article}{
   author={Goodman, Dennis},
   title={Some stability properties of two-dimensional linear
   shift-invariant digital filters},
   journal={IEEE Trans. Circuits and Systems},
   volume={CAS-24},
   date={1977},
   number={4},
   pages={201--208},
   issn={0098-4094},
   review={\MR{0444266 (56 \#2624)}},
}


\bib{KneseSchwarz}{article}{
   author={Knese, Greg},
   title={A Schwarz lemma on the polydisk},
   journal={Proc. Amer. Math. Soc.},
   volume={135},
   date={2007},
   number={9},
   pages={2759--2768 (electronic)},
   issn={0002-9939},
   review={\MR{2317950 (2008j:32004)}},
   doi={10.1090/S0002-9939-07-08766-7},
}

\bib{KneseDV}{article}{
   author={Knese, Greg},
   title={Polynomials defining distinguished varieties},
   journal={Trans. Amer. Math. Soc.},
   volume={362},
   date={2010},
   number={11},
   pages={5635--5655},
   issn={0002-9947},
   review={\MR{2661491 (2011f:47022)}},
   doi={10.1090/S0002-9947-2010-05275-4},
}

\bib{KneseAPDE}{article}{
   author={Knese, Greg},
   title={Polynomials with no zeros on the bidisk},
   journal={Anal. PDE},
   volume={3},
   date={2010},
   number={2},
   pages={109--149},
   issn={1948-206X},
   review={\MR{2657451 (2011i:42051)}},
   doi={10.2140/apde.2010.3.109},
}

\bib{KneseRIF}{article}{
   author={Knese, Greg},
   title={Rational inner functions in the Schur-Agler class of the polydisk},
   journal={Publ. Mat.},
   volume={55},
   date={2011},
   number={2},
   pages={343--357},
   issn={0214-1493},
   review={\MR{2839446 (2012k:47033)}},
   doi={10.5565/PUBLMAT\_55211\_04},
}

\bib{KneseSemi}{misc}{
author={Knese, Greg},
title={Determinantal representations for semi-hyperbolic polynomials},
date={2013},
status={submitted},
}

\bib{KM}{article}{
   author={Knorn, Steffi},
   author={Middleton, Richard H.},
   title={Stability of two-dimensional linear systems with singularities on
   the stability boundary using LMIs},
   journal={IEEE Trans. Automat. Control},
   volume={58},
   date={2013},
   number={10},
   pages={2579--2590},
   issn={0018-9286},
   review={\MR{3106063}},
   doi={10.1109/TAC.2013.2264852},
}

\bib{Kummert}{article}{
   author={Kummert, Anton},
   title={Synthesis of two-dimensional lossless $m$-ports with prescribed
   scattering matrix},
   journal={Circuits Systems Signal Process.},
   volume={8},
   date={1989},
   number={1},
   pages={97--119},
   issn={0278-081X},
   review={\MR{998029 (90e:94048)}},
   doi={10.1007/BF01598747},
}

\bib{Lind1}{article}{
   author={Lind, Douglas},
   author={Schmidt, Klaus},
   title={Homoclinic points of algebraic ${\bf Z}^d$-actions},
   journal={J. Amer. Math. Soc.},
   volume={12},
   date={1999},
   number={4},
   pages={953--980},
   issn={0894-0347},
   review={\MR{1678035 (2000d:37002)}},
   doi={10.1090/S0894-0347-99-00306-9},
}

\bib{Lind2}{article}{
   author={Lind, Douglas},
   author={Schmidt, Klaus},
   author={Verbitskiy, Evgeny},
   title={Homoclinic points, atoral polynomials, and periodic points of
   algebraic $\mathbb{Z}^d$-actions},
   journal={Ergodic Theory Dynam. Systems},
   volume={33},
   date={2013},
   number={4},
   pages={1060--1081},
   issn={0143-3857},
   review={\MR{3082539}},
   doi={10.1017/S014338571200017X},
}

\bib{Pascoe}{misc}{
author={Pascoe, J.E.},
title={An inductive Julia-Carath\'eodory theorem for Pick functions in
  two variables},
status={preprint},
}



\bib{Rosenblatt}{article}{
   author={Rosenblatt, Murray},
   title={A multi-dimensional prediction problem},
   journal={Ark. Mat.},
   volume={3},
   date={1958},
   pages={407--424},
   issn={0004-2080},
   review={\MR{0092332 (19,1098c)}},
}

\bib{Rosenblum}{article}{
   author={Rosenblum, Marvin},
   title={Vectorial Toeplitz operators and the Fej\'er-Riesz theorem},
   journal={J. Math. Anal. Appl.},
   volume={23},
   date={1968},
   pages={139--147},
   issn={0022-247x},
   review={\MR{0227794 (37 \#3378)}},
}

\bib{rudin}{book}{
   author={Rudin, Walter},
   title={Function theory in polydiscs},
   publisher={W. A. Benjamin, Inc., New York-Amsterdam},
   date={1969},
   pages={vii+188},
   review={\MR{0255841 (41 \#501)}},
}

\bib{Scheinker}{article}{
   author={Scheinker, David},
   title={Hilbert function spaces and the Nevanlinna-Pick problem on the
   polydisc II},
   journal={J. Funct. Anal.},
   volume={266},
   date={2014},
   number={1},
   pages={355--367},
   issn={0022-1236},
   review={\MR{3121734}},
   doi={10.1016/j.jfa.2013.07.027},
}

\bib{shaf}{book}{
   author={Shafarevich, Igor R.},
   title={Basic algebraic geometry. 1},
   edition={3},
   edition={Translated from the 2007 third Russian edition},
   note={Varieties in projective space},
   publisher={Springer, Heidelberg},
   date={2013},
   pages={xviii+310},
   isbn={978-3-642-37955-0},
   isbn={978-3-642-37956-7},
   review={\MR{3100243}},
}

\bib{simon}{book}{
   author={Simon, Barry},
   title={Orthogonal polynomials on the unit circle. Part 1},
   series={American Mathematical Society Colloquium Publications},
   volume={54},
   note={Classical theory},
   publisher={American Mathematical Society, Providence, RI},
   date={2005},
   pages={xxvi+466},
   isbn={0-8218-3446-0},
   review={\MR{2105088 (2006a:42002a)}},
}


\end{biblist}
\end{bibdiv}
\end{document}